\documentclass[numbers=enddot,12pt,final,onecolumn,notitlepage]{scrartcl}%
\usepackage[headsepline,footsepline,manualmark]{scrlayer-scrpage}
\usepackage{tikz}
\usetikzlibrary{matrix,calc}
\usetikzlibrary{cd}
\usepackage[all,cmtip]{xy}
\usepackage{amssymb}
\usepackage{amsmath}
\usepackage{amsthm}
\usepackage{framed}
\usepackage{comment}
\usepackage{color}
\usepackage[breaklinks=true]{hyperref}
\usepackage[sc]{mathpazo}
\usepackage[T1]{fontenc}
\usepackage{tikz}
\usepackage{needspace}
\usepackage{tabls}
\usepackage{authblk}
\usepackage[outerbars]{changebar} % permet de positionner les
\providecommand{\U}[1]{\protect\rule{.1in}{.1in}}
\usetikzlibrary{arrows}
\newtheoremstyle{plainsl}% <name>
  {8pt plus 2pt minus 4pt}% <Space above>
  {8pt plus 2pt minus 4pt}% <Space below>
  {\slshape}% <Body font>
  {0pt}% <Indent amount>
  {\bfseries}% <Theorem head font>
  {.}% <Punctuation after theorem head>
  {5pt plus 1pt minus 1pt}% <Space after theorem headi>
  {}% <Theorem head spec (can be left empty, meaning `normal')>

\theoremstyle{plainsl}
  \newtheorem{theorem}{Theorem}[section]
  \newtheorem{proposition}[theorem]{Proposition}
  \newtheorem{lemma}[theorem]{Lemma}
  \newtheorem{corollary}[theorem]{Corollary}
  
\theoremstyle{definition}
  \newtheorem{definition}[theorem]{Definition}
  \newtheorem{example}[theorem]{Example}

 \theoremstyle{remark}
  \newtheorem{remark}[theorem]{Remark}

\newenvironment{statement}{\begin{quote}}{\end{quote}}
\usepackage{ulem}
\let\sumnonlimits\sum
\let\prodnonlimits\prod
\let\cupnonlimits\bigcup
\let\capnonlimits\bigcap
\renewcommand{\sum}{\sumnonlimits\limits}
\renewcommand{\prod}{\prodnonlimits\limits}
\renewcommand{\bigcup}{\cupnonlimits\limits}
\renewcommand{\bigcap}{\capnonlimits\limits}
\newcommand{\kk}{\mathbf{k}}
\newcommand{\FF}{\mathbf{F}}
\newcommand{\NN}{\mathbb{N}}
\newcommand{\QQ}{\mathbb{Q}}
\newcommand{\id}{\operatorname{id}}
\newcommand{\Hom}{\operatorname{Hom}}
\newcommand{\End}{\operatorname{End}}
\newcommand{\conc}{\operatorname{conc}}
\newcommand{\Prim}{\operatorname{Prim}}
\newcommand{\oast}{\circledast}
\newcommand{\full}{\operatorname{full}}

\newcommand{\set}[1]{\left\{ #1 \right\}}
\newcommand{\tup}[1]{\left( #1 \right)}
\newcommand\arxiv[1]{\href{http://www.arxiv.org/abs/#1}{\texttt{arXiv:#1}}}

\excludecomment{verlong}
\includecomment{vershort}
\excludecomment{noncompile}
\ihead{Variations on linear independence of grouplikes}
\ohead{page \thepage}
\makeatletter
\newcommand{\bigoast}[1]{%
  \vcenter{\hbox{\scalebox{1}{$\oast$}}}%
}
\makeatother
\def\scal#1#2{\langle #1\mid #2 \rangle}
\newcommand{\kx}{\kk\left[x\right]}
\newcommand\rsetminus{\mathbin{\mathpalette\rsetminusaux\relax}}
\newcommand\rsetminusaux[2]{\mspace{-4mu}
  \raisebox{\rsmraise{#1}\depth}{\rotatebox[origin=c]{-20}{$#1\smallsetminus$}}
 \mspace{-4mu}
}
\newcommand\rsmraise[1]{%
  \ifx#1\displaystyle .8\else
    \ifx#1\textstyle .8\else
      \ifx#1\scriptstyle .6\else
        .45%
      \fi
    \fi
  \fi}
\usepackage{amsmath}
\usepackage{amsfonts}
\usepackage{amssymb}
\newcommand{\calA}{{\mathcal A}}

\newcommand{\calU}{{\mathcal U}}

\newcommand{\N}{{\mathbb N}}

\newcommand{\Z}{{\mathbb Z}}
\newcommand{\Q}{{\mathbb Q}}

 \def\shuffle{\mathop{_{^{\sqcup\!\sqcup}}}}

\def\deg{\mathrm{deg}}

\newcommand{\poly}[2]{#1 \langle #2 \rangle}

\def\QY{\poly{\Q}{Y}}

\def\QY_0{\Q\left\langle{Y_0}\right\rangle}

 \newcommand{\calB}{{\cal B}}

\def\path{\rightsquigarrow}

\def\bv{\mid}

\def\deg{\mathop\mathrm{deg}\nolimits}

\def\dbinom#1#2{{#1\choose#2}}

\def\shuffle{\mathop{_{^{\sqcup\!\sqcup}}}} 
\begingroup
\def\2#1{\ifnum#1<10 0\fi\the#1}
\xdef\isodayandtime{
{\2\day-\2\month-\the\year\space\2{\count0}:%
\2{\count2}}}
\endgroup
\def\scal#1#2{\langle #1\bv#2 \rangle}
\def\ncp#1#2{#1\langle #2\rangle}
\def\ncs#1#2{#1\langle \!\langle #2\rangle \!\rangle}
\def\vt{\vartheta}
\newcommand{\arinj}{\ar@{_{(}->}} % Monomorphism.
\newcommand{\arinjrev}{\ar@{^{(}->}} % Monomorphism.
\newcommand{\arsurj}{\ar@{->>}} % Epimorphism.
\newcommand{\arelem}{\ar@{|->}} % Arrow to be used in describing a map on elements.
\newcommand{\arback}{\ar@{<-}} % Backward arrow.

\begin{document}
\title{Three variations on the linear independence of grouplikes in a coalgebra}
\author[1]{G\'erard H. E. Duchamp\thanks{gheduchamp@gmail.com}} 
\author[2]{Darij Grinberg\thanks{darijgrinberg@gmail.com}}
\author[1,3]{\\Vincel Hoang Ngoc Minh\thanks{minh@lipn.univ-paris13.fr}}
\affil[1]{\small{LIPN, Northern Paris University, Sorbonne Paris City, 93430 Villetaneuse, France}}
\affil[2]{\small{Drexel University, Korman Center, Room 291, 15 S 33rd Street,\break 
Philadelphia PA, 19104, USA}}
\affil[3]{\small{University of Lille, 1 Place D\'eliot, 59024 Lille, France}}
\date{July 3, 2026}
\maketitle
\begin{abstract}
\textbf{Abstract.}
The grouplike elements of a coalgebra over a field are
known to be linearly independent over said field.
Here we prove three variants of this result.
One is a generalization to coalgebras over a commutative ring
(in which case the linear independence has to be replaced by
a weaker statement).
Another is a stronger statement that holds (under stronger
assumptions) in a commutative bialgebra.
The last variant is a linear independence result for
characters (as opposed to grouplike elements) of a
bialgebra.
\end{abstract}

\newpage
\tableofcontents
\section{Introduction}

% A classical result in algebra, commonly ascribed to Dedekind,
% says that any set of distinct field homomorphisms from a field $L$
% to a field $K$ is $K$-linearly independent.
% (See, e.g., \cite[Section V.6.1, Corollary 2]{Bourbaki-Alg2} for a
% version of this result.)
% The purpose of this paper is to generalize Dedekind's classical result
% in three different ways.
% [DG] I have commented this out, since linear independence of grouplikes
% does not directly generalize linear independence of characters in
% Galois theory.

A classical result in the theory of coalgebras
(\cite[Proposition 3.2.1 b)]{Sweedl69},
\cite[Lemma 2.1.12]{Radford-Hopf},
\cite[Theorem 2.1.2 (i)]{Abe}) says that the grouplike elements
of a coalgebra over a field are linearly independent over said
field.
We shall prove three variants of this result.
The first variant (Theorem~\ref{thm.old1},
in Section~\ref{sect.dg1}) generalizes it to coalgebras over an
arbitrary commutative ring (at the expense of obtaining a
subtler claim than literal linear independence).
The second variant (Theorem~\ref{thm.1},
in Section~\ref{sect.dg2}) gives a stronger independence
claim under a stronger assumption (viz., that the coalgebra
is a commutative bialgebra, and that the grouplike elements
and their pairwise differences are regular).
The third variant (Theorem~\ref{thm.gen3} \textbf{(b)},
in Section~\ref{sect.gen3}) is a linear independence statement
in the dual algebra of a bialgebra;
namely, it claims that (again under certain conditions) a set
of characters of a bialgebra (i.e., algebra homomorphisms from
the bialgebra to the base ring) are linearly independent not
just over the base ring, but over a certain subalgebra of the
dual.
We discuss the connection between grouplike elements and
characters.

%The paper is structured as follows

\subsection*{Acknowledgments}

We thank Jeremy Rickard for Example~\ref{exa.dual-tensor.not-tf} and Christophe Reutenauer for interesting (electronic) discussions about the Arrow~\ref{DualInjArr} and its link with rational closures.

DG thanks the Mathematisches Forschungsinstitut Oberwolfach for its hospitality at the time this paper was finished. This research was supported through the programme ``Oberwolfach Leibniz Fellows'' by the Mathematisches Forschungsinstitut Oberwolfach in 2020.

\section{Background}

\subsection{Notations and generalities}

\subsubsection{General conventions}

We shall study coalgebras and bialgebras. We refer to
the literature on Hopf algebras and bialgebras -- e.g., \cite[Chapter 1]{GriRei}
or \cite[Section III.11]{Bourbaki-Alg1} -- for these concepts.%
\footnote{We note that terminology is not entirely
standardized across the literature. What we call a
``coalgebra'', for example, is called a
``counital coassociative cogebra'' in
\cite[Section III.11]{Bourbaki-Alg1}.
What we call a ``bialgebra'' is called a ``bigebra''
in \cite[Section III.11]{Bourbaki-Alg1}.}

We let $\NN$ denote the set $\left\{0,1,2,\ldots\right\}$.

Rings are always associative and have unity (but are not
always commutative).

We fix a commutative ring $\kk$. All algebras, linear maps and tensor
signs that appear in the following are over $\kk$ unless specified otherwise.
The symbol $\Hom$ shall always stand for $\kk$-module homomorphisms.

% care this afterwardss

\subsubsection{Algebras, coalgebras, bialgebras}

Recall that

\begin{itemize}

\item a \emph{$\kk$-algebra} can be defined as
      a $\kk$-module $A$ equipped
      with a $\kk$-linear map $\mu : A \otimes A \to A$ (called
      \emph{multiplication}) and a $\kk$-linear map
      $\eta : \kk \to A$ (called \emph{unit map})
      satisfying the associativity axiom
      (which says that
      $\mu \circ \left(\mu \otimes \id\right)
       = \mu \circ \left(\id \otimes \mu\right)$)
      and the unitality axiom
      (which says that
      $\mu \circ \left(\eta \otimes \id\right)$ and
      $\mu \circ \left(\id \otimes \eta\right)$ are the
      canonical isomorphisms from $\kk \otimes A$ and
      $A \otimes \kk$ to $A$).

\item a \emph{$\kk$-coalgebra} is defined as
      a $\kk$-module $C$ equipped
      with a $\kk$-linear map $\Delta : C \to C \otimes C$ (called
      \emph{comultiplication}) and a $\kk$-linear map
      $\epsilon : C \to \kk$ (called \emph{counit map})
      satisfying the coassociativity axiom
      (which says that
      $\left(\Delta \otimes \id\right) \circ \Delta
       = \left(\id \otimes \Delta\right) \circ \Delta$)
      and the counitality axiom
      (which says that
      $\left(\epsilon \otimes \id\right) \circ \Delta$ and
      $\left(\id \otimes \epsilon\right) \circ \Delta$ are the
      canonical isomorphisms from $C$ to $\kk \otimes C$ and
      $C \otimes \kk$).

\item a \emph{$\kk$-bialgebra} means a $\kk$-module $B$
      that is simultaneously a $\kk$-algebra and a $\kk$-coalgebra,
      with the property that the comultiplication $\Delta$ and
      the counit $\epsilon$ are $\kk$-algebra homomorphisms
      (where the $\kk$-algebra structure on $B \otimes B$ is the
      standard one, induced by the one on $B$).
      
\end{itemize}

(Note that Lie algebras are not considered to be $\kk$-algebras.)

The multiplication and the unit map of a $\kk$-algebra $A$
will always be denoted by $\mu_A$ and $\eta_A$.
Likewise, the comultiplication and the counit of a $\kk$-coalgebra
$C$ will always be denoted by $\Delta_C$ and $\epsilon_C$.
We will occasionally omit the subscripts when it is clear what they
should be (e.g., we will write $\Delta$ instead of $\Delta_C$
when it is clear that the only coalgebra we could possibly
be referring to is $C$).

\subsubsection{Convolution and the dual algebra of a coalgebra}

If $C$ is a $\kk$-coalgebra, and if $A$ is a $\kk$-algebra, then
the $\kk$-module $\Hom\left(C,A\right)$ itself becomes a
$\kk$-algebra using a multiplication operation known as
\emph{convolution}.
We denote it by $\oast$, and recall how it is defined:
For any two $\kk$-linear maps $f, g \in \Hom\left(C,A\right)$,
we have
\[
f \oast g = \mu_A \circ \left(f \otimes g\right) \circ \Delta_C
: C \to A.
\]
The map
$\eta_A \circ \epsilon_C : C \to A$ is a neutral element
for this operation $\oast$.

(Note that the operation $\oast$ is denoted by $\star$ in
\cite[Definition 1.4.1]{GriRei}.)

If $f$ is a $\kk$-linear map from a coalgebra $C$ to an
algebra $A$, and if $n\in \NN$, then $f^{\oast n}$
denotes the $n$-th power of $f$ with respect to convolution
(i.e., the $n$-th power of $f$ in the algebra
$\left(\Hom\left(C,A\right), \oast, \eta_A \circ \epsilon_C\right)$).

If $M$ is any $\kk$-module, then the dual $\kk$-module
$\Hom_{\kk} \left( M, \kk \right)$ shall be denoted by $M^\vee$.
Thus, if $C$ is a $\kk$-coalgebra, then its dual
$\kk$-module $C^\vee = \Hom\left(C, \kk\right)$ becomes
a $\kk$-algebra via the convolution product $\oast$.
The unity of this $\kk$-algebra $C^\vee$ is exactly
the counit $\epsilon$ of $C$.

\subsubsection{Grouplike elements}

One of the simplest classes of elements in a coalgebra
are the \emph{grouplike elements}:

\begin{definition}
An element $g$ of a $\kk$-coalgebra $C$ is said to be
\emph{grouplike} if it satisfies $\Delta\left(  g\right)  =g\otimes g$ and
$\epsilon\left(  g\right)  =1$.
\end{definition}

We reject the alternative definition of
``grouplike'' (preferred by some authors) that replaces the
``$\epsilon\left(  g\right)  =1$'' condition by the
weaker requirement ``$g\neq 0$''; this definition is
equivalent to ours when $\kk$ is
a field, but ill-behaved and useless when $\kk$ is
merely a commutative ring.

The following examples illustrate the notion of
grouplike elements in different $\kk$-bialgebras.
% Some of these will also be used later to show the
% limits of two of our linear independence results.

\begin{example}
\label{exa.grouplike.pol}
Consider the polynomial ring
$\kx$ in one variable $x$ over $\kk$.
Define two $\kk$-algebra homomorphisms
$\Delta : \kx \to \kx \otimes \kx$
and $\epsilon : \kx \to \kk$ by setting
\[
\Delta \left( x \right)
= x \otimes 1 + 1 \otimes x
\qquad \text{ and } \qquad
\epsilon \left( x \right) = 0 .
\]
Then, it is easy to check that the $\kk$-algebra
$\kx$, equipped with the comultiplication $\Delta$
and the counit $\epsilon$, is a $\kk$-bialgebra.
This is the ``standard'' $\kk$-bialgebra structure
on the polynomial ring $\kx$.
% A polynomial $p \in \kx$ is a grouplike element
% of this $\kk$-bialgebra if and only if it
% satisfies $p\tup{x+y} = p\tup{x} + p\tup{y}$
% and $p\tup{0} = 1$.
If $\kk$ is a reduced ring (i.e., a commutative
ring with no nonzero nilpotent elements), then
it is not hard to show that $1$ is the only
grouplike element of this $\kk$-bialgebra.
On the other hand, if $u \in \kk$ satisfies
$u^2 = 0$, then the element $1 + ux \in \kx$
is also grouplike.
\end{example}

\begin{example}
\label{exa.grouplike.q-pol}
Let $q \in \kk$. Consider the polynomial ring
$\kx$ in one variable $x$ over $\kk$.
Define two $\kk$-algebra homomorphisms
$\Delta_{\uparrow_q} : \kx \to \kx \otimes \kx$
and $\epsilon : \kx \to \kk$ by setting
\[
\Delta_{\uparrow_q} \left( x \right)
= x \otimes 1 + 1 \otimes x + q x \otimes x
\qquad \text{ and } \qquad
\epsilon \left( x \right) = 0 .
\]
Then, it is easy to check that the $\kk$-algebra
$\kx$, equipped with the comultiplication $\Delta_{\uparrow_q}$
and the counit $\epsilon$, is a $\kk$-bialgebra.
This $\kk$-bialgebra $\tup{\kx, \Delta_{\uparrow_q}, \epsilon}$
is known as the \emph{univariate $q$-infiltration bialgebra}\footnote{See \cite{CFL,Duchamp-Luque,Duchamp-MO} for the multivariate case.}.
Note the following special cases:

\begin{enumerate}

\item If $q = 0$, then the
univariate $q$-infiltration bialgebra
$\tup{\kx, \Delta_{\uparrow_q}, \epsilon}$ is the
``standard'' $\kk$-bialgebra $\tup{\kx, \Delta, \epsilon}$
from Example~\ref{exa.grouplike.pol}.
(Indeed, $\Delta_{\uparrow_q} = \Delta$ when $q = 0$.)

\item When $q$ is nilpotent, the
univariate $q$-infiltration bialgebra
$\tup{\kx, \Delta_{\uparrow_q}, \epsilon}$ is a
\emph{Hopf algebra} (see, e.g., \cite[Definition 1.4.6]{GriRei}
or \cite[Definition 7.1.1]{Radford-Hopf} for the
definition of this concept). In this case, the
antipode of this Hopf algebra sends $x$ to $-x / \left(1+qx\right)$.

If $\kk$ is a $\Q$-algebra, then we can say even more:
Assume that $\kk$ is a $\Q$-algebra and $q$ is
nilpotent.
Let $y$ be the element $\sum_{k=1}^\infty \dfrac{1}{k!} q^{k-1} x^k$
of $\kx$.
(This sum is actually finite, since $q$ is nilpotent.
It can be viewed as the result of evaluating the
formal power series $\dfrac{\exp\tup{tx} - 1}{t} \in \kx[[t]]$
at $t = q$.)
Consider the unique $\kk$-algebra homomorphism
$\Phi : \kx \to \kx$ that sends $x$ to $y$.
Then, $\Phi$ is an isomorphism from
the $\kk$-bialgebra $\tup{\kx, \Delta_{\uparrow_q}, \epsilon}$
to the $\kk$-bialgebra $\tup{\kx, \Delta, \epsilon}$
(defined in Example~\ref{exa.grouplike.pol}).
% We will later obtain stronger (and more general) results
% about it.
% In particular,
% if $\kk$ is a $\Q$-algebra and $q$ is nilpotent, then
% this Hopf algebra is the universal enveloping algebra
% of a Lie algebra (see Theorem~\ref{infiltr_main} (d)),
% and thus is the symmetric algebra of a $\kk$-module
% (since it is commutative).
% [DG] This is left to the next paper now.

No such isomorphism exists in general when $\kk$
is not a $\Q$-algebra. For example, if $\kk$ is an
$\mathbb{F}_2$-algebra, and if $q \in \kk$ satisfies
$q \neq 0$ and $q^2 = 0$, then
the Hopf algebra $\tup{\kx, \Delta, \epsilon}$ has
the property that its antipode is the identity, but
the Hopf algebra $\tup{\kx, \Delta_{\uparrow_q}, \epsilon}$
does not have this property.
Since any $\kk$-bialgebra isomorphism between
Hopf algebras must respect their antipodes
(see, e.g., \cite[Corollary 1.4.27]{GriRei}), this
precludes any $\kk$-bialgebra isomorphism.
\end{enumerate}

Let us return to the general case (with $\kk$
and $q$ arbitrary).
From the above definitions of $\Delta_{\uparrow_q}$
and $\epsilon$, it is easy to see that
\[
\Delta_{\uparrow_q} \left( 1 + qx\right)
= \underbrace{\Delta_{\uparrow_q}\left(1\right)}_{= 1 \otimes 1}
+ q \underbrace{\Delta_{\uparrow_q}\left(x\right)}_{= x \otimes 1 + 1 \otimes x + q x \otimes x}
= \left( 1 + qx\right) \otimes \left( 1 + qx\right)
\]
and
\[
\epsilon \left( 1 + qx\right)
= \underbrace{\epsilon\left(1\right)}_{= 1}
+ q \underbrace{\epsilon\left(x\right)}_{= 0}
= 1 .
\]
Thus, the element $1 + qx$ of the
univariate $q$-infiltration bialgebra is grouplike.
The element $1$ is grouplike as well (in fact,
$1$ is grouplike in any $\kk$-bialgebra).
\end{example}

\begin{example}
\label{exa.grouplike.frob-q}
Let $p$ be a prime.
Let $\kk$ be a commutative $\mathbb{F}_p$-algebra.
Let $q \in \kk$.
Let $B$ be the quotient ring of the polynomial ring
$\kx$ by the ideal $I$ generated by $x^p$.
For any $f \in \kx$, we let
$\overline{f}$ denote the projection $f+I$ of $f$
onto this quotient $B$.
(Thus, $\overline{x}^p = \overline{x^p} = 0$.)
Define two $\kk$-algebra homomorphisms
$\Delta_{\uparrow_q} : B \to B \otimes B$
and $\epsilon : B \to \kk$ by setting
\[
\Delta_{\uparrow_q} \left( \overline{x} \right)
= \overline{x} \otimes 1 + 1 \otimes \overline{x} + q \overline{x} \otimes \overline{x}
\qquad \text{ and } \qquad
\epsilon \left( \overline{x} \right) = 0 .
\]
These homomorphisms are well-defined, because Freshman's Dream
(i.e., the property of the Frobenius homomorphism
to be a $\mathbb{F}_p$-algebra homomorphism) shows
that
\begin{align*}
\tup{\overline{x} \otimes 1 + 1 \otimes \overline{x} + q \overline{x} \otimes \overline{x}}^p
&= \tup{\overline{x} \otimes 1}^p + \tup{1 \otimes \overline{x}}^p + q^p \tup{\overline{x} \otimes \overline{x}}^p \\
&= \underbrace{\overline{x}^p}_{=0} \otimes 1 + 1 \otimes \underbrace{\overline{x}^p}_{=0} + q^p \underbrace{\overline{x}^p}_{=0} \otimes \underbrace{\overline{x}^p}_{=0} = 0
\end{align*}
in the commutative $\mathbb{F}_p$-algebra $B \otimes B$.
The $\kk$-algebra $B$, equipped with the comultiplication $\Delta_{\uparrow_q}$
and the counit $\epsilon$, is a $\kk$-bialgebra
-- actually a quotient of the
univariate $q$-infiltration bialgebra defined in
Example~\ref{exa.grouplike.q-pol}.
Unlike the latter, it is always a Hopf algebra
(whether or not $q$ is nilpotent).
Examples of grouplike elements in $B$ are
$1$, $1 + q\overline{x}$ and
$\tup{1 + q\overline{x}}^{-1}
= 1 - q\overline{x} + q^2\overline{x}^2 - q^3\overline{x}^3 \pm \cdots
+ \tup{-q}^{p-1} \overline{x}^{p-1}$.
\end{example}

\begin{example}
\label{exa.grouplike.gx}
Consider the polynomial ring $\kk\left[g,x\right]$
in two (commuting) variables $g$ and $x$ over $\kk$.
Let $B$ be the quotient ring of this ring by the ideal
$J$ generated by $gx$.
For any $f \in \kk\left[g,x\right]$, we let
$\overline{f}$ denote the projection $f+J$ of $f$
onto this quotient $B$.
Define a $\kk$-algebra homomorphism
$\Delta : B \to B \otimes B$
by
\[
\Delta\left(\overline{g}\right) = \overline{g} \otimes \overline{g}
\qquad \text{ and } \qquad
\Delta\left(\overline{x}\right) = \overline{x} \otimes 1 + 1 \otimes \overline{x} ,
\]
and define a $\kk$-algebra homomorphism
$\epsilon : B \to \kk$ by
\[
\epsilon\left(\overline{g}\right) = 1
\qquad \text{ and } \qquad
\epsilon\left(\overline{x}\right) = 0 .
\]
(It is straightforward to see that both of these
homomorphisms are well-defined, since
$\left(\overline{g} \otimes \overline{g}\right)
\left(\overline{x} \otimes 1 + 1 \otimes \overline{x}\right)
= 0$ and $1 \cdot 0 = 0$.)
These homomorphisms (taken as comultiplication and
counit) turn the $\kk$-algebra $B$ into a
$\kk$-bialgebra. Its element $\overline{g}$ is grouplike,
and so are all its powers $\overline{g}^i$.
\end{example}

\begin{example}\label{co-Hadamard}
Let $M$ be a monoid, and let $\tup{\kk[M], \mu_M, 1_M}$ be the monoid algebra of $M$.
As a $\kk$-module, $\kk[M]$ is $\kk^{(M)}$, the set of finitely supported functions $M\to \kk$.
As usual, we identify each $w \in M$ with the indicator function $\delta_w \in \kk^{(M)} = \kk[M]$ that sends $w$ to $1_\kk$ and all other elements of $M$ to $0$.
Thus, $M$ becomes a basis of $\kk[M]$.
The multiplication on $\kk[M]$ is given by $\kk$-bilinearly extending the multiplication of $M$ from this basis to the entire $\kk$-module $\kk[M]$.
Thus, $\kk[M]$ becomes a $\kk$-algebra, with unity equal to the neutral element of $M$ (or, more precisely, its indicator function).

The $\kk$-module dual $\tup{\kk[M]}^\vee$ of $\kk[M]$ can be identified with $\kk^M$, the set of all functions $M\to \kk$, via the natural $\kk$-bilinear pairing between $\kk^{M}$ and $\kk^{(M)}$ given by
\begin{equation}\label{main_pairing_M}
\scal{S}{P}=\sum_{w\in M}S(w)P(w)
\qquad \text{ for all $S \in \kk^{M}$ and $P \in \kk[M] = \kk^{(M)}$.}
\end{equation}

The $\kk$-linear map $\Delta_{\odot} : \kk[M] \to \kk[M] \otimes \kk[M]$ that sends each $w \in M$ to $w \otimes w$ is a $\kk$-algebra homomorphism; it is called the diagonal comultiplication.
The $\kk$-linear map $\epsilon : \kk[M] \to \kk$ sending each $w \in M$ to $1_\kk$ is a $\kk$-algebra homomorphism as well.
Equipping $\kk[M]$ with these two maps $\Delta_{\odot}$ and $\epsilon$, we obtain a $\kk$-bialgebra $\calB = \tup{\kk[M],\mu_M,1_M,\Delta_{\odot},\epsilon}$, in which every $w\in M$ is grouplike.
More generally, a $\kk$-linear combination $\sum_{w \in M} a_w w$ (with $a_w \in \kk$) is grouplike in $\calB$ if and only if the $a_w$ are orthogonal idempotents (i.e., satisfy $a_w^2 = a_w$ for all $w \in M$, and $a_v a_w = 0$ for any distinct $v, w \in M$) and satisfy $\sum_{w \in M} a_w = 1$.
We call $\calB$ the \emph{monoid bialgebra} of $M$.

If $M$ is a group, then $\calB$ is a Hopf algebra.
The converse is also true when $\kk$ is nontrivial.

We remark that there is some overlap between monoid bialgebras and $q$-infiltration bialgebras.
Indeed, let $M$ be the monoid $\left\{x^n \mid n \in \NN\right\}$ of all monomials in one variable $x$.
Let $q \in \kk$, and let
$\tup{\kx, \Delta_{\uparrow_q}, \epsilon}$ be the univariate $q$-infiltration bialgebra as in Example \ref{exa.grouplike.q-pol}.
Then, the $\kk$-algebra homomorphism $\kx \to \kk[M]$ that sends $x$ to $1 + qx$ is a homomorphism of $\kk$-bialgebras from the univariate $q$-infiltration bialgebra $\tup{\kx, \Delta_{\uparrow_q}, \epsilon}$ to the monoid bialgebra
$\calB = \tup{\kk[M],\mu_M,1_M,\Delta_{\odot},\epsilon}$.
When $q$ is invertible, this homomorphism is an isomorphism.
\end{example}

\section{\label{sect.dg1}Grouplikes in a coalgebra}

Let us first state a simple fact in commutative algebra:

\begin{theorem}
\label{thm.old1}Let $g_{1},g_{2},\ldots,g_{n}$ be $n$ elements of a
commutative ring $A$. Let $c_{1},c_{2},\ldots,c_{n}$ be $n$ further elements
of $A$. Assume that
\begin{equation}
\sum_{i=1}^{n}c_{i}g_{i}^{k}=0\ \ \ \ \ \ \ \ \ \ \text{for all }
k\in \NN. \label{eq.thm.old1.ass}
\end{equation}
Then,
\begin{equation}
c_{i}\prod_{j\neq i}\left(  g_{i}-g_{j}\right)
=0\ \ \ \ \ \ \ \ \ \ \text{for all }i\in\left\{  1,2,\ldots,n\right\}  .
\label{eq.thm.old1.claim}
\end{equation}

\end{theorem}

\begin{proof}
Consider the ring $A\left[  \left[  t\right]  \right]  $ of formal power
series in one variable $t$ over $A$.
In this ring, we have
\begin{align*}
\sum_{i=1}^{n}c_{i}\cdot\underbrace{\dfrac{1}{1-tg_{i}}}_{=\sum_{k\in
 \NN}\left(  tg_{i}\right)  ^{k}}  &  =\sum_{i=1}^{n}c_{i}\cdot
\sum_{k\in \NN}
\underbrace{\left(  tg_{i}\right)  ^{k}}_{=t^{k}g_{i}^{k}}
=\sum_{i=1}^{n}c_{i}\cdot\sum_{k\in \NN}t^{k}g_{i}^{k} =\sum
_{k\in \NN}\underbrace{\sum_{i=1}^{n}c_{i}g_{i}^{k}}%
_{\substack{=0\\\text{(by (\ref{eq.thm.old1.ass}))}}} t^k = 0.
\end{align*}
Multiplying this equality by $\prod_{j=1}^{n}\left(  1-tg_{j}\right)  $, we
obtain
\begin{equation}
\sum_{i=1}^{n}c_{i}\prod_{j\neq i}\left(  1-tg_{j}\right)  =0.
\label{pf.thm.old1.polyeq}
\end{equation}
This is an equality in the polynomial ring $A\left[  t\right]  $ (which is a
subring of $A\left[  \left[  t\right]  \right]  $).

Now consider the ring $A\left[t, t^{-1}\right]$ of Laurent
polynomials in $t$ over $A$. There is an $A$-algebra
homomorphism $A\left[  t\right]  \rightarrow A\left[ t, t^{-1}\right]
,\ t\mapsto t^{-1}$ (by the universal property of $A\left[t\right]$).
Applying this homomorphism to both sides of the equality
\eqref{pf.thm.old1.polyeq}, we obtain
\[
\sum_{i=1}^{n}c_{i}\prod_{j\neq i}\left(  1-t^{-1}g_{j}\right)  =0.
\]
Hence,
\begin{align}
\sum_{i=1}^{n}c_{i}\prod_{j\neq i}
 \underbrace{\left(  t-g_{j}\right)}_{= t \left(1 - t^{-1} g_j\right)}
&= \sum_{i=1}^{n}c_{i}\prod_{j\neq i} \left( t \left(1 - t^{-1} g_j\right) \right)
 = \sum_{i=1}^{n}c_{i} t^{n-1} \prod_{j\neq i} \left(1 - t^{-1} g_j\right)
 \nonumber \\
&= t^{n-1} \underbrace{\sum_{i=1}^{n}c_{i}\prod_{j\neq i}\left(  1-t^{-1}g_{j}\right)}_{=0}
 = 0.
\label{pf.thm.old1.4}
\end{align}
This is an equality in the polynomial ring $A\left[  t\right]  $
(which is a subring of $A\left[ t, t^{-1}\right]$); hence, we
can substitute arbitrary values for $t$ in it.

Now, fix $h\in\left\{  1,2,\ldots,n\right\}  $. Substitute $g_h$ for $t$ in
the equality (\ref{pf.thm.old1.4}). The result is
\[
\sum_{i=1}^{n}c_{i}\prod_{j\neq i}\left(  g_{h}-g_{j}\right)  =0.
\]
Hence,
\begin{align*}
0  &  =\sum_{i=1}^{n}c_{i}\prod_{j\neq i}\left(  g_{h}-g_{j}\right)
=c_{h}\prod_{j\neq h}\left(  g_{h}-g_{j}\right)  +\sum_{\substack{i\in\left\{
1,2,\ldots,n\right\}  ;\\i\neq h}}c_{i}\underbrace{\prod_{j\neq i}\left(
g_{h}-g_{j}\right)  }_{\substack{=0\\\text{(since }g_{h}-g_{h}=0\text{ is
among}\\\text{the factors of this product)}}}\\
& \qquad \qquad
\left(\text{here, we have split off the addend for $i=h$ from the sum}\right)
\\
&= c_{h}\prod_{j\neq h}\left(  g_{h}-g_{j}\right)
+ \underbrace{ \sum_{\substack{i\in\left\{
1,2,\ldots,n\right\}  ;\\i\neq h}}c_{i} 0}_{=0}
 =c_{h}\prod_{j\neq h}\left(  g_{h}-g_{j}\right)  .
\end{align*}
Now forget that we fixed $h$. Thus, we have shown that $c_{h}\prod_{j\neq
h}\left(  g_{h}-g_{j}\right)  =0$ for all $h\in\left\{  1,2,\ldots,n\right\}
$. Renaming $h$ as $i$, we obtain (\ref{eq.thm.old1.claim}). Theorem
\ref{thm.old1} is proven.
\end{proof}

Our first main result is now the following:

\begin{theorem}
\label{thm.old2}
Let $g_{1},g_{2},\ldots,g_{n}$ be $n$ grouplike elements of a
$\kk$-coalgebra $C$. Let $c_{1},c_{2},\ldots,c_{n}\in \kk$.
Assume that $\sum_{i=1}^{n}c_{i}g_{i}=0$. Then,
\[
c_{i}\prod_{j\neq i}\left(  g_{i}-g_{j}\right)  =0\ \ \ \ \ \ \ \ \ \ \text{in
}\operatorname*{Sym}C\text{ for all }i\in\left\{  1,2,\ldots,n\right\}  .
\]
Here, $\operatorname*{Sym}C$ denotes the symmetric algebra of the
$\kk$-module $C$.
\end{theorem}

Our proof of this theorem will rely on a certain family of
maps defined for any $\kk$-coalgebra:

\begin{definition} \label{def.iterated-com}
Let $C$ be a $\kk$-coalgebra.
We then define a sequence of $\kk$-linear maps
$\Delta^{\left(-1\right)} ,
\Delta^{\left(0\right)} ,
\Delta^{\left(1\right)} , \ldots$, where
$\Delta^{\left(k\right)}$ is a $\kk$-linear map
from $C$ to $C^{\otimes \left(k+1\right)}$ for any integer $k \geq -1$.
Namely, we define them recursively by setting
$\Delta^{\left(-1\right)} = \epsilon$
and $\Delta^{\left(0\right)} = \id_C$
and $\Delta^{\left(k\right)} = \left(\id_C \otimes \Delta^{\left(k-1\right)}\right) \circ \Delta$
for all $k \geq 1$.
These maps $\Delta^{\left(  k\right)  }$ are known as the
\emph{iterated comultiplications} of $C$, and are studied
further, e.g., in \cite[Exercise 1.4.20]{GriRei}. 
\end{definition}

\begin{proof}[Proof of Theorem~\ref{thm.old2}.]
Every grouplike element $g\in C$ and every $k\in \NN$
satisfy
\begin{equation}
\Delta^{\left(  k-1\right)  }g=g^{\otimes k} \label{pf.thm.old2.Delkg}
\end{equation}
(where $g^{\otimes k}$ means $\underbrace{g\otimes g\otimes\cdots\otimes
g}_{k\text{ times}}\in C^{\otimes k}$). Indeed, this can be easily proven by
induction on $k$.

The $g_{i}$ (for all $i \in \left\{1,2,\ldots,n\right\}$)
are grouplike. Thus, (\ref{pf.thm.old2.Delkg}) (applied to $g = g_i$)
shows that $\Delta^{\left(  k-1\right)  }g_{i}=g_{i}^{\otimes k}$
for each $k\in \NN$ and each $i \in \left\{1,2,\ldots,n\right\}$.
Hence, applying the $\kk$-linear map $\Delta^{\left(  k-1\right)  }$ to
both sides of the equality $\sum_{i=1}^{n}c_{i}g_{i}=0$, we obtain
\begin{align}
\sum_{i=1}^{n}c_{i}g_{i}^{\otimes k}=0
\label{pf.thm.old2.tensor-sum}
\end{align}
for each $k\in \NN$.

But recall that the symmetric algebra $\operatorname*{Sym}C$ is
defined as a quotient of the tensor algebra
$T\left(  C\right)  $. Hence, there is a canonical projection
from $T\left(  C\right)$ to $\operatorname*{Sym}C$ that sends
each tensor $x_1 \otimes x_2 \otimes \cdots \otimes x_m
\in T\left(  C\right)$ to $x_1 x_2 \cdots x_m \in
\operatorname*{Sym}C$. In particular, this projection sends
$g_{i}^{\otimes k}$ to $g_{i}^{k}$ for each
$i \in \left\{1,2,\ldots,n\right\}$ and each $k \in \NN$.
Thus, applying this projection to both sides of
\eqref{pf.thm.old2.tensor-sum}, we obtain
$\sum_{i=1}^{n}c_{i}g_{i}^{k}=0$ in $\operatorname*{Sym}C$ for all
$k\in \NN$. Thus, Theorem \ref{thm.old1} can be
applied to $A=\operatorname*{Sym}C$.
We conclude that
\[
c_{i}\prod_{j\neq i}\left(  g_{i}-g_{j}\right)  =0\ \ \ \ \ \ \ \ \ \ \text{in
}\operatorname*{Sym}C\text{ for all }i\in\left\{  1,2,\ldots,n\right\}  .
\]
This proves Theorem \ref{thm.old2}.
\end{proof}

From Theorem \ref{thm.old2}, we obtain the following classical fact
\cite[Lemma 2.1.12]{Radford-Hopf}:

\begin{corollary}
\label{cor.old2-field}Assume that $\kk$ is a field. Let
$g_{1},g_{2},\ldots,g_{n}$ be $n$ distinct grouplike elements
of a $\kk$-coalgebra $C$.
Then, $g_{1},g_{2},\ldots,g_{n}$ are $\kk$-linearly independent.
\end{corollary}

\begin{proof}
[Proof of Corollary \ref{cor.old2-field}.]The $\kk$-module $C$
is free (since $\kk$ is a field). Thus, the $\kk$-algebra
$\operatorname*{Sym}C$ is isomorphic to a polynomial algebra
over $\kk$,
and thus is an integral domain (again since $\kk$ is a field). 
Its elements
$g_{i}-g_{j}$ for $j\neq i$ are nonzero (since $g_{1},g_{2},\ldots,g_{n}$ are
distinct), and thus their products $\prod_{j\neq i}\left(  g_{i}-g_{j}\right)
$ (for $i\in\left\{  1,2,\ldots,n\right\}  $) are nonzero as well (since
$\operatorname*{Sym}C$ is an integral domain).

Let $c_{1},c_{2},\ldots,c_{n}\in \kk$ be such that $\sum_{i=1}^{n}
c_{i}g_{i}=0$. Then, Theorem \ref{thm.old2} yields
\[
c_{i}\prod_{j\neq i}\left(  g_{i}-g_{j}\right)  =0\ \ \ \ \ \ \ \ \ \ \text{in
}\operatorname*{Sym}C\text{ for all }i\in\left\{  1,2,\ldots,n\right\}  .
\]
We can cancel the $\prod_{j\neq i}\left(  g_{i}-g_{j}\right)  $ factors from
this equality (since these factors $\prod_{j\neq i}\left(  g_{i}-g_{j}\right)
$ are nonzero, and since $\operatorname*{Sym}C$ is an integral domain). Thus,
we obtain the equalities $c_{i}=0$ in $\operatorname*{Sym}C$ for all
$i\in\left\{  1,2,\ldots,n\right\}  $. In other words, $c_{i}=0$ in
$\kk$ for all $i\in\left\{  1,2,\ldots,n\right\}  $ (since $\kk$
embeds into $\operatorname*{Sym}C$).

Now, forget that we fixed $c_{1},c_{2},\ldots,c_{n}$. We thus have proven that
if $c_{1},c_{2},\ldots,c_{n} \in \kk$
are such that $\sum_{i=1}^{n}c_{i}g_{i}=0$, then
$c_{i}=0$ for all $i\in\left\{  1,2,\ldots,n\right\}  $. In other words,
$g_{1},g_{2},\ldots,g_{n}$ are $\kk$-linearly independent. This proves
Corollary \ref{cor.old2-field}.
\end{proof}

\begin{remark}
\label{rmk.old2-field.necc}
Example~\ref{exa.grouplike.q-pol} illustrates why
we required $\kk$ to be a field in
Corollary \ref{cor.old2-field}. Indeed, if $q \in \kk$
is a zero-divisor but nonzero, then the two grouplike
elements $1$ and $1 + qx$ of the $\kk$-bialgebra in
Example~\ref{exa.grouplike.q-pol} fail to be
$\kk$-linearly independent.
Theorem \ref{thm.old2}, however, provides a weaker
version of linear independence that is still
satisfied.
\end{remark}

\begin{remark}
A different generalization of Corollary~\ref{cor.old2-field}
to commutative rings was found by Mr\v cun in
\cite[Proposition 1.1]{Mrcun}.
We do not know whether it can be derived from our
Theorem \ref{thm.old2}.
\end{remark}

\section{\label{sect.dg2}Grouplikes over id-unipotents in a bialgebra}

\subsection{The notion of id-unipotence}

In this section, we shall use the following concept:

\begin{definition}\label{local_nilp}
Let $B$ be a $\kk$-bialgebra.
An element $b\in B$ is said to be \emph{id-unipotent} if there exists
some $m\in\mathbb{Z}$ such that every nonnegative integer $n>m$ satisfies
\begin{equation}
\left(  \eta\epsilon-\id\right)  ^{\oast n}\left(  b\right)
=0 \label{eq.locnil}
\end{equation}
(where $\id$ means $\id_B$). We
shall refer to such an $m$ as a \emph{degree-upper bound} of $b$.
\end{definition}

Before we move on to studying id-unipotent elements, let us show
several examples:

\begin{example}
\label{exam.id-unipotence.graded}
Let $B$ be a connected graded $\kk$-bialgebra.
(The word ``graded'' here means ``$\NN$-graded'', and it is assumed
that all structure maps $\mu, \eta, \Delta, \epsilon$ preserve the
grading. The word ``connected'' means that the $0$-th homogeneous
component $B_0$ is isomorphic to $\kk$.)
Then, each $b\in B$ is id-unipotent, and if $b\in B$ is
homogeneous of degree $k$, then $k$ is a degree-upper bound of $b$.
\end{example}

\begin{example}
\label{exam.id-unipotence.x}
Let $q$, $\kx$, $\Delta_{\uparrow_q}$ and $\epsilon$ be as in
Example~\ref{exa.grouplike.q-pol}.
Assume that $q$ is nilpotent, with $q^m = 0$ for some $m \in \NN$.
It is easy to see (by induction on $n$) that
$\tup{\eta \epsilon - \id}^{\oast n}\tup{x} = -(-q)^{n-1} x^n$
in $\kx$ for each $n \geq 1$.
Thus, for each $n > m$, we have
$\tup{\eta \epsilon - \id}^{\oast n}\tup{x} = -(-q)^{n-1} x^n = 0$
(since $q^m = 0$ and thus $(-q)^{n-1} = 0$).
Hence, the element $x$ is id-unipotent in $\kx$, with $m$
being a degree-upper bound of $x$.
Combining this observation with
Corollary~\ref{cor.L.subalg} further below, it follows
easily that every element of $\kx$ is id-unipotent
(albeit $m$ will not always be a degree-upper bound).
\end{example}

\begin{example}
\label{exam.id-unipotence.xquot}
Let $p$, $\kk$, $q$ and $B$ be as in Example~\ref{exa.grouplike.frob-q}.
It is easy to see (by induction on $n$) that
$\tup{\eta \epsilon - \id}^{\oast n}\tup{\overline{x}} = -(-q)^{n-1} \overline{x}^n$
for each $n \geq 1$.
Thus, for each $n \geq p$, we have
$\tup{\eta \epsilon - \id}^{\oast n}\tup{\overline{x}} = -(-q)^{n-1} \overline{x}^n = 0$
(since $\overline{x}^p = 0$ and thus $\overline{x}^n = 0$).
Hence, the element $\overline{x}$ is id-unipotent in $B$, with $p-1$
being a degree-upper bound of $\overline{x}$.
\end{example}

\begin{remark} \label{rmk.Deltaplusnilp}
The notion of id-unipotence is connected with the
classical notion of \emph{$\Delta_+$-nilpotence}, but it
is a weaker notion.
Let us briefly describe the latter notion and the connection.

Consider a $\kk$-bialgebra $B$ with comultiplication $\Delta$,
counit $\epsilon$ and unit map $\eta$.
Then, we have a canonical direct sum decomposition
$B= \kk 1_B\oplus B_+$ of the $\kk$-module $B$,
where $B_+=\ker(\epsilon)$ (and where $1_B = \eta(1_{\kk})$
is the unity of $B$).
The corresponding projections are
$\eta \epsilon$ (projecting $B$ onto $\kk 1_B$)
and $\id-\eta \epsilon$ (projecting $B$ onto $B_+$).
Thus, we set
$\overline{\id} = \id-\eta \epsilon : B \to B$
(this projection, in a way, ``eliminates the constant term'').
For each $k \in \NN$, we define a map
\[
\Delta_+^{\left(k\right)}
= \overline{\id}^{\otimes \left(k+1\right)} \circ \Delta^{\left(k\right)}
: B \to B^{\otimes \left(k+1\right)} ,
\]
where $\overline{\id}^{\otimes \left(k+1\right)}
= \overline{\id} \otimes \overline{\id} \otimes \cdots \otimes \overline{\id}$
(with $k+1$ tensorands) and where
$\Delta^{\left(k\right)} : B \to B^{\otimes \left(k+1\right)}$
is the iterated comultiplication
from Definition~\ref{def.iterated-com}.

Now, an element $b$ of $B$ is said to be \emph{$\Delta_+$-nilpotent} if there exists
some $m\in\mathbb{Z}$ such that every nonnegative integer $k>m$ satisfies
\begin{equation}
\Delta_+^{\left(k\right)} \left(  b\right) =0 .
\label{eq.locnil-Deltaplus}
\end{equation}

It is easy to see that every $n \geq 1$ satisfies
\[
\left(\eta\epsilon-\id\right)^{\oast n}=(-\overline{\id})^{\oast n}
=(-1)^n\mu^{\left(n-1\right)} \circ \Delta_+^{\left(n-1\right)},
\]
where $\mu^{\left(n-1\right)} : B^{\otimes n} \to B$ is the
``iterated multiplication'' map that sends each pure tensor
$b_1 \otimes b_2 \otimes \cdots \otimes b_n$ to
$b_1 b_2 \cdots b_n \in B$.
Thus, every $\Delta_+$-nilpotent element of $B$ is
id-unipotent.
The converse, however, is not true.
For instance,
the element $\overline{x}$ in Example~\ref{exam.id-unipotence.xquot}
is id-unipotent but (in general) not $\Delta_+$-nilpotent.
(But the elements $x$ in Example~\ref{exam.id-unipotence.x}
and $b$ in Example~\ref{exam.id-unipotence.graded} are
$\Delta_+$-nilpotent.)%
\footnote{Nevertheless, at least in one important case, the
two concepts are closely intertwined.
Namely, if $\kk$ is a $\Q$-algebra and $B$ is cocommutative,
then we have the following chain of equivalences:
\[
\tup{\text{every $b\in B$ is id-unipotent}}
\Longleftrightarrow
\tup{B=\calU(\Prim(B))}
\Longleftrightarrow
\tup{\text{every $b\in B$ is }\Delta_+\mbox{-nilpotent}} ,
\]
where ``$B=\calU(\Prim(B))$'' means that $B$ is the universal enveloping bialgebra
(see \cite[ch II \S 1]{Bourbaki-Lie1} for the definition) of its primitive elements.
The two equivalence signs follow easily from \cite[Theorem 1]{DMTCN14}
(noticing that a filtration as in \cite[Theorem 1 condition 1]{DMTCN14}
forces every $b \in B$ to be $\Delta_+$-nilpotent).
This is a variant of the well-known
Cartier--Quillen--Milnor--Moore theorem \cite{Cartier,MM65}
that relies on
(formal) convergence of infinite sums instead of primitive
generation of the Hopf algebra.}
\end{remark}

%\begin{remark} \tcb{This notion is linked to the one of locally finite monoids. 
%Let $(M,\mu,1_M)$ be a monoid considered as an alphabet and $M_+^*$ the free 
%monoid built on $M_+:=M\rsetminus \{1_M\}$ (the words of which will be noted as lists to avoid confusion). 
%One can define the notion of iterated multiplication 
%$\mu^{*}:\ M_+^*\to M$ by 
%\begin{equation}
%\mu_{*}[m_1,\cdots ,m_n]=m_1\cdots m_n\mbox{ (product within $M$)}
%\end{equation}
%a monoid will be said \emph{locally finite}\footnote{This notion is also defined within the exercises in  
%\cite[Ch. III, \S 2 Ex 21]{Bourbaki-Alg1} } (see, \cite{Eil}, VII.4) iff $\mu^{*}$ has finite fibers.
%We have 
%\begin{proposition} Let $\kk$ be a ring, then\\ 
%i) The law $\mu:\ \kk[M]\otimes \kk[M]\to \kk[M]$ is dualizable as  
%\begin{equation}\label{dual1}
%\Delta_\mu:\ \kk[M]\to \kk[M]\otimes \kk[M]\hookrightarrow \kk[[M\otimes M]]
%\end{equation}
%iff $M$ satisfies condition (D) in \cite[Ch. III, \S 2.10]{Bourbaki-Alg1}\\
%ii) In this case, within the coalgebra $b\in (\kk[M],\Delta_\mu,\epsilon)$ is locally nilpotent iff $\mu_*^{-1}(b)$ is 
%finite.\\ 
%iii) Every $m\in M_+$ is locally nilpotent iff $M$ is locally finite.  
%\end{proposition}
%}
%The arrow $\kk[M]\hookrightarrow \kk[[M\otimes M]]$ (where $\kk[[M\otimes M]]$ is the total algebra of 
%$M\otimes M\simeq M\times M$) in \eqref{dual1} is into is due to the fact that $\kk[M]\otimes \kk[M]$ is free with basis $M\times M$ and this can be tested by hand. 
%\end{remark}

\subsection{The linear independence}

We shall also use a slightly generalized notion of linear independence:

\begin{definition}
Let $g_{1},g_{2},\ldots,g_{n}$ be some elements of a $\kk$-algebra $A$.
Let $S$ be a subset of $A$.

\begin{enumerate}

\item[\textbf{(a)}] We say that the elements $g_{1},g_{2},\ldots,g_{n}$ are
\emph{left $S$-linearly independent} if every $n$-tuple $\left(
s_{1},s_{2},\ldots,s_{n}\right)  \in S^{n}$ of elements of $S$ satisfying
$\sum_{i=1}^{n}s_{i}g_{i}=0$ must satisfy $\left(  s_{i}=0\text{ for all
}i\right)  $.

\item[\textbf{(b)}] The notion of ``\emph{right $S$-linearly
independent}'' is defined in the same way 
as ``left $S$-linearly independent'', except
that the sum $\sum_{i=1}^{n}s_{i}g_{i}$ is replaced by $\sum_{i=1}^{n}
g_{i}s_{i}$.

\item[\textbf{(c)}] If the $\kk$-algebra $A$ is commutative, then these two
notions are identical, and we just call them
``\emph{$S$-linearly independent}''.

\end{enumerate}
\end{definition}

We also recall that an element $a$ of a commutative
ring $A$ is said to be \emph{regular} if it is a
non-zero-divisor -- i.e., if it has the property that
whenever $b$ is an element of $A$ satisfying $ab = 0$,
then $b = 0$.

We can now state the second of our main theorems:

\begin{theorem}
\label{thm.1}
Let $B$ be a commutative $\kk$-bialgebra.
Let $L$ denote the set of all id-unipotent
elements of $B$.

Let $g_{1},g_{2},\ldots,g_{n}$ be $n$ grouplike elements of $B$.
Let us make two assumptions:
\begin{itemize}
\item \textit{Assumption 1:} For any $1 \leq i < j \leq n$, the element
$g_i - g_j$ of $B$ is regular. % (i.e., a non-zero-divisor).
\item \textit{Assumption 2:} For any $1 \leq i \leq n$, the element
$g_i$ of $B$ is regular.
\end{itemize}
Then, $g_{1},g_{2},\ldots,g_{n}$ are $L$-linearly independent.
\end{theorem}

Before we prove this theorem, some remarks are in order.

\begin{remark}
\label{rmk.L.submodule}
The $L$ in Theorem~\ref{thm.1} is a $\kk$-submodule
of $B$. (This is easily seen directly: If $u \in B$
and $v \in B$ are id-unipotent with degree-upper
bounds $p$ and $q$, respectively, then $\lambda u
+ \mu v$ is id-unipotent with degree-upper
bound $\max\left\{p,q\right\}$ for any
$\lambda, \mu \in \kk$.)

It can be shown that $L$ is a $\kk$-subalgebra of $B$.
(See Corollary~\ref{cor.L.subalg} further below;
we will not need this to prove Theorem~\ref{thm.1}.)
In general, $L$ is not a $\kk$-subcoalgebra of $B$.
(See Example~\ref{exa.L.not-subcoalg} for a
counterexample.)
\end{remark}

\begin{remark}
\label{rmk.thm.1.gx}
Example~\ref{exa.grouplike.gx} illustrates why
we required Assumption 2 to hold in
Theorem~\ref{thm.1}. Indeed, in this example,
$\overline{g}$ is grouplike, and $\overline{x}$ is
id-unipotent (with degree-upper bound $1$), so
the equality
$\overline{x} \overline{g} = \overline{g} \overline{x} = 0$
shows that even the single grouplike element
$\overline{g}$ is not $L$-linearly independent.

The necessity of Assumption 1 can be illustrated by
Remark~\ref{rmk.old2-field.necc} again.
\end{remark}

To prove Theorem~\ref{thm.1}, we shall use a few lemmas.
The first one is a classical result \cite[Exercise 1.5.11(a)]{GriRei}:

\begin{lemma}
\label{lem.BtoA.bimap*bimap}Let $C$ be a $\kk$-bialgebra. Let $A$ be a
commutative $\kk$-algebra. Let $p:C\rightarrow A$ and $q:C\rightarrow
A$ be two $\kk$-algebra homomorphisms. Then, the map $p\oast
q:C\rightarrow A$ is also a $\kk$-algebra homomorphism.
\end{lemma}

\begin{proof}
[Proof of Lemma \ref{lem.BtoA.bimap*bimap}.]Let $a\in C$ and $b\in C$. We
shall prove that $\left(  p\oast q\right)  \left(  ab\right)  =\left(  p\oast
q\right)  \left(  a\right)  \cdot\left(  p\oast q\right)  \left(  b\right)  $.

Using Sweedler notation, write $\Delta\left(  a\right)  =\sum_{\left(
a\right)  }a_{\left(  1\right)  }\otimes a_{\left(  2\right)  }$ and
$\Delta\left(  b\right)  =\sum_{\left(  b\right)  }b_{\left(  1\right)
}\otimes b_{\left(  2\right)  }$. Then, the multiplicativity of $\Delta$
yields $\Delta\left(  ab\right)  =\sum_{\left(  a\right)  }\sum_{\left(
b\right)  }a_{\left(  1\right)  }b_{\left(  1\right)  }\otimes a_{\left(
2\right)  }b_{\left(  2\right)  }$. Thus,
\begin{align*}
\left(  p\oast q\right)  \left(  ab\right)   &  =\sum_{\left(  a\right)  }
\sum_{\left(  b\right)  }\underbrace{p\left(  a_{\left(  1\right)  }b_{\left(
1\right)  }\right)  }_{\substack{=p\left(  a_{\left(  1\right)  }\right)
p\left(  b_{\left(  1\right)  }\right)  \\\text{(since }p\text{ is a
}\kk\text{-algebra}\\\text{homomorphism)}}}\quad \underbrace{q\left(
a_{\left(  2\right)  }b_{\left(  2\right)  }\right)  }_{\substack{=q\left(
a_{\left(  2\right)  }\right)  q\left(  b_{\left(  2\right)  }\right)
\\\text{(since }q\text{ is a }\kk\text{-algebra}\\\text{homomorphism)}}}
\\
&  =\sum_{\left(  a\right)  }\sum_{\left(  b\right)  }p\left(  a_{\left(
1\right)  }\right)  p\left(  b_{\left(  1\right)  }\right)  q\left(
a_{\left(  2\right)  }\right)  q\left(  b_{\left(  2\right)  }\right)  \\
&  =\underbrace{\left(  \sum_{\left(  a\right)  }p\left(  a_{\left(  1\right)
}\right)  q\left(  a_{\left(  2\right)  }\right)  \right)  }_{=\left(  p\oast
q\right)  \left(  a\right)  }\cdot\underbrace{\left(  \sum_{\left(  b\right)
}p\left(  b_{\left(  1\right)  }\right)  q\left(  b_{\left(  2\right)
}\right)  \right)  }_{=\left(  p\oast q\right)  \left(  b\right)  }\\
&  \qquad\qquad\left(  \text{since }A\text{ is commutative}\right)  \\
&  =\left(  p\oast q\right)  \left(  a\right)  \cdot\left(  p\oast q\right)
\left(  b\right)  .
\end{align*}
Now, forget that we fixed $a$ and $b$. We thus have proven that $\left(  p\oast
q\right)  \left(  ab\right)  =\left(  p\oast q\right)  \left(  a\right)
\cdot\left(  p\oast q\right)  \left(  b\right)  $ for each $a,b\in C$. An even
simpler argument shows that $\left(  p\oast q\right)  \left(  1\right)  =1$.
Combining these results, we conclude that $p\oast q:C\rightarrow A$ is a
$\kk$-algebra homomorphism. This completes the proof of Lemma
\ref{lem.BtoA.bimap*bimap}.
\end{proof}

\begin{lemma}
\label{lem.BtoA.preserveidk}
Let $B$ be a commutative $\kk$-bialgebra.
Let $k\in \NN$. Then, $\id^{\oast k}:B\rightarrow B$
is a $\kk$-algebra homomorphism.
\end{lemma}

\begin{proof}
[Proof of Lemma \ref{lem.BtoA.preserveidk}.]This follows by straightforward
induction on $k$. Indeed, the base case $k=0$ follows from the fact that
$\eta_B \epsilon_B$ is a $\kk$-algebra homomorphism, while the
induction step uses Lemma \ref{lem.BtoA.bimap*bimap} (applied to $C=B$ and
$A=B$) and the fact that $\id_B$ is a $\kk$-algebra homomorphism.
\end{proof}

\begin{lemma}
\label{lem.id*k.grouplike}
Let $B$ be a $\kk$-bialgebra.
Let $g\in B$ be any grouplike element.
Let $k\in \NN$.
Then, $\id^{\oast k}\left(  g\right) =g^{k}$.
(Here, $\id$ means $\id_B$.)
\end{lemma}

\begin{proof}
[Proof of Lemma \ref{lem.id*k.grouplike}.]This follows easily by induction on
$k$.
\end{proof}

\begin{lemma}
\label{lem.dub.nonneg}Let $B$ be a $\kk$-bialgebra.
Let $b\in B$ be id-unipotent and nonzero. Let $m$
be a degree-upper bound of $b$. Then, $m\in \NN$.
\end{lemma}

\begin{proof}
[Proof of Lemma \ref{lem.dub.nonneg}.]Assume the contrary. Thus, $m$ is
negative (since $m$ is an integer). Therefore, every $n \in \NN$
satisfies $n > m$.
Hence, by the definition of a degree-upper
bound, we must have $\left(  \eta\epsilon-\id\right)  ^{\oast
n}\left(  b\right)  =0$ for each $n\in \NN$. Thus, in particular, we
have $\left(  \eta\epsilon-\id\right)  ^{\oast0}\left(
b\right)  =0$ and $\left(  \eta\epsilon-\id\right)  ^{\oast
1}\left(  b\right)  =0$. Hence,
\[
\underbrace{\left(  \eta\epsilon
-\id\right)  ^{\oast0}\left(  b\right)  }_{=0}
-\underbrace{\left(  \eta\epsilon-\id\right)  ^{\oast
1}\left(  b\right)  }_{=0}=0 .
\]
This contradicts
\[
\underbrace{\left(  \eta\epsilon-\id\right)  ^{\oast0}}_{=\eta\epsilon}
\left(  b\right)  -\underbrace{\left(  \eta\epsilon
-\id\right)  ^{\oast1}}_{=\eta\epsilon-\id}\left(  b\right)
=\left(  \eta\epsilon\right)  \left(  b\right)  -\left(
\eta\epsilon-\id\right)  \left(  b\right)
=\id\left(  b\right)  =b\neq0
\]
(since $b$ is nonzero). This contradiction completes the proof of Lemma
\ref{lem.dub.nonneg}.
\end{proof}

\begin{lemma}
\label{lem.2}Let $B$ be a $\kk$-bialgebra.
Let $b\in B$ be id-unipotent. Let $m$ be a degree-upper
bound of $b$. Then, in the ring $B\left[  \left[  t\right]  \right]  $ of
power series\footnote{This ring is defined in the usual way even
if $B$ is not commutative. Thus, $t$ will commute with every
power series in $B\left[\left[t\right]\right]$.}, we have
\[
\left(  1-t\right)  ^{m+1}\sum_{k\in \NN}
\id^{\oast k}\left(  b\right)  t^{k}
=
\sum_{k=0}^{m} \left(-1\right)^k \left(  \eta
\epsilon-\id\right)  ^{\oast k}\left(  b\right)  t^{k}\left(
1-t\right)  ^{m-k}.
\]
(Here, $\id$ means $\id_B$.)
\end{lemma}

\begin{proof}
[Proof of Lemma \ref{lem.2}.]
We know that $m$ is a degree-upper bound of $b$.
In other words, we have
$\left(  \eta\epsilon-\id\right)  ^{\oast n}\left(  b\right)
=0$ for every nonnegative integer $n > m$
(by \eqref{eq.locnil}).
In other words, we have
\begin{equation}
\left(
\eta\epsilon-\id\right)  ^{\oast k}\left(  b\right) = 0
\qquad \text{for every nonnegative integer $k > m$.}
\label{pf.lem.2.locnil}
\end{equation}

Let $\iota$ denote the canonical inclusion
$B \hookrightarrow B\left[\left[t\right]\right]$.
This is an injective $\kk$-algebra homomorphism, and
will be regarded as an inclusion.

The convolution algebra
$\left(  \Hom\left(  B,B  \right),\oast\right)  $
embeds in the convolution algebra
$\left(  \Hom\left(  B,B\left[\left[t\right]\right]\right)
,\oast\right)  $
(indeed, there is an injective $\kk$-algebra
homomorphism from the former to the latter, which sends
each
$f \in \Hom\left(  B,B  \right)$
to
$\iota \circ f \in \Hom\left(  B,B\left[\left[t\right]\right]\right)$).
Again, we shall regard this embedding as an inclusion
(so we will identify each
$f \in \Hom\left(  B,B  \right)$ with
$\iota \circ f$).

The convolution algebra
$\left(  \Hom\left(  B,B\left[\left[t\right]\right]\right)
,\oast\right)  $
has unity
$\eta\epsilon = \eta_{B\left[\left[t\right]\right]} \epsilon_B$.
Thus, from the classical power-series identity
$\left(1-u\right)^{-1} = \sum_{k\in \NN} u^k$, we obtain the equality
\begin{align}
\left(\eta\epsilon - t f\right)^{\oast\left(-1\right)}
= \sum_{k \in \NN} \left(tf\right)^{\oast k}
= \sum_{k \in \NN} t^k f^{\oast k}
\label{pf.lem.2.geom-series}
\end{align}
for every $f \in \Hom\left(  B,B\left[\left[t\right]\right]\right)$.

Now,
\begin{align}
\sum_{k\in \NN}\id^{\oast k}\left(  b\right)
t^{k}
&=\underbrace{\left(  \sum_{k\in \NN} t^k \id^{\oast k}\right)
}_{\substack{=\left(  \eta\epsilon
-t\id\right)  ^{\oast\left(  -1\right)  }
\\ \text{(by \eqref{pf.lem.2.geom-series}, applied to $f = \id$)}
}}\left(  b\right) \nonumber\\
&=\left(  \eta\epsilon-t\id\right)  ^{\oast\left(  -1\right)
}\left(  b\right)  .
\label{pf.lem.2.1}
\end{align}
But
\[
\eta\epsilon-t\id
= \eta\epsilon\left(  1-t\right)
- t \left( \id - \eta\epsilon \right)
= \left(  \eta\epsilon
- t \left( \id - \eta\epsilon \right)
\dfrac{1}{1-t} \right) \left(1-t\right)
\]
and therefore
\begin{align*}
\left(  \eta\epsilon-t\id\right)  ^{\oast\left(  -1\right)
}  &  = \left( \left(  \eta\epsilon
- t \left( \id - \eta\epsilon \right)
\dfrac{1}{1-t} \right) \left(1-t\right) \right)^{\oast\left(-1\right)} \\
&  =\left(  1-t\right)  ^{-1}\cdot\left( \eta\epsilon
- t \left( \id - \eta\epsilon \right)  \dfrac{1}{1-t} \right)
^{\oast\left(  -1\right)  }\\
&  =\left(  1-t\right)  ^{-1}\cdot\sum_{k\in \NN} t^k
\left( \left( \id - \eta\epsilon \right)  \dfrac{1}{1-t} \right)^{\oast k}
% \left(
% \id - \eta \epsilon \right)  ^{\oast k}\left(  \dfrac{1}{1-t}\right)
% ^{k}
\end{align*}
(by \eqref{pf.lem.2.geom-series}, applied
to $f = \left( \id - \eta\epsilon \right)  \dfrac{1}{1-t}$).
Hence, \eqref{pf.lem.2.1} becomes
\begin{align*}
\sum_{k\in \NN}\id^{\oast k}\left(  b\right)
t^{k}
&  =\underbrace{\left(  \eta\epsilon-t\id\right)
^{\oast\left(  -1\right)  }}_{=
\left(  1-t\right)  ^{-1}\cdot\sum_{k\in \NN} t^k
\left( \left( \id - \eta\epsilon \right)  \dfrac{1}{1-t} \right)^{\oast k}
}\left(  b\right) \\
&= \left(  1-t\right)  ^{-1}\cdot\sum_{k\in \NN} t^k
\underbrace{\left( \left( \id - \eta\epsilon \right)
\dfrac{1}{1-t} \right)^{\oast k} \left(b\right)}_{
= \left( \id - \eta\epsilon \right)^{\oast k} \left(b\right)
\left(\dfrac{1}{1-t} \right)^k} \\
&  =\left(  1-t\right)  ^{-1}\cdot
\sum_{k\in \NN}\left(
\underbrace{\id - \eta\epsilon}_{= -\left(\eta\epsilon - \id\right)}
\right)  ^{\oast k}\left(  b\right)  \left(
\dfrac{t}{1-t}\right)  ^{k} \\
&  =\left(  1-t\right)  ^{-1}\cdot
\sum_{k\in \NN}\underbrace{\left(
-\left(\eta\epsilon - \id\right)
\right)  ^{\oast k}}_{=\left(-1\right)^k
\left(\eta \epsilon - \id\right)^{\oast k}}
\left(  b\right)  \left(
\dfrac{t}{1-t}\right)  ^{k} \\
&  =\left(  1-t\right)  ^{-1}\cdot\underbrace{\sum_{k\in \NN}
\left(-1\right)^k\left(
\eta\epsilon-\id\right)  ^{\oast k}\left(  b\right)  \left(
\dfrac{t}{1-t}\right)  ^{k}}_{\substack{\text{All addends of this sum for
}k>m\text{ are}\\\text{zero, due to (\ref{pf.lem.2.locnil}).}}}\\
&  =\left(  1-t\right)  ^{-1}\cdot\sum_{k=0}^{m}
\left(-1\right)^k \left(  \eta\epsilon
-\id\right)  ^{\oast k}\left(  b\right)
\left(  \dfrac{t}{1-t}\right)  ^{k}.
\end{align*}
Multiplying both sides of this equality by $\left(  1-t\right)  ^{m+1}$,
we get
\begin{align*}
\left(  1-t\right)  ^{m+1}
\sum_{k\in \NN}\id^{\oast k}\left(  b\right)  t^{k}
&= \left(  1-t\right)  ^m\cdot\sum_{k=0}^{m}
\left(-1\right)^k \left(  \eta\epsilon
-\id\right)  ^{\oast k}\left(  b\right)
\left(  \dfrac{t}{1-t}\right)  ^{k}
\\
&=\sum_{k=0}^{m} \left(-1\right)^k \left(  \eta
\epsilon-\id\right)  ^{\oast k}\left(  b\right)  t^{k}\left(
1-t\right)  ^{m-k}.
\end{align*}
This proves Lemma \ref{lem.2}.
\end{proof}

\begin{lemma}
\label{lem.laurent-sub}
Let $A$ be a commutative ring.
Let $m \in \NN$ and $u \in A$.
Let $F\left(t\right) \in A\left[t\right]$ be a polynomial of degree $\leq m$.
Set
\[
G\left(t\right) = t^m F\left(\dfrac{u}{t}\right) \in A\left[t,t^{-1}\right] .
\]
Then, $G\left(t\right)$ is actually a polynomial in $A\left[t\right]$ and
satisfies
\begin{equation}
G\left(u\right) = u^m F\left(1\right) .
\label{eq.lem.laurent-sub.Gu=}
\end{equation}
\end{lemma}

\begin{proof}
[Proof of Lemma \ref{lem.laurent-sub}.]
We have $G\left(t\right) = t^m F\left(\dfrac{u}{t}\right) \in A\left[t\right]$
because $F\left(t\right)$ has degree $\leq m$.
It remains to show that $G\left(u\right) = u^m F\left(1\right)$.
It is tempting to argue this by substituting
$u$ for $t$ in the equality $G\left(t\right) = t^m F\left(\dfrac{u}{t}\right)$,
but this is not completely justified (we cannot arbitrarily
substitute $u$ for $t$ in an equality of Laurent polynomials unless we
know that $u$ is invertible\footnote{since the universal property of the Laurent
polynomial ring $A\left[  t,t^{-1}\right]  $ is only stated for invertible
elements}). But we can argue as follows instead:

The polynomial $F\left(t\right)$ has degree $\leq m$. Hence,
we can write it in the form $F\left(t\right)=a_{0}+a_{1}t+a_{2}%
t^{2}+\cdots+a_{m}t^{m}$ for some $a_{0},a_{1},\ldots,a_{m}\in A$.
Consider these $a_{0},a_{1},\ldots,a_{m}$. Thus, $F\left(  1\right)
=a_{0}+a_{1}1+a_{2}1^{2}+\cdots+a_{m}1^{m}=a_{0}+a_{1}+a_{2}+\cdots+a_{m}$ and
$F\left(  \dfrac{u}{t}\right)  =a_{0}+a_{1}\dfrac{u}{t}%
+a_{2}\left(  \dfrac{u}{t}\right)  ^{2}+\cdots+a_{m}\left(  \dfrac{u%
}{t}\right)  ^{m}$. The definition of $G\left(t\right)$ yields
\begin{align*}
G\left(  t\right)   &  =t^m\underbrace{F\left(  \dfrac{u}%
{t}\right)  }_{=a_{0}+a_{1}\dfrac{u}{t}+a_{2}\left(  \dfrac{u}%
{t}\right)  ^{2}+\cdots+a_{m}\left(  \dfrac{u}{t}\right)  ^{m}}%
=t^{m}\left(  a_{0}+a_{1}\dfrac{u}{t}+a_{2}\left(  \dfrac{u}%
{t}\right)  ^{2}+\cdots+a_{m}\left(  \dfrac{u}{t}\right)  ^{m}\right) \\
&  =a_{0}t^{m}+a_{1}ut^{m-1}+a_{2}u^{2}t^{m-2}+\cdots+a_{m}u%
^{m}t^{0}.
\end{align*}
Substituting $u$ for $t$ in this equality of polynomials, we obtain%
\begin{align*}
G\left(  u\right)   &  =a_{0}u^{m}+a_{1}uu^{m-1}%
+a_{2}u^{2}u^{m-2}+\cdots+a_{m}u^{m}u^{0}\\
&  =a_{0}u^{m}+a_{1}u^{m}+a_{2}u^{m}+\cdots+a_{m}u^{m}%
=u^{m}\underbrace{\left(  a_{0}+a_{1}+a_{2}+\cdots+a_{m}\right)  }%
_{=F\left(  1\right)  }=u^{m}F\left(  1\right)  ,
\end{align*}
and thus Lemma \ref{lem.laurent-sub} is proven.
\end{proof}

\begin{proof}
[Proof of Theorem \ref{thm.1}.]Let $b_{1},b_{2},\ldots,b_{n}\in L$
be such that $\sum_{i=1}^{n}b_{i}g_{i}=0$. We must show that $b_{i}=0$ for all
$i$.

Assume the contrary. Thus, there exists some $i$ such that $b_{i}\neq0$. We
WLOG assume that \textbf{each} $i$ satisfies $b_{i}\neq0$, since otherwise we
can just drop the violating $b_{i}$ and the corresponding $g_{i}$ and reduce
the problem to a smaller value of $n$.

For each $i\in\left\{  1,2,\ldots,n\right\}  $, the element $b_{i}\in B$ is
id-unipotent (since it belongs to $L$) and thus has a degree-upper bound.
We let $m_{i}$ be the \textbf{smallest} degree-upper bound of $b_{i}$. This is
well-defined, because Lemma \ref{lem.dub.nonneg} (applied to $b=b_{i}$) shows
that every degree-upper bound of $b_{i}$ must be $\in \NN$.
Thus, each $m_i$ belongs to $\NN$.

For each $i\in\left\{  1,2,\ldots,n\right\}  $, Lemma \ref{lem.2}
(applied to $b = b_i$ and $m = m_i$) shows
that in the ring $B\left[  \left[  t\right]  \right]  $, we have
\begin{align}
& \left(  1-t\right)  ^{m_{i}+1}\sum_{k\in \NN}\id^{\oast k}
\left(  b_{i}\right)  t^{k}
\nonumber\\
&=\sum_{k=0}^{m_{i}} \left(-1\right)^k \left(
\eta\epsilon-\id\right)  ^{\oast k}\left(  b_{i}\right)
t^{k}\left(  1-t\right)  ^{m_{i}-k}
\label{pf.thm.1.0}
\end{align}
(since $m_i$ is a degree-upper bound of $b_i$).
The right-hand side of this equality is a polynomial in $t$ of degree $\leq
m_{i}$. Denote this polynomial by $Q_{i}\left(  t\right)  $. Hence,
(\ref{pf.thm.1.0}) becomes%
\[
\left(  1-t\right)  ^{m_{i}+1}\sum_{k\in \NN}\id^{\oast k}
\left(  b_{i}\right)  t^{k}=Q_{i}\left(  t\right)  ,
\]
so that%
\begin{equation}
\sum_{k\in \NN}\id^{\oast k}\left(  b_{i}\right)
t^{k}=\dfrac{Q_{i}\left(  t\right)  }{\left(  1-t\right)  ^{m_{i}+1}}.
\label{pf.thm.1.1}%
\end{equation}

Now, let $k\in \NN$. Lemma \ref{lem.BtoA.preserveidk} yields that the
map $\id^{\oast k}$ is a $\kk$-algebra
homomorphism. Thus,%
\[
\id^{\oast k}\left(  \sum_{i=1}^{n}b_{i}g_{i}\right)
=\sum_{i=1}^{n}\id^{\oast k}\left(  b_{i}\right)
\underbrace{\id^{\oast k}\left(  g_{i}\right)
}_{\substack{=g_{i}^{k}\\\text{(by Lemma \ref{lem.id*k.grouplike}%
,}\\\text{since }g_{i}\text{ is grouplike)}}}
=\sum_{i=1}^{n}\id^{\oast k}\left(  b_{i}\right)  g_{i}^{k}.
\]
Hence,%
\begin{equation}
\sum_{i=1}^{n}\id^{\oast k}\left(  b_{i}\right)
g_{i}^{k}=\id^{\oast k}\left(  \underbrace{\sum
_{i=1}^{n}b_{i}g_{i}}_{=0}\right)  =0. \label{pf.thm.1.2}%
\end{equation}

Forget that we fixed $k$. We thus have proved \eqref{pf.thm.1.2} for each
$k\in \NN$. Now, in the commutative ring $B\left[  \left[  t\right]
\right]  $, we have%
\begin{align*}
\sum_{i=1}^{n}\underbrace{\dfrac{Q_{i}\left(  tg_{i}\right)  }{\left(
1-tg_{i}\right)  ^{m_{i}+1}}}_{\substack{=\sum_{k\in \NN}%
\id^{\oast k}\left(  b_{i}\right)  \left(
tg_{i}\right)  ^{k}\\\text{(by substituting }tg_{i}\\\text{for }t\text{ in
(\ref{pf.thm.1.1}))}}}  &  =\sum_{i=1}^{n}\sum_{k\in \NN}%
\id^{\oast k}\left(  b_{i}\right)  \underbrace{\left(
tg_{i}\right)  ^{k}}_{=t^{k}g_{i}^{k}}=\sum_{i=1}^{n}\sum_{k\in \NN%
}\id^{\oast k}\left(  b_{i}\right)  t^{k}g_{i}^{k}\\
&  =\sum_{k\in \NN}\underbrace{\sum_{i=1}^{n}\id^{\oast k}
\left(  b_{i}\right)  g_{i}^{k}}_{\substack{=0\\\text{(by
(\ref{pf.thm.1.2}))}}}t^{k}=0.
\end{align*}
Multiplying this equality by $\prod_{j=1}^{n}\left(  1-tg_{j}\right)
^{m_{j}+1}$, we obtain
\[
\sum_{i=1}^{n}Q_{i}\left(  tg_{i}\right)  \prod_{j\neq i}\left(
1-tg_{j}\right)  ^{m_{j}+1}=0.
\]
This is an equality in the polynomial ring $B\left[  t\right]  $. Hence, we
can apply the $B$-algebra homomorphism%
\[
B\left[  t\right]  \rightarrow B\left[  t,t^{-1}\right]
,\ \ \ \ \ \ \ \ \ \ t\mapsto t^{-1}%
\]
(into the Laurent polynomial ring $B\left[  t,t^{-1}\right]$)
to this equality. We obtain
\[
\sum_{i=1}^{n}Q_{i}\left(  \dfrac{g_{i}}{t}\right)  \prod_{j\neq i}\left(
1-\dfrac{g_{j}}{t}\right)  ^{m_{j}+1}=0.
\]
Multiplying this equality with $t^{m_{1}+m_{2}+\cdots+m_{n}+n-1}$, we
transform this equality into
\begin{equation}
\sum_{i=1}^{n}t^{m_{i}}Q_{i}\left(  \dfrac{g_{i}}{t}\right)  \prod_{j\neq
i}\left(  t-g_{j}\right)  ^{m_{j}+1}=0 \label{pf.thm.1.4a}%
\end{equation}
(an equality in the Laurent polynomial ring $B\left[  t,t^{-1}\right]  $).

For each $i\in\left\{  1,2,\ldots,n\right\}  $, we set
\[
R_{i}\left(  t\right)
=t^{m_{i}}Q_{i}\left(  \dfrac{g_{i}}{t}\right)  \in B\left[  t,t^{-1}\right]
.
\]
This $R_{i}\left(  t\right)  $ is actually a polynomial in $B\left[
t\right]  $ %of degree $\leq m_{i}$
(since $Q_{i}$ is a polynomial of degree
$\leq m_{i}$). Using the definition of $R_{i}\left(  t\right)  $, the equality
(\ref{pf.thm.1.4a}) rewrites as
\begin{equation}
\sum_{i=1}^{n}R_{i}\left(  t\right)  \prod_{j\neq i}\left(  t-g_{j}\right)
^{m_{j}+1}=0.
\label{pf.thm.1.5}
\end{equation}

Now fix $h\in\left\{  1,2,\ldots,n\right\}  $. We have
$m_h \in \NN$ (since each $m_i$ belongs
to $\NN$), so that $m_h + 1$ is a positive integer. Therefore,
$0^{m_h+1} = 0$. In other words,
$\left(g_{h}-g_{h}\right)^{m_h+1}=0$.

The equality (\ref{pf.thm.1.5})
is an equality in the polynomial ring $B\left[  t\right]  $ (since
each $m_i$ belongs to $\NN$, and since each $R_i \left( t \right)$ is a
polynomial in $B\left[t\right]$), so we can substitute $g_{h}$ for $t$ in it.
We thus obtain
\begin{equation}
\sum_{i=1}^{n}R_{i}\left(  g_{h}\right)  \prod_{j\neq i}\left(  g_{h}%
-g_{j}\right)  ^{m_{j}+1}=0.
\label{pf.thm.1.6}
\end{equation}
But all addends on the left-hand side of this equality
are $0$ except for the addend with $i=h$
(since the product $\prod_{j\neq i}\left(  g_{h}-g_{j}\right)  ^{m_{j}+1}$
contains the factor $\left(g_{h}-g_{h}\right)^{m_h+1}=0$
unless $i=h$).
Thus, the equality
\eqref{pf.thm.1.6} simplifies to
\[
R_{h}\left(  g_{h}\right)  \prod_{j\neq h}\left(  g_{h}-g_{j}\right)
^{m_{j}+1}=0.
\]
Since all the factors $g_{h}-g_{j}$ (with $j\neq h$) are regular elements of
$B$ (by Assumption 1),
we can cancel $\prod_{j\neq h}\left(  g_{h}-g_{j}\right)  ^{m_{j}+1}$
from this equality, and obtain $R_{h}\left(  g_{h}\right)  =0$.

But recall that $Q_{h}\left(t\right) \in B\left[t\right]$ is a polynomial
of degree $\leq m_h$, and we have
$R_{h}\left(  t\right)  =t^{m_{h}}Q_{h}\left(  \dfrac{g_{h}%
}{t}\right)  $ (by the definition of $R_{h}$). Hence,
\eqref{eq.lem.laurent-sub.Gu=} (applied to $A = B$, $m = m_h$, $u = g_h$,
$F\left(t\right) = Q_h \left(t\right)$ and
$G\left(t\right) = R_h \left(t\right)$) yields
$R_{h}\left(  g_{h}\right)  = g_{h}^{m_{h}}Q_{h}\left(  1\right)  $.
Hence, $g_{h}^{m_{h}}Q_{h}\left(  1\right)  =R_{h}\left(  g_{h}\right)
=0$, so that $Q_{h}\left(  1\right)  =0$ (since Assumption 2 shows that
$g_{h}\in B$ is regular).

But the definition of $Q_{h}$ yields%
\[
Q_{h}\left(  t\right)  =\sum_{k=0}^{m_{h}} \left(-1\right)^k
\left(  \eta\epsilon
-\id\right)  ^{\oast k}\left(  b_{h}\right)  t^{k}\left(
1-t\right)  ^{m_{h}-k}.
\]
Substituting $1$ for $t$ in this equality, we find%
\begin{align*}
Q_{h}\left(  1\right)
&  =\sum_{k=0}^{m_{h}}
\left(-1\right)^k 
\left(  \eta\epsilon
-\id\right)  ^{\oast k}\left(  b_{h}\right)
\underbrace{1^{k}}_{=1}
\underbrace{\left(  1-1\right)  ^{m_{h}-k}}_{= 0^{m_{h}-k}=%
\begin{cases}
1, & \text{if }k=m_{h};\\
0, & \text{if }k\neq m_{h}%
\end{cases}
}\\
&= \sum_{k=0}^{m_{h}}
\left(-1\right)^k 
\left(  \eta\epsilon
-\id\right)  ^{\oast k}\left(  b_{h}\right)
\begin{cases}
1, & \text{if }k=m_{h};\\
0, & \text{if }k\neq m_{h}%
\end{cases}
\\
&  = \left(-1\right)^{m_h}
\left(  \eta\epsilon-\id\right)  ^{\oast m_{h}}\left(
b_{h}\right)  .
\end{align*}
Hence, $\left(-1\right)^{m_h} \left(  \eta\epsilon-\id\right)  ^{\oast m_{h}%
}\left(  b_{h}\right)  =Q_{h}\left(  1\right)  =0$.
Therefore, $\left(  \eta\epsilon-\id\right)  ^{\oast m_{h}}
\left(  b_{h}\right)  =0$.

Recall that $m_h$ is the smallest degree-upper bound of $b_h$
(since this is how $m_h$ was defined).
Thus, every nonnegative integer $n > m_h$ satisfies the equality
$\left(  \eta\epsilon-\id\right)  ^{\oast n}\left(  b_h\right)
=0$ (by the definition of a degree-upper bound).
Since $n = m_h$ also satisfies this equality
(because we just showed that
$\left(  \eta\epsilon-\id\right)  ^{\oast m_{h}}
\left(  b_{h}\right)  =0$),
we thus conclude that every nonnegative integer $n > m_h - 1$
satisfies $\left(  \eta\epsilon-\id\right)  ^{\oast n}
\left(  b_{h}\right)  =0$.
In other words, $m_h - 1$ is a degree-upper bound of $b_h$
(by the definition of a degree-upper bound).
Thus, $m_{h}$ is not the
\textbf{smallest} degree-upper bound of $b_{h}$ (since $m_{h}-1$
is smaller). This contradicts our definition of $m_{h}$. This
contradiction completes our proof of Theorem \ref{thm.1}.
\end{proof}

\subsection{\label{subsect.id-subalg}Appendix: The id-unipotents form a subalgebra}

In this subsection, we shall show that the id-unipotent elements in a
commutative $\kk$-bialgebra form a $\kk$-subalgebra. The proof
will rely on a sequence of lemmas. We begin with two identities for binomial coefficients:

\begin{lemma}
\label{lem.binom-altsum}Let $N\in \NN$. Let $\left(  a_{0},a_{1}%
,\ldots,a_{N}\right)  $ and $\left(  b_{0},b_{1},\ldots,b_{N}\right)  $ be two
$\left(  N+1\right)  $-tuples of rational numbers. Assume that
\[
b_{n}=\sum_{i=0}^{n}\left(  -1\right)  ^{i}\dbinom{n}{i}a_{i}
\ \ \ \ \ \ \ \ \ \ \text{for each }n\in\left\{  0,1,\ldots,N\right\}  .
\]
Then,
\[
a_{n}=\sum_{i=0}^{n}\left(  -1\right)  ^{i}\dbinom{n}{i}b_{i}
\ \ \ \ \ \ \ \ \ \ \text{for each }n\in\left\{  0,1,\ldots,N\right\}  .
\]

\end{lemma}

\begin{lemma}
\label{lem.binom-prod}Let $a,b,m\in \NN$. Then,
\[
\dbinom{m}{a}\dbinom{m}{b}
= \sum_{i=a}^{a+b}\dbinom{i}{a}\dbinom{a}{a+b-i}\dbinom{m}{i}.
\]

\end{lemma}

Both of these lemmas are not hard to prove; they are also easily found in the
literature (e.g., Lemma \ref{lem.binom-altsum} is \cite[Proposition
7.26]{detnotes}, while Lemma \ref{lem.binom-prod} is \cite[Proposition
3.37]{detnotes}). Note that we have
\begin{equation}
\dbinom{n}{i}=0 \label{eq.binom.ni=0}
\end{equation}
for any $n,i\in \NN$ satisfying $i>n$. Thus, the nonzero addends on the
right-hand side in Lemma \ref{lem.binom-prod} don't start until $i=\max
\left\{  a,b\right\}  $.

The following modified version of Lemma \ref{lem.binom-altsum} will be useful
to us:\footnote{An ``abelian group, written
additively'' means an abelian group whose binary operation is
denoted by $+$ and whose neutral element is denoted by $0$.}

\begin{lemma}
\label{lem.binom-altsum2}Let $A$ be an abelian group, written additively. Let
$\left(  a_{0},a_{1},a_{2},\ldots\right)  $ and $\left(  b_{0},b_{1}%
,b_{2},\ldots\right)  $ be two sequences of elements of $A$. Assume that
\[
b_{n}=\sum_{i=0}^{n}\left(  -1\right)  ^{i}\dbinom{n}{i}a_{i}
\ \ \ \ \ \ \ \ \ \ \text{for each }n\in \NN.
\]
Then,
\[
a_{n}=\sum_{i=0}^{n}\left(  -1\right)  ^{i}\dbinom{n}{i}b_{i}
\ \ \ \ \ \ \ \ \ \ \text{for each }n\in \NN.
\]

\end{lemma}

\begin{proof}
[First proof of Lemma \ref{lem.binom-altsum2}.]The proof of Lemma
\ref{lem.binom-altsum} that was given in \cite[proof of Proposition
7.26]{detnotes} can be immediately reused as a proof of Lemma
\ref{lem.binom-altsum2} (after removing all inequalities of the form
``$n\leq N$'', and after replacing every
appearance of ``$\left\{  0,1,\ldots,N\right\}  $'' by
``$\NN$'').
\end{proof}

\begin{proof}[Second proof of Lemma~\ref{lem.binom-altsum2} (sketched).]
We will interpret binomial transforms in terms of sequence transformations (see \cite{GoF04} for motivations and details). 
A \emph{row-finite matrix} will mean a family
\[
\left(M[n,k]\right)_{n,k\geq 0}
\qquad \text{with } M[n,k] \in \kk
\]
(i.e., a matrix of type $(\N,\N)$ with elements in $\kk$, in the terminology of \cite[Chapter II, \S 10]{Bourbaki-Alg1}) with the property that for every fixed row index $n$, the sequence 
$\left(M[n,k]\right)_{k\geq 0}$ has finite support.
We let $\mathcal{L}(\kk^{\N})$ be the algebra of row-finite matrices;
its product is given by the 
usual formula\footnote{This is just the definition in \cite[Chapter II, \S 10, (3)]{Bourbaki-Alg1}, extended a bit.}
\begin{equation}
\tup{M_1M_2}[i,j]=\sum_{k\in \N}M_1[i,k]M_2[k,j]\ .
\end{equation}
Any row-finite matrix $M \in \mathcal{L}(\kk^{\N})$ canonically acts on $A^\N$ for any $\kk$-module $A$:
Namely, if $a=(a_n)_{n\geq 0} \in A^\N$ is a sequence of elements of $A$, then $Ma$ is defined to be the sequence $(b_n)_{n\geq 0} \in A^\N$ with
\begin{equation}
b_n=\sum_{k\geq 0} M[n,k]\,a_k \qquad \text{for all $n \in \N$.}
\end{equation}
This action gives rise to a $\kk$-algebra homomorphism
\begin{equation}\label{Phi_1}
\Phi_1:\ \mathcal{L}(\kk^{\N})\to \End(A^\N) ,
\end{equation}
which is injective in the case when $A=\kk$
(but neither injective nor surjective in the general case).
Applying this to $\kk = \Q$ and $A = \Q$, we thus have an injective $\Q$-algebra homomorphism
$\Phi_1 : \ \mathcal{L}(\Q^{\N}) \to \End(\Q^\N)$.
We also have a $\Q$-vector space isomorphism
\begin{align*}
\Q^\N & \overset{\cong}{\to} \Q[[z]], \\
(a_n)_{n\in \N} &\mapsto \sum_{n\ge0}a_n\frac{z^n}{n!}
\end{align*}
(where $z$ is a formal variable),
thus a $\Q$-algebra isomorphism $\End(\Q^\N) \to \End(\Q[[z]])$.
Composing the latter with $\Phi_1$,
we get an injective $\Q$-algebra homomorphism
$\Phi_2:\ \mathcal{L}(\Q^{\N})\to \End(\Q[[z]])$.
Explicitly, for any power series $f=f(z)=\sum_{n\ge0}a_n\frac{z^n}{n!} \in \Q[[z]]$
and any row-finite matrix $M\in \mathcal{L}(\Q^{\N})$, we have  
\begin{equation}
\Phi_2(M)(f)=\sum_{n\ge0}b_n\frac{z^n}{n!}=\sum_{n\ge0}\sum_{k\geq 0} M[n,k]\,a_k\frac{z^n}{n!}\ .
\end{equation}

Now, let $M_1 \in \mathcal{L}(\Z^{\N})\subseteq \mathcal{L}(\Q^{\N})$ be the
row-finite matrix whose entries are
\begin{equation}\label{binom-altsum2}
M_1[n,i]:=\left(-1\right)^{i}\dbinom{n}{i} .
\end{equation} 
Then, the assumption in Lemma~\ref{lem.binom-altsum2} is saying that
$b = M_1 a$, where $a, b \in A^\NN$ are given by
$a = \left(a_n\right)_{n \geq 0}$ and $b = \left(b_n\right)_{n \geq 0}$.
Likewise, the claim of Lemma~\ref{lem.binom-altsum2} is saying that
$a = M_1 b$.
Hence, in order to prove Lemma~\ref{lem.binom-altsum2}, it suffices
to show that $M_1^2 a = a$. Thus, it suffices to show that
$M_1^2 = I_{\mathcal{L}(\Z^{\N})}$ (the identity
matrix in $\mathcal{L}(\Z^{\N})$).

But this can be done via the injective $\Z$-algebra homomorphism
$\Phi_2 :\ \mathcal{L}(\Q^{\N})\to \End(\Q[[z]])$.
Indeed, a direct calculation shows that any $f \in \Q[[z]]$
satisfies
\[
\Phi_2(M_1)(f)(z)=f(-z)e^z ,
\]
and hence
\[
\Phi_2(M_1^2)(f)(z)=\Phi_2(M_1)(f(-z)e^z)=
(f(-(-z))e^{-z})e^z = f(z) .
\]
This shows that $\Phi_2(M_1^2) = \id = \Phi_2(I_{\mathcal{L}(\Q^{\N})})$.
Since $\Phi_2$ is injective, this entails
$M_1^2=I_{\mathcal{L}(\Q^{\N})}=I_{\mathcal{L}(\Z^{\N})}$ and proves Lemma~\ref{lem.binom-altsum2}
(as every abelian group is a $\Z$-module).
\end{proof}

Our next lemma is a classical fact about finite differences:

\begin{lemma}
\label{lem.findiff.1}Let $A$ be an abelian group, written additively. Let
$\left(  a_{0},a_{1},a_{2},\ldots\right)  \in A^{\NN}$ be a sequence of
elements of $A$. Let $m\geq-1$ be an integer. Then, the following two
statements are equivalent:

\begin{enumerate}
\item We have $\sum_{i=0}^{n}\left(  -1\right)  ^{i}\dbinom{n}{i}a_{i}=0$ for
every integer $n>m$.

\item There exist $m+1$ elements $c_{0},c_{1},\ldots,c_{m}$ of $A$ such that
every $n\in \NN$ satisfies $a_{n}=\sum_{i=0}^{m}\dbinom{n}{i}c_{i}$.
\end{enumerate}
\end{lemma}

The sequences $\left(  a_{0},a_{1},a_{2},\ldots\right)  $ satisfying these two
equivalent statements can be regarded as a generalization of polynomial
sequences (i.e., sequences whose $n$-th entry is given by evaluating a fixed
integer-valued polynomial at $n$).

\begin{proof}
[Proof of Lemma \ref{lem.findiff.1}.]We must prove the equivalence of
Statement 1 and Statement 2. We shall do so by proving the implications 1
$\Longrightarrow$ 2 and 2 $\Longrightarrow$ 1 separately:

\textit{Proof of the implication 1 $\Longrightarrow$ 2:} Let us first
show that Statement 1 implies Statement 2.

Indeed, assume that Statement 1 holds. In other words, we have
\begin{equation}
\sum_{i=0}^{n}\left(  -1\right)  ^{i}\dbinom{n}{i}a_{i}=0
\label{pf.lem.findiff.1.1to2.ass}
\end{equation}
for every integer $n>m$.

We shall now show that Statement 2 holds.

Indeed, let us set
\begin{equation}
b_{n}=\sum_{i=0}^{n}\left(  -1\right)  ^{i}\dbinom{n}{i}a_{i}
\label{pf.lem.findiff.1.1to2.bn=}
\end{equation}
for each $n\in \NN$. Then, for every integer $n>m$, we have
\[
b_{n}=\sum_{i=0}^{n}\left(  -1\right)  ^{i}\dbinom{n}{i}a_{i}=0
\]
(by \eqref{pf.lem.findiff.1.1to2.ass}). Renaming $n$ as $i$ in this statement,
we obtain the following: For every integer $i>m$, we have
\begin{equation}
b_{i}=0. \label{pf.lem.findiff.1.1to2.assb}
\end{equation}

Furthermore, recall that $b_{n}=\sum_{i=0}^{n}\left(  -1\right)  ^{i}
\dbinom{n}{i}a_{i}$ for each $n\in \NN$. Thus, Lemma
\ref{lem.binom-altsum2} shows that we have
\begin{equation}
a_{n}=\sum_{i=0}^{n}\left(  -1\right)  ^{i}\dbinom{n}{i}b_{i}
\label{pf.lem.findiff.1.1to2.an=1}
\end{equation}
for each $n\in \NN$.

Now, let $n\in \NN$. Let $N=\max\left\{  n,m\right\}  $; then, $N\geq n$
and $N\geq m$. Thus, $0\leq n\leq N$ and $-1\leq m\leq N$. Now, comparing
\begin{align*}
\sum_{i=0}^{N}\dbinom{n}{i}\left(  -1\right)  ^{i}b_{i}  &  =\sum_{i=0}%
^{n}\underbrace{\dbinom{n}{i}\left(  -1\right)  ^{i}}_{=\left(  -1\right)
^{i}\dbinom{n}{i}}b_{i}+\sum_{i=n+1}^{N}\underbrace{\dbinom{n}{i}%
}_{\substack{=0\\\text{(by \eqref{eq.binom.ni=0},}\\\text{since }i\geq
n+1>n\text{)}}}\left(  -1\right)  ^{i}b_{i}\\
&  \ \ \ \ \ \ \ \ \ \ \left(  \text{here, we have split the sum at
}i=n\text{, since }0\leq n\leq N\right) \\
&  =\sum_{i=0}^{n}\left(  -1\right)  ^{i}\dbinom{n}{i}b_{i}+\underbrace{\sum
_{i=n+1}^{N}0\left(  -1\right)  ^{i}b_{i}}_{=0}=\sum_{i=0}^{n}\left(
-1\right)  ^{i}\dbinom{n}{i}b_{i}\\
&  =a_{n}\ \ \ \ \ \ \ \ \ \ \left(  \text{by
\eqref{pf.lem.findiff.1.1to2.an=1}}\right)
\end{align*}
with
\begin{align*}
\sum_{i=0}^{N}\dbinom{n}{i}\left(  -1\right)  ^{i}b_{i}  &  =\sum_{i=0}%
^{m}\dbinom{n}{i}\left(  -1\right)  ^{i}b_{i}+\sum_{i=m+1}^{N}\dbinom{n}%
{i}\left(  -1\right)  ^{i}\underbrace{b_{i}}_{\substack{=0\\\text{(by
\eqref{pf.lem.findiff.1.1to2.assb},}\\\text{since }i\geq m+1>m\text{)}}}\\
&  \ \ \ \ \ \ \ \ \ \ \left(  \text{here, we have split the sum at
}i=m\text{, since }-1\leq m\leq N\right) \\
&  =\sum_{i=0}^{m}\dbinom{n}{i}\left(  -1\right)  ^{i}b_{i}+\underbrace{\sum
_{i=m+1}^{N}\dbinom{n}{i}\left(  -1\right)  ^{i}0}_{=0}=\sum_{i=0}^{m}%
\dbinom{n}{i}\left(  -1\right)  ^{i}b_{i},
\end{align*}
we obtain
\[
a_{n}=\sum_{i=0}^{m}\dbinom{n}{i}\left(  -1\right)  ^{i}b_{i}.
\]

Forget that we fixed $n$. We thus have proved that every $n\in \NN$
satisfies $a_{n}=\sum_{i=0}^{m}\dbinom{n}{i}\left(  -1\right)  ^{i}b_{i}$.
Thus, there exist $m+1$ elements $c_{0},c_{1},\ldots,c_{m}$ of $A$ such that
every $n\in \NN$ satisfies $a_{n}=\sum_{i=0}^{m}\dbinom{n}{i}c_{i}$
(namely, these $m+1$ elements are given by $c_{i}=\left(  -1\right)  ^{i}%
b_{i}$). In other words, Statement 2 holds.

We thus have proved the implication 1 $\Longrightarrow$ 2.

\textit{Proof of the implication 2 $\Longrightarrow$ 1:} Let us now
show that Statement 2 implies Statement 1.

Indeed, assume that Statement 2 holds. That is, there exist $m+1$ elements
$c_{0},c_{1},\ldots,c_{m}$ of $A$ such that every $n\in \NN$ satisfies
\begin{equation}
a_{n}=\sum_{i=0}^{m}\dbinom{n}{i}c_{i}. \label{pf.lem.findiff.1.2to1.ass}
\end{equation}
Consider these $c_{0},c_{1},\ldots,c_{m}$.

We must prove that Statement 1 holds. In other words, we must prove that we
have $\sum_{i=0}^{n}\left(  -1\right)  ^{i}\dbinom{n}{i}a_{i}=0$ for every
integer $n>m$.

Extend the $\left(  m+1\right)  $-tuple $\left(  c_{0},c_{1},\ldots
,c_{m}\right)  \in A^{m+1}$ to an infinite sequence $\left(  c_{0},c_{1}%
,c_{2},\ldots\right)  \in A^{\NN}$ by setting
\begin{equation}
\left(  c_{i}=0\ \ \ \ \ \ \ \ \ \ \text{for all integers }i>m\right)  .
\label{pf.lem.findiff.1.2to1.ci0}
\end{equation}
Furthermore, define a sequence $\left(  d_{0},d_{1},d_{2},\ldots\right)  \in
A^{\NN}$ by setting
\begin{equation}
\left(  d_{i}=\left(  -1\right)  ^{i}c_{i}\ \ \ \ \ \ \ \ \ \ \text{for all
}i\in \NN\right)  . \label{pf.lem.findiff.1.2to1.di}
\end{equation}
Then, each $i\in \NN$ satisfies
\begin{equation}
\left(  -1\right)  ^{i}\underbrace{d_{i}}_{\substack{=\left(  -1\right)
^{i}c_{i}\\\text{(by \eqref{pf.lem.findiff.1.2to1.di})}}}=\underbrace{\left(
-1\right)  ^{i}\left(  -1\right)  ^{i}}_{=\left(  -1\right)  ^{2i}=1}%
c_{i}=c_{i}. \label{pf.lem.findiff.1.2to1.-1di}
\end{equation}

Let $n\in \NN$. Let $N=\max\left\{  n,m\right\}  $; then, $N\geq n$ and
$N\geq m$. Thus, $0\leq n\leq N$ and $-1\leq m\leq N$. Now, comparing
\begin{align*}
\sum_{i=0}^{N}\dbinom{n}{i}c_{i}  &  =\sum_{i=0}^{n}\dbinom{n}{i}c_{i}%
+\sum_{i=n+1}^{N}\underbrace{\dbinom{n}{i}}_{\substack{=0\\\text{(by
\eqref{eq.binom.ni=0},}\\\text{since }i\geq n+1>n\text{)}}}c_{i}\\
&  \ \ \ \ \ \ \ \ \ \ \left(  \text{here, we have split the sum at
}i=n\text{, since }0\leq n\leq N\right) \\
&  =\sum_{i=0}^{n}\dbinom{n}{i}c_{i}+\underbrace{\sum_{i=n+1}^{N}0c_{i}}%
_{=0}=\sum_{i=0}^{n}\dbinom{n}{i}c_{i}
\end{align*}
with
\begin{align*}
\sum_{i=0}^{N}\dbinom{n}{i}c_{i}  &  =\sum_{i=0}^{m}\dbinom{n}{i}c_{i}%
+\sum_{i=m+1}^{N}\dbinom{n}{i}\underbrace{c_{i}}_{\substack{=0\\\text{(by
\eqref{pf.lem.findiff.1.2to1.ci0},}\\\text{since }i\geq m+1>m\text{)}}}\\
&  \ \ \ \ \ \ \ \ \ \ \left(  \text{here, we have split the sum at
}i=m\text{, since }-1\leq m\leq N\right) \\
&  =\sum_{i=0}^{m}\dbinom{n}{i}c_{i}+\underbrace{\sum_{i=m+1}^{N}\dbinom{n}%
{i}0}_{=0}=\sum_{i=0}^{m}\dbinom{n}{i}c_{i}=a_{n}\ \ \ \ \ \ \ \ \ \ \left(
\text{by \eqref{pf.lem.findiff.1.2to1.ass}}\right)  ,
\end{align*}
we obtain
\[
a_{n}=\sum_{i=0}^{n}\dbinom{n}{i}\underbrace{c_{i}}_{\substack{=\left(
-1\right)  ^{i}d_{i}\\\text{(by \eqref{pf.lem.findiff.1.2to1.-1di})}}%
}=\sum_{i=0}^{n}\dbinom{n}{i}\left(  -1\right)  ^{i}d_{i}=\sum_{i=0}%
^{n}\left(  -1\right)  ^{i}\dbinom{n}{i}d_{i}.
\]

Now, forget that we fixed $n$. We thus have shown that $a_{n}=\sum_{i=0}%
^{n}\left(  -1\right)  ^{i}\dbinom{n}{i}d_{i}$ for each $n\in \NN$.
Thus, Lemma \ref{lem.binom-altsum2} (applied to $d_{n}$ and $a_{n}$ instead of
$a_{n}$ and $b_{n}$) shows that
\begin{equation}
d_{n}=\sum_{i=0}^{n}\left(  -1\right)  ^{i}\dbinom{n}{i}a_{i}%
\ \ \ \ \ \ \ \ \ \ \text{for each }n\in \NN.
\label{pf.lem.findiff.1.2to1.d-through-a}
\end{equation}

Now, let $n>m$ be an integer. Then, \eqref{pf.lem.findiff.1.2to1.ci0} (applied
to $i=n$) yields $c_{n}=0$. But \eqref{pf.lem.findiff.1.2to1.di} (applied to
$i=n$) yields $d_{n}=\left(  -1\right)  ^{n}\underbrace{c_{n}}_{=0}=0$. But
\eqref{pf.lem.findiff.1.2to1.d-through-a} yields $d_{n}=\sum_{i=0}^{n}\left(
-1\right)  ^{i}\dbinom{n}{i}a_{i}$. Comparing the latter two equalities, we
obtain $\sum_{i=0}^{n}\left(  -1\right)  ^{i}\dbinom{n}{i}a_{i}=0$.

Forget that we fixed $n$. We thus have shown that $\sum_{i=0}^{n}\left(
-1\right)  ^{i}\dbinom{n}{i}a_{i}=0$ for every integer $n>m$. In other words,
Statement 1 holds.

We thus have proved the implication 2 $\Longrightarrow$ 1.

Having now shown both implications 1 $\Longrightarrow$ 2 and 2
$\Longrightarrow$ 1, we conclude that Statement 1 and Statement 2 are
equivalent. This proves Lemma \ref{lem.findiff.1}.
\end{proof}

% Next, let $M_2 \in \mathcal{L}(\Z^{\N})\subseteq \mathcal{L}(\Q^{\N})$ be the
% row-finite matrix whose entries are
% \begin{equation}\label{findiff.1}
% M_1[n,i]:=\dbinom{n}{i} .
% \end{equation} 
% To prove Lemma~\ref{lem.findiff.1}, set
% $a = \left(a_n\right)_{n \geq 0} \in A^\N$, and
% let $b = \left(b_n\right)_{n \geq 0} \in A^\N$ be the
% sequence defined by $b_n = \sum_{i=0}^n \tup{-1}^i \dbinom{n}{i} a_i$.
% Then,

% \tcb{
% The second transformation matrix $M_2$ (in Lemma~\ref{lem.findiff.1}) also belongs to 
% $\mathcal{L}(\Z^{\N})\subseteq \mathcal{L}(\Q^{\N})$ and its entries read
% %
% \begin{equation}\label{findiff.1}
% M_2[n,i]:=\dbinom{n}{i}
% \end{equation} 
% a direct calculation shows that $\Phi_2(M_2)(f)(z)=f(z)e^z$, then $\Phi_2(M_1M_2)(f)(z)=f(-z)$ which shows that 
% $M_1M_2$ is diagonal with $M_1M_2[i,i]=(-1)^i$. Returning to transformations without denominators let us call 
% $u=(a_i)_{i\ge0}\in A^\N$ (resp. $v=(c_i)_{i\ge0}\in A^\N$) and set $v=\Phi_1(M_2)(u)$
% ($\Phi_1$ is as in \eqref{Phi_1}), we see that $u$ and $v$ have the same support. This proves Lemma~\ref{lem.findiff.1}.
% }      

\begin{definition}
Let $A$ be an abelian group, written additively. Let $\left(  a_{0}%
,a_{1},a_{2},\ldots\right)  \in A^{\NN}$ be a sequence of elements of
$A$. Let $m\geq-1$ be an integer. We say that the sequence $\left(
a_{0},a_{1},a_{2},\ldots\right)  $ is $m$\emph{-polynomial} if the two
equivalent statements 1 and 2 of Lemma \ref{lem.findiff.1} are satisfied.
\end{definition}

Note that if $A$ is a commutative $\QQ$-algebra, then a sequence
$\left(a_0, a_1, a_2, \ldots\right) \in A^{\NN}$ is $m$-polynomial
if and only if it is the sequence of values
$\left(f\left(0\right), f\left(1\right), f\left(2\right), \ldots\right)$
of some polynomial $f \in A\left[x\right]$ of degree $\leq m$.
This is because the expression $\sum_{i=0}^{m}\dbinom{n}{i}c_{i}$
in Lemma~\ref{lem.findiff.1} is just a generic form for a polynomial
function in $n$ of degree $\leq m$ over a commutative $\QQ$-algebra.
While we shall not use this fact below
(except in Example~\ref{exa.L.not-subcoalg}), it explains the provenance
of the word ``$m$-polynomial''.

\begin{lemma}
\label{lem.findiff.prod}Let $A$ be a ring. Let
$p\geq-1$ and $q\geq-1$ be two integers such that $p+q\geq-1$. Let $\left(
a_{0},a_{1},a_{2},\ldots\right)  \in A^{\NN}$ be a $p$-polynomial
sequence of elements of $A$. Let $\left(  b_{0},b_{1},b_{2},\ldots\right)  \in
A^{\NN}$ be a $q$-polynomial sequence of elements of $A$. Then,
$\left(  a_{0}b_{0},a_{1}b_{1},a_{2}b_{2},\ldots\right)  \in A^{\NN}$
is a $\left(  p+q\right)  $-polynomial sequence of elements of $A$.
\end{lemma}

\begin{proof}
[Proof of Lemma \ref{lem.findiff.prod}.]We have assumed that the sequence
$\left(  a_{0},a_{1},a_{2},\ldots\right)  $ is $p$-polynomial. In other words,
this sequence satisfies the two equivalent statements 1 and 2 of Lemma
\ref{lem.findiff.1} for $m=p$ (by the definition of ``$p$-polynomial'').
Thus, in particular, it satisfies Statement
2 of Lemma \ref{lem.findiff.1} for $m=p$. In other words, there exist $p+1$
elements $c_{0},c_{1},\ldots,c_{p}$ of $A$ such that every $n\in \NN$
satisfies $a_{n}=\sum_{i=0}^{p}\dbinom{n}{i}c_{i}$. Consider these
$c_{0},c_{1},\ldots,c_{p}$, and denote them by $u_{0},u_{1},\ldots,u_{p}$.
Thus, $u_{0},u_{1},\ldots,u_{p}$ are $p+1$ elements of $A$ with the property
that every $n\in \NN$ satisfies
\begin{equation}
a_{n}=\sum_{i=0}^{p}\dbinom{n}{i}u_{i}. \label{pf.lem.findiff.prod.a-ass}
\end{equation}

The same argument (but applied to the sequence $\left(  b_{0},b_{1}%
,b_{2},\ldots\right)  $ and the integer $q$ instead of the sequence $\left(
a_{0},a_{1},a_{2},\ldots\right)  $ and the integer $p$) helps us construct
$q+1$ elements $v_{0},v_{1},\ldots,v_{q}$ of $A$ with the property that every
$n\in \NN$ satisfies
\begin{equation}
b_{n}=\sum_{i=0}^{q}\dbinom{n}{i}v_{i}. \label{pf.lem.findiff.prod.b-ass}
\end{equation}
Consider these $q+1$ elements $v_{0},v_{1},\ldots,v_{q}$.

For each $i\in \NN$, we set
\begin{equation}
w_{i}=\sum_{\substack{\left(  j,k\right)  \in \NN\times \NN%
;\\j\leq p;\ k\leq q;\\j\leq i\leq j+k}}\dbinom{i}{j}\dbinom{j}{j+k-i}%
u_{j}v_{k}. \label{pf.lem.findiff.prod.wi=}
\end{equation}
Thus, if $i \in \NN$ satisfies $i > p+q$, then
\begin{align}
w_{i}  &  = 0 .
\label{pf.lem.findiff.prod.wi=0}
\end{align}

[\textit{Proof of \eqref{pf.lem.findiff.prod.wi=0}:}
Let $i \in \NN$ satisfy $i > p+q$.
Then, there exists no $\left(j, k\right) \in \NN \times \NN$
satisfying $j \leq p$ and $k \leq q$ and $j \leq i \leq j+k$
(since any such $\left(j, k\right)$ would satisfy
$i\leq\underbrace{j}_{\leq p}+\underbrace{k}_{\leq q}\leq p+q$,
which would contradict $i > p+q$).
Hence, the sum on the right-hand side of
\eqref{pf.lem.findiff.prod.wi=}
is empty, and thus equals $0$.
Therefore, \eqref{pf.lem.findiff.prod.wi=} rewrites as
$w_i = 0$.
This proves \eqref{pf.lem.findiff.prod.wi=0}.]
% Note that if $i\in \NN$ satisfies $i>p+q$, then
% \begin{align}
% w_{i}  &  =\sum_{\substack{\left(  j,k\right)  \in \NN\times
 % \NN;\\j\leq p;\ k\leq q;\\j\leq i\leq j+k}}\dbinom{i}{j}\dbinom
% {j}{j+k-i}u_{j}v_{k}=\left(  \text{empty sum}\right) \nonumber\\
% &  \ \ \ \ \ \ \ \ \ \ \left(
% \begin{array}
% [c]{c}%
% \text{since there exists no }\left(  j,k\right)  \in \NN\times
 % \NN\\
% \text{satisfying }j\leq p\text{ and }k\leq q\text{ and }j\leq i\leq j+k\\
% \text{(because any such }\left(  j,k\right)  \text{ would satisfy }
% i\leq\underbrace{j}_{\leq p}+\underbrace{k}_{\leq q}\leq p+q\text{,}\\
% \text{which would contradict }i>p+q\text{)}%
% \end{array}
% \right) \nonumber\\
% &  =0. \label{pf.lem.findiff.prod.wi=0}
% \end{align}

Let $n\in \NN$. Then, \eqref{pf.lem.findiff.prod.a-ass} yields
\[
a_{n}=\sum_{i=0}^{p}\dbinom{n}{i}u_{i}=\sum_{j=0}^{p}\dbinom{n}{j}u_{j}.
\]
Meanwhile, \eqref{pf.lem.findiff.prod.b-ass} yields
\[
b_{n}=\sum_{i=0}^{q}\dbinom{n}{i}v_{i}=\sum_{k=0}^{q}\dbinom{n}{k}v_{k}.
\]
Multiplying these two equalities, we obtain
\begin{align*}
a_{n}b_{n} &  =\left(  \sum_{j=0}^{p}\dbinom{n}{j}u_{j}\right)  \left(
\sum_{k=0}^{q}\dbinom{n}{k}v_{k}\right)  =\underbrace{\sum_{j=0}^{p}\sum
_{k=0}^{q}}_{=\sum_{\substack{\left(  j,k\right)  \in \NN\times
 \NN;\\j\leq p;\ k\leq q}}}\ \ \ \underbrace{\dbinom{n}{j}\dbinom{n}{k}%
}_{\substack{=\sum_{i=j}^{j+k}\dbinom{i}{j}\dbinom{j}{j+k-i}\dbinom{n}%
{i}\\\text{(by Lemma \ref{lem.binom-prod},}\\\text{applied to }m=n\text{,
}a=j\text{ and }b=k\text{)}}}u_{j}v_{k}\\
&  =\sum_{\substack{\left(  j,k\right)  \in \NN\times \NN;\\j\leq
p;\ k\leq q}}\ \ \underbrace{\sum_{i=j}^{j+k}}_{=\sum_{\substack{i\in
 \NN;\\j\leq i\leq j+k}}}\dbinom{i}{j}\dbinom{j}{j+k-i}\dbinom{n}%
{i}u_{j}v_{k}\\
&  =\underbrace{\sum_{\substack{\left(  j,k\right)  \in \NN%
\times \NN;\\j\leq p;\ k\leq q}}\ \ \sum_{\substack{i\in \NN%
;\\j\leq i\leq j+k}}}_{=\sum_{i\in \NN}\ \ \sum_{\substack{\left(
j,k\right)  \in \NN\times \NN;\\j\leq p;\ k\leq q;\\j\leq i\leq
j+k}}}\dbinom{i}{j}\dbinom{j}{j+k-i}\dbinom{n}{i}u_{j}v_{k}\\
&  =\sum_{i\in \NN}\ \ \sum_{\substack{\left(  j,k\right)  \in
 \NN\times \NN;\\j\leq p;\ k\leq q;\\j\leq i\leq j+k}}\dbinom
{i}{j}\dbinom{j}{j+k-i}\dbinom{n}{i}u_{j}v_{k}\\
&  =\sum_{i\in \NN}\dbinom{n}{i}\underbrace{\sum_{\substack{\left(
j,k\right)  \in \NN\times \NN;\\j\leq p;\ k\leq q;\\j\leq i\leq
j+k}}\dbinom{i}{j}\dbinom{j}{j+k-i}u_{j}v_{k}}_{\substack{=w_{i}\\\text{(by
\eqref{pf.lem.findiff.prod.wi=})}}}\\
&  =\sum_{i\in \NN}\dbinom{n}{i}w_{i}=\underbrace{\sum_{\substack{i\in
 \NN;\\i\leq p+q}}}_{=\sum_{i=0}^{p+q}}\dbinom{n}{i}w_{i}+\sum
_{\substack{i\in \NN;\\i>p+q}}\dbinom{n}{i}\underbrace{w_{i}%
}_{\substack{=0\\\text{(by \eqref{pf.lem.findiff.prod.wi=0})}}}\\
&  =\sum_{i=0}^{p+q}\dbinom{n}{i}w_{i}+\underbrace{\sum_{\substack{i\in
 \NN;\\i>p+q}}\dbinom{n}{i}0}_{=0}=\sum_{i=0}^{p+q}\dbinom{n}{i}w_{i}.
\end{align*}

Forget that we fixed $n$. We thus have shown that every $n\in \NN$
satisfies $a_{n}b_{n}=\sum_{i=0}^{p+q}\dbinom{n}{i}w_{i}$. Hence, there exist
$p+q+1$ elements $c_{0},c_{1},\ldots,c_{p+q}$ of $A$ such that every
$n\in \NN$ satisfies $a_{n}b_{n}=\sum_{i=0}^{p+q}\dbinom{n}{i}c_{i}$
(namely, these $p+q+1$ elements are given by $c_{i}=w_{i}$). In other words,
Statement 2 of Lemma \ref{lem.findiff.1} with $m$ and $\left(  a_{0}%
,a_{1},a_{2},\ldots\right)  $ replaced by $p+q$ and $\left(  a_{0}b_{0}%
,a_{1}b_{1},a_{2}b_{2},\ldots\right)  $ is satisfied. Thus, the two equivalent
statements 1 and 2 of Lemma \ref{lem.findiff.1} with $m$ and $\left(
a_{0},a_{1},a_{2},\ldots\right)  $ replaced by $p+q$ and $\left(  a_{0}%
b_{0},a_{1}b_{1},a_{2}b_{2},\ldots\right)  $ are satisfied. In other words, the
sequence $\left(  a_{0}b_{0},a_{1}b_{1},a_{2}b_{2},\ldots\right)  $ is
$\left(  p+q\right)  $-polynomial (by the definition of
``$\left(  p+q\right)  $-polynomial''). This proves Lemma
\ref{lem.findiff.prod}.
\end{proof}

\begin{lemma}
\label{lem.L.polynomiality}Let $B$ be a $\kk$-bialgebra. Let $b\in
B$. Let $m\geq-1$ be an integer.

Then, $m$ is a degree-upper bound of $b$ if and only if the sequence
\[
\left(
\id^{\oast0}\left(  b\right)  ,\id^{\oast1}\left(  b\right)  ,\id^{\oast2}%
\left(  b\right)  ,\ldots\right)
\]
is $m$-polynomial.
\end{lemma}

\begin{proof}
[Proof of Lemma \ref{lem.L.polynomiality}.]The elements $\eta\epsilon$ and
$\id$ of the convolution algebra $\left(  \Hom
\left(  B,B\right)  ,\oast\right)  $ commute (since $\eta\epsilon$ is the
unity of this algebra). Thus, the binomial formula shows that
\[
\left(  \eta\epsilon-\id\right)  ^{\oast n}=\sum_{i=0}%
^{n}\left(  -1\right)  ^{i}\dbinom{n}{i}\id^{\oast i}%
\]
for every $n\in \NN$ (again because $\eta\epsilon$ is the
unity of the convolution algebra). Hence, for every $n\in \NN$,
we have
\begin{equation}
\left(  \eta\epsilon-\id\right)  ^{\oast n}\left(  b\right)
=\sum_{i=0}^{n}\left(  -1\right)  ^{i}\dbinom{n}{i}
\id^{\oast i}\left(  b\right)  . \label{pf.lem.L.polynomiality.2}
\end{equation}

Now, we have the following chain of equivalences:
\begin{align*}
\  &  \left(  m\text{ is a degree-upper bound of }b\right) \\
\Longleftrightarrow\  &  \left(  \left(  \eta\epsilon-\id%
\right)  ^{\oast n}\left(  b\right)  =0\text{ for every integer }n>m\right) \\
&  \qquad\qquad\left(  \text{by the definition of a ``degree-upper bound''}
\right) \\
\Longleftrightarrow\ \  &  \left(  \sum_{i=0}^{n}\left(  -1\right)
^{i}\dbinom{n}{i}\id^{\oast i}\left(  b\right)
=0\text{ for every integer }n>m\right) \\
&  \qquad\qquad\left(  \text{by \eqref{pf.lem.L.polynomiality.2}}\right) \\
\Longleftrightarrow\  &  \left(  \text{Statement 1 of Lemma
\ref{lem.findiff.1} holds for }A=B\text{ and }a_{i}=
\id^{\oast i}\left(  b\right)  \right) \\
\Longleftrightarrow\  &  \left(  \text{the sequence }\left(
\id^{\oast0}\left(  b\right)  ,\id^{\oast1}\left(  b\right)  ,\id^{\oast2}%
\left(  b\right)  ,\ldots\right)  \text{ is }m\text{-polynomial}\right) \\
&  \qquad\qquad\left(
\text{by the definition of ``}m\text{-polynomial''}\right)  .
\end{align*}
This proves Lemma \ref{lem.L.polynomiality}.
\end{proof}

\begin{proposition}
\label{prop.L.pq}Let $B$ be a commutative $\kk$-bialgebra. Let $b,c\in B$. Let
$p,q\geq-1$ be two integers with $p+q\geq-1$. Assume that $p$ is a
degree-upper bound of $b$. Assume that $q$ is a degree-upper bound of $c$.
Then, $p+q$ is a degree-upper bound of $bc$.
\end{proposition}

\begin{proof}
[Proof of Proposition \ref{prop.L.pq}.]For each $k\in \NN$, we know that
the map $\id^{\oast k}:B\rightarrow B$ is a
$\kk$-algebra homomorphism (by Lemma \ref{lem.BtoA.preserveidk}), and
thus satisfies $\id^{\oast k}\left(  bc\right)
=\id^{\oast k}\left(  b\right)  \cdot
\id^{\oast k}\left(  c\right)  $. Hence,
\begin{align}
&  \left(  \id^{\oast0}\left(  bc\right)
,\id^{\oast1}\left(  bc\right)  ,
\id^{\oast2}\left(  bc\right)  ,\ldots\right) \nonumber\\
&  =\left(  \id^{\oast0}\left(  b\right)
\id^{\oast0}\left(  c\right)  ,
\id^{\oast1}\left(  b\right)  \id^{\oast1}%
\left(  c\right)  ,\id^{\oast2}\left(  b\right)
\id^{\oast2}\left(  c\right)  ,\ldots\right)  .
\label{pf.prop.L.pq.id-prod}
\end{align}

Lemma \ref{lem.L.polynomiality} (applied to $m=p$) shows that $p$ is a
degree-upper bound of $b$ if and only if the sequence $\left(
\id^{\oast0}\left(  b\right)  ,
\id^{\oast1}\left(  b\right)  ,\id^{\oast2}%
\left(  b\right)  ,\ldots\right)  $ is $p$-polynomial. Thus, the sequence
$\left(  \id^{\oast0}\left(  b\right)
,\id^{\oast1}\left(  b\right)  ,\id^{\oast2}\left(  b\right)  ,
\ldots\right)  $ is $p$-polynomial (since
$p$ is a degree-upper bound of $b$). The same argument (applied to $c$ and $q$
instead of $b$ and $p$) shows that the sequence
$\left(  \id^{\oast0}\left(  c\right)  ,\id^{\oast1}%
\left(  c\right)  ,\id^{\oast2}\left(  c\right)
,\ldots\right)  $ is $q$-polynomial. Hence, Lemma \ref{lem.findiff.prod}
(applied to $A=B$, $a_{i}=\id^{\oast i}\left(
b\right)  $ and $b_{i}=\id^{\oast i}\left(  c\right)
$) shows that $\left(  \id^{\oast0}\left(  b\right)
\id^{\oast0}\left(  c\right)  ,\id^{\oast1}\left(  b\right)  \id^{\oast1}%
\left(  c\right)  ,\id^{\oast2}\left(  b\right)
\id^{\oast2}\left(  c\right)  ,\ldots\right)  $ is a
$\left(  p+q\right)  $-polynomial sequence of elements of $B$. In view of
\eqref{pf.prop.L.pq.id-prod}, this rewrites as follows: \\ $\left(
\id^{\oast0}\left(  bc\right)  ,\id^{\oast1}\left(  bc\right)  ,\id^{\oast2}%
\left(  bc\right)  ,\ldots\right)  $ is a $\left(  p+q\right)  $-polynomial
sequence of elements of $B$.

But Lemma \ref{lem.L.polynomiality} (applied to $p+q$ and $bc$ instead of $m$
and $b$) shows that $p+q$ is a degree-upper bound of $bc$ if and only if the
sequence $\left(  \id^{\oast0}\left(  bc\right)
,\id^{\oast1}\left(  bc\right)  ,\id^{\oast2}\left(  bc\right)  ,
\ldots\right)  $ is $\left(  p+q\right)
$-polynomial. Hence, $p+q$ is a degree-upper bound of $bc$ (since the sequence
$\left(  \id^{\oast0}\left(  bc\right)
,\id^{\oast1}\left(  bc\right)  ,\id^{\oast2}\left(  bc\right)  ,
\ldots\right)  $ is $\left(  p+q\right)
$-polynomial). This proves Proposition \ref{prop.L.pq}.
\end{proof}

\begin{corollary}
\label{cor.L.subalg}Let $B$ be a commutative $\kk$-bialgebra. Let $L$ denote
the set of all id-unipotent elements of $B$. Then, $L$ is a
$\kk$-subalgebra of $B$.
\end{corollary}

\begin{proof}
[Proof of Corollary \ref{cor.L.subalg}.]We have already seen in Remark
\ref{rmk.L.submodule} that $L$ is a $\kk$-submodule of $B$. Thus, it
suffices to show that $1\in L$ and that all $b,c\in L$ satisfy $bc\in L$.

It is easy to see that $1\in L$: Indeed, it is easy to see (by induction) that
$\left(  \eta\epsilon-\id\right)  ^{\oast n}\left(  1\right)
=0$ for every positive integer $n$. Thus, $0$ is a degree-upper bound of $1$.
Hence, the element $1$ is id-unipotent, i.e., we have $1\in L$.

It remains to show that all $b,c\in L$ satisfy $bc\in L$. So let $b,c\in L$ be
arbitrary. Thus, $b$ and $c$ are two id-unipotent elements of $B$. Clearly,
$b\in B$ has a degree-upper bound (since $b$ is id-unipotent); let us denote
this bound by $p$. We WLOG assume that $p$ is nonnegative (since otherwise, we
can replace $p$ by $0$). Likewise, we can find a nonnegative degree-upper
bound $q$ of $c$. Now, Proposition \ref{prop.L.pq} shows that $p+q$ is a
degree-upper bound of $bc$. Hence, $bc$ is id-unipotent. In other words,
$bc\in L$.

We thus have shown that all $b,c\in L$ satisfy $bc\in L$. As we have seen,
this concludes the proof of Corollary \ref{cor.L.subalg}.
\end{proof}

\subsection{Appendix: The id-unipotents do not form a subcoalgebra}

In this subsection, we shall present an example in which the id-unipotent
elements in a commutative $\mathbf{k}$-bialgebra do not form a $\mathbf{k}%
$-subcoalgebra. This example was found by GPT-5.5 in June 2026:

\begin{example}
\label{exa.L.not-subcoalg}Let $\mathbf{k}$ be a field of characteristic $0$.
Let $W$ be the free monoid on two generators $a,b$. Let $\varnothing\in W$ be
the empty word (i.e., the neutral element of $W$). Consider the commutative
$\mathbf{k}$-algebra $\mathbf{k}^{W}$ of all maps $f:W\rightarrow\mathbf{k}$,
with pointwise algebra structure (i.e., with $\left(  fg\right)  \left(
w\right)  =f\left(  w\right)  g\left(  w\right)  $ for all $w\in W$ and
$f,g\in\mathbf{k}^{W}$).

Let $\left[  \mathcal{S}\right]  $ denote the truth value of a logical
statement $\mathcal{S}$ (that is, the number $1$ if $\mathcal{S}$ is true, and
the number $0$ otherwise). Define four elements $E,p,s,c\in\mathbf{k}^{W}$ by
setting%
\begin{align*}
E\left(  w\right)    & =\left[  w=\varnothing\right]  ;\\
p\left(  w\right)    & =\left[  w\text{ begins with }b\right]  ;\\
s\left(  w\right)    & =\left[  w\text{ ends with }a\right]  ;\\
c\left(  w\right)    & =\left(  \text{\# of cyclic }ab\text{'s in }w\right)
\end{align*}
for each $w\in W$. Here:

\begin{itemize}
\item A word is said to \textquotedblleft begin with $b$\textquotedblright\ if
it is nonempty and its first letter is $b$.

\item A word is said to \textquotedblleft end with $a$\textquotedblright\ if
it is nonempty and its last letter is $a$.

\item A \textquotedblleft cyclic $ab$\textquotedblright\ in a word $w$ is a
position of $w$ that contains the letter $a$ and is followed directly by a
letter $b$; here, we understand that the first letter of $w$ follows directly
the last letter of $w$ (thus \textquotedblleft cyclic\textquotedblright). The
empty word has no cyclic $ab$'s. For instance, the word $abbaaabab$ has $3$
cyclic $ab$'s, whereas the word $babaa$ has $2$.
\end{itemize}

Also, let $1_{\mathbf{k}^{W}}\in\mathbf{k}^{W}$ be the unity of the
$\mathbf{k}$-algebra $\mathbf{k}^{W}$ (that is, the map sending each word $w$
to $1$). Define two further elements $r=sp$ and $A=1_{\mathbf{k}^{W}}-E$ of
$\mathbf{k}^{W}$. Note that $r\left(  w\right)  =s\left(  w\right)  p\left(
w\right)  =\left[  w\text{ ends with }a\text{ and begins with }b\right]  $ is
$1$ precisely when $w$ has a cyclic $ab$ that \textquotedblleft wraps around
the cycle\textquotedblright; thus, $\left(  c-r\right)  \left(  w\right)  $
counts the non-cyclic $ab$'s in the word $w$. Of course, $A\left(  w\right)
=\left[  w\neq\varnothing\right]  $ for any word $w$.

It is now easy to see that any two words $u,v\in W$ satisfy the four
equalities%
\begin{align}
E\left(  uv\right)    & =E\left(  u\right)  \cdot E\left(  v\right)
;\label{eq.exa.L.not-subcoalg.uv1}\\
p\left(  uv\right)    & =p\left(  u\right)  \cdot1_{\mathbf{k}^{W}}\left(
v\right)  +E\left(  u\right)  \cdot p\left(  v\right)
;\label{eq.exa.L.not-subcoalg.uv2}\\
s\left(  uv\right)    & =1_{\mathbf{k}^{W}}\left(  u\right)  \cdot s\left(
v\right)  +s\left(  u\right)  \cdot E\left(  v\right)
;\label{eq.exa.L.not-subcoalg.uv3}\\
c\left(  uv\right)    & =c\left(  u\right)  \cdot1_{\mathbf{k}^{W}}\left(
v\right)  +1_{\mathbf{k}^{W}}\left(  u\right)  \cdot c\left(  v\right)
+s\left(  u\right)  \cdot p\left(  v\right)  +p\left(  u\right)  \cdot
s\left(  v\right)  \nonumber\\
& \ \ \ \ \ \ \ \ \ \ -r\left(  u\right)  \cdot A\left(  v\right)  -A\left(
u\right)  \cdot r\left(  v\right)  .\label{eq.exa.L.not-subcoalg.uv4}%
\end{align}

If $B$ is any $\kk$-subalgebra of $\mathbf{k}^W$, then
there is a canonical $\mathbf{k}$-algebra morphism $\iota:B\otimes
B\rightarrow\mathbf{k}^{W}\otimes\mathbf{k}^{W}\rightarrow\mathbf{k}^{W\times
W}$ obtained from the inclusion map $B\rightarrow\mathbf{k}^{W}$ and the map
\begin{align}
\mathbf{k}^{W}\otimes\mathbf{k}^{W}  & \rightarrow\mathbf{k}^{W\times
W},\nonumber\\
f\otimes g  & \mapsto\left(  \left(  u,v\right)  \mapsto f\left(  u\right)
g\left(  v\right)  \right)  .\label{eq.exa.L.not-subcoalg.iota2}%
\end{align}
This morphism $\iota$ is easily seen to be injective. (Here, we use the fact
that $\mathbf{k}$ is a field. This guarantees that the map $B\otimes
B\rightarrow\mathbf{k}^{W}\otimes\mathbf{k}^{W}$ is injective. It remains to
show that (\ref{eq.exa.L.not-subcoalg.iota2}) is injective. Viewing
$\mathbf{k}^{W}$ and $\mathbf{k}^{W\times W}$ as the $\mathbf{k}$-linear duals
of the free $\mathbf{k}$-vector spaces $\mathbf{k}^{\left(  W\right)  }$ and
$\mathbf{k}^{\left(  W\times W\right)  }$, and recalling that $\mathbf{k}%
^{\left(  W\times W\right)  }\cong\mathbf{k}^{\left(  W\right)  }%
\otimes\mathbf{k}^{\left(  W\right)  }$ canonically, this boils down to showing
that the canonical $\mathbf{k}$-linear map $M^{\vee}\otimes N^{\vee
}\rightarrow\left(  M\otimes N\right)  ^{\vee}$ is injective whenever $M$ and
$N$ are two $\mathbf{k}$-vector spaces. But this is a particular case of
Proposition \ref{DualInj} below.) Thus, we can view $B\otimes B$ as a
$\mathbf{k}$-subalgebra of $\mathbf{k}^{W\times W}$ via the embedding $\iota$.

Now, let $B$ be the $\mathbf{k}$-subalgebra of $\mathbf{k}^{W}$ generated by
$E,p,s,c$. Then, $B$ contains $1_{\mathbf{k}^{W}}$, $r$ and $A$ as well. We
claim that $B$ becomes a $\mathbf{k}$-bialgebra with $\Delta:B\rightarrow
B\otimes B$ and $\epsilon:B\rightarrow\mathbf{k}$ defined by the equalities%
\begin{align*}
\left(  \Delta\left(  f\right)  \right)  \left(  u,v\right)    & =f\left(
uv\right)  \ \ \ \ \ \ \ \ \ \ \text{for all }f\in B\text{ and }u,v\in
W\ \ \ \ \ \ \ \ \ \ \text{and}\\
\epsilon\left(  f\right)    & =f\left(  \varnothing\right)
\ \ \ \ \ \ \ \ \ \ \text{for all }f\in B,
\end{align*}
where $\left(  \Delta\left(  f\right)  \right)  \left(  u,v\right)  $ denotes
the value of $\iota\left(  \Delta\left(  f\right)  \right)  \in\mathbf{k}%
^{W\times W}$ on the pair $\left(  u,v\right)  \in W\times W$. Indeed, a
priori, our definition of $\Delta\left(  f\right)  $ only defines a map from
$W\times W$ to $\mathbf{k}$, not an element of $B\otimes B$; however, the
equalities (\ref{eq.exa.L.not-subcoalg.uv1}), (\ref{eq.exa.L.not-subcoalg.uv2}%
), (\ref{eq.exa.L.not-subcoalg.uv3}) and (\ref{eq.exa.L.not-subcoalg.uv4})
show that%
\begin{align*}
\Delta\left(  E\right)    & =E\otimes E,\\
\Delta\left(  p\right)    & =p\otimes1_{\mathbf{k}^{W}}+E\otimes p,\\
\Delta\left(  s\right)    & =1_{\mathbf{k}^{W}}\otimes s+s\otimes E,\\
\Delta\left(  c\right)    & =c\otimes1_{\mathbf{k}^{W}}+1_{\mathbf{k}^{W}%
}\otimes c+s\otimes p+p\otimes s-r\otimes A-A\otimes r
\end{align*}
all lie in the image of $\iota$, and thus $\Delta\left(  B\right)  \subseteq
B\otimes B$ is well-defined. (Note also that $\epsilon\left(  E\right)  =1$
and $\epsilon\left(  p\right)  =\epsilon\left(  s\right)  =\epsilon\left(
c\right)  =0$.) Note that our construction of $\Delta$ shows that if $f\in B$
and $w\in W$, then%
\begin{equation}
\left(  \operatorname*{id}\nolimits^{\oast k}\left(  f\right)  \right)
\left(  w\right)  =f\left(  w^{k}\right)  \ \ \ \ \ \ \ \ \ \ \text{for all
}k\in\mathbb{N}.\label{eq.exa.L.not-subcoalg.wk}%
\end{equation}

Now, it is easy to see that each $k\in\mathbb{N}$ and $w\in W$ satisfy
$c\left(  w^{k}\right)  =kc\left(  w\right)  $. Hence, by
(\ref{eq.exa.L.not-subcoalg.wk}), we conclude that%
\[
\operatorname*{id}\nolimits^{\oast k}\left(  c\right)
=kc\ \ \ \ \ \ \ \ \ \ \text{for all }k\in\mathbb{N}.
\]
Therefore, the sequence $\left(  \operatorname{id}^{\oast0}\left(
c\right)  ,\operatorname{id}^{\oast1}\left(  c\right)
,\operatorname{id}^{\oast2}\left(  c\right)  ,\ldots\right)  $ is
$1$-polynomial. Hence, by Lemma \ref{lem.L.polynomiality}, the number $1$ is a
degree-upper bound of $c$. Therefore, $c$ is id-unipotent. In other words,
$c\in L$.

However, $\Delta\left(  c\right)  \notin L\otimes L$. In order to see this, we
observe that any $n,m\in\mathbb{N}$ satisfy%
\begin{equation}
\left(  \Delta\left(  c\right)  \right)  \left(  a^{n},b^{m}\right)  =c\left(
a^{n}b^{m}\right)  =%
\begin{cases}
1, & \text{if }n>0\text{ and }m>0;\\
0, & \text{otherwise}.
\end{cases}
\label{eq.exa.L.not-subcoalg.c1}%
\end{equation}
If $\Delta\left(  c\right)  $ belonged to $L\otimes L$, then there would be a
finite set $I$ and some elements  $f_{i},g_{i}\in L$ for each $i\in I$ such
that $\Delta\left(  c\right)  =\sum_{i\in I}f_{i}\otimes g_{i}$. Then, we
would have%
\[
\left(  \Delta\left(  c\right)  \right)  \left(  a^{n},b^{m}\right)  =\left(
\sum_{i\in I}f_{i}\otimes g_{i}\right)  \left(  a^{n},b^{m}\right)
=\sum_{i\in I}f_{i}\left(  a^{n}\right)  g_{i}\left(  b^{m}\right)  ,
\]
and therefore%
\begin{equation}
\sum_{i\in I}f_{i}\left(  a^{n}\right)  g_{i}\left(  b^{m}\right)  =%
\begin{cases}
1, & \text{if }n>0\text{ and }m>0;\\
0, & \text{otherwise}%
\end{cases}
\label{eq.exa.L.not-subcoalg.c2}%
\end{equation}
(by (\ref{eq.exa.L.not-subcoalg.c1})).

However, for each $i\in I$, the map $f_{i}$ is id-unipotent (since $f_{i}\in
L$); thus, the sequence $\left(  \operatorname{id}^{\oast0}\left(
f_{i}\right)  ,\operatorname{id}^{\oast1}\left(  f_{i}\right)
,\operatorname{id}^{\oast2}\left(  f_{i}\right)  ,\ldots\right)  $ is
$r$-polynomial for some $r\in\mathbb{N}$ (by Lemma \ref{lem.L.polynomiality}).
Hence, the sequence $\left(  \left(  \operatorname{id}^{\oast0}\left(
f_{i}\right)  \right)  \left(  a\right)  ,\left(  \operatorname{id}%
^{\oast1}\left(  f_{i}\right)  \right)  \left(  a\right)  ,\left(
\operatorname{id}^{\oast2}\left(  f_{i}\right)  \right)  \left(
a\right)  ,\ldots\right)  $ must also be $r$-polynomial. But this latter
sequence is simply the sequence \newline
$\left(  f_{i}\left(  a^{0}\right)
,f_{i}\left(  a^{1}\right)  ,f_{i}\left(  a^{2}\right)  ,\ldots\right)  $ (by
(\ref{eq.exa.L.not-subcoalg.wk})). Thus, the sequence $\left(  f_{i}\left(
a^{0}\right)  ,f_{i}\left(  a^{1}\right)  ,f_{i}\left(  a^{2}\right)
,\ldots\right)  $ must be $r$-polynomial. In other words, $f_{i}\left(
a^{n}\right)  $ is a polynomial function in $n$ with coefficients in
$\mathbf{k}$ (here, we use $\operatorname*{char}\mathbf{k}=0$). Similarly,
$g_{i}\left(  b^{m}\right)  $ is a polynomial function in $m$ with
coefficients in $\mathbf{k}$. Therefore, the finite sum $\sum_{i\in I}%
f_{i}\left(  a^{n}\right)  g_{i}\left(  b^{m}\right)  $ is a polynomial
function in $n$ and $m$ with coefficients in $\mathbf{k}$. But the equality
(\ref{eq.exa.L.not-subcoalg.c2}) shows that this polynomial function has the
value $1$ on all pairs $\left(  n,m\right)  $ with $n>0$ and $m>0$; thus, it
must be identically $1$ (since a polynomial function that is constant on an
infinite grid must be constant everywhere). This contradicts the $n=0$
case of (\ref{eq.exa.L.not-subcoalg.c2}). This contradiction shows that
$\Delta\left(  c\right)  $ cannot belong to $L\otimes L$. Thus, $L$ is not a
$\mathbf{k}$-subcoalgebra of $B$ (since $c\in L$ but $\Delta\left(  c\right)
\notin L\otimes L$).
\end{example}

\subsection{Appendix: Theorem \ref{thm.1} and noncommutative bialgebras}

The requirement for $B$ to be commutative
in Theorem~\ref{thm.1} can be
replaced by the weaker requirement that $g_1, g_2,
\ldots, g_n$ lie in the center of $B$.
This generalization requires some changes to the proof.
If $A$ is a noncommutative ring and
$f = \sum_{i \in \NN} f_i t^i \in A\left[t\right]$
is a polynomial, then any element $a$ of 
the center of $A$ can be substituted into $f$ to obtain
$f\left(a\right) = \sum_{i \in \NN} f_i a^i$;
this kind of substitution is well-behaved (i.e., we have
$\left(f+g\right)\left(a\right) = f\left(a\right) +
g\left(a\right)$ and $\left(fg\right)\left(a\right)
= f\left(a\right) \cdot g\left(a\right)$ for any
$f, g \in A\left[t\right]$) since $a$ is central.
Lemma~\ref{lem.laurent-sub} still holds for
noncommutative $A$ as long as we require $u$ to lie
in the center of $A$.
Lemma~\ref{lem.BtoA.preserveidk} cannot itself be
generalized to the noncommutative situation, but it
can instead be replaced by the following (more targeted)
lemma:

\begin{lemma}
\label{lem.BtoA.preserveidk-noncomm}
Let $B$ be a $\kk$-bialgebra.
Let $g$ be a grouplike element of $B$ that lies in
the center of $B$.
Let $k\in \NN$. Then, $\id^{\oast k} \left(bg\right)
= \id^{\oast k} \left(b\right) g^k$
for each $b \in B$.
\end{lemma}

The interested reader can fill in the details.

However, we cannot push our luck much further:
If we remove the commutativity requirement from
Theorem~\ref{thm.1} altogether (without offering any
substitute), then Theorem~\ref{thm.1} becomes false,
even if the regularity assumptions are interpreted
as two-sided regularity\footnote{An element $a$ of a
ring $A$ is said to be
\begin{itemize}
\item \emph{left regular} if every
$b \in A$ satisfying $ab = 0$ must satisfy $b = 0$.
\item \emph{right regular} if every
$b \in A$ satisfying $ba = 0$ must satisfy $b = 0$.
\item \emph{two-sided regular} if it is both
left regular and right regular.
\end{itemize}
}.
The following rather intricate counterexample was
discovered and written up by GPT-5.5:

\begin{example}
\label{exa.thm1.noncomm}
Let $\kk$ be any field. Let $B$ be the associative $\kk$-algebra
with generators $e,f,g,x$ subject to the relations
\begin{align}
 e^2&=e, & f^2&=f, & ef&=f, & fe&=e, \qquad
\label{eq.exa.thm1.noncomm.rel-idem}
\\
 ex&=x, & xe&=x, & fx&=x,
\label{eq.exa.thm1.noncomm.rel-x}\\
 eg&=f, & xg&=xf
\label{eq.exa.thm1.noncomm.rel-g}
\end{align}
(note: we do not want $xf = x$).
This is a monoid algebra of a certain non-abelian monoid.
We define a coalgebra structure on $B$ by setting
\begin{align*}
\Delta\left(  e\right)  &= e\otimes e,
& \Delta\left(  f\right)  &= f\otimes f,
& \Delta\left(  g\right)  &= g\otimes g,\\
\Delta\left(  x\right)  &= x\otimes e+e\otimes x
\end{align*}
and
\[
\epsilon\left(  e\right)  =\epsilon\left(  f\right)
=\epsilon\left(  g\right)  =1,
\qquad
\epsilon\left(  x\right)  =0.
\]
These formulas make $B$ into a $\kk$-bialgebra. Indeed, the formulas clearly
respect the relations \eqref{eq.exa.thm1.noncomm.rel-idem}. They also respect
\eqref{eq.exa.thm1.noncomm.rel-x}; for instance,
\begin{align*}
\Delta\left(  f\right)  \Delta\left(  x\right)
&=\left(  f\otimes f\right)  \left(  x\otimes e+e\otimes x\right)  \\
&=fx\otimes fe+fe\otimes fx
=x\otimes e+e\otimes x
=\Delta\left(  x\right)  .
\end{align*}
Finally, the relation $eg=f$ is respected because
\[
\Delta\left(  e\right)  \Delta\left(  g\right)
=eg\otimes eg=f\otimes f=\Delta\left(  f\right)  ,
\]
and the relation $xg=xf$ is respected because
\begin{align*}
\Delta\left(  x\right)  \Delta\left(  g\right)
&=xg\otimes eg+eg\otimes xg
=xf\otimes f+f\otimes xf,\\
\Delta\left(  x\right)  \Delta\left(  f\right)
&=xf\otimes ef+ef\otimes xf
=xf\otimes f+f\otimes xf.
\end{align*}
Thus, $\Delta$ is well-defined as a $\kk$-algebra homomorphism
$B\to B\otimes B$. Coassociativity and counitality are immediate on the
generators. Thus, $B$ is a $\kk$-bialgebra.

The elements $g$ and $1$ are grouplike. We shall show that they satisfy the
regularity assumptions that would appear in a noncommutative analogue of
Theorem~\ref{thm.1}, but nevertheless are not left $L$-linearly independent.

First we record a normal form for the algebra $B$. Consider the following
rewriting rules for the underlying monoid of the monoid algebra $B$:
\begin{align*}
 ee&\longrightarrow e, & ff&\longrightarrow f,
& ef&\longrightarrow f, & fe&\longrightarrow e,\\
 ex&\longrightarrow x, & xe&\longrightarrow x,
& fx&\longrightarrow x,\\
 eg&\longrightarrow f, & xg&\longrightarrow xf.
\end{align*}
This rewriting system is terminating and confluent.
(Termination is clear by assigning the generators
$e,f,x$ degree $1$ and the generator $g$ degree $2$; then,
each rewriting rule reduces the degree.
Thus, by the diamond lemma \cite[Theorem 1.7.10]{Sapir},
its confluence follows from local confluence, which is easily
verified.)
The only ambiguities are length-$3$ overlap ambiguities,
and they are directly checked. Hence, by the Diamond
Lemma, a $\kk$-basis of $B$ is formed by the (residue classes
of) all words in the alphabet $\left\{  e,f,g,x\right\}$
that contain none of the adjacent pairs
\[
 ee,\quad ff,\quad ef,\quad fe,\quad ex,\quad xe,\quad fx,\quad eg,\quad xg.
\]
In other words, it is formed by all words in which
\begin{itemize}
\item each $f$ (unless it is the last letter) is followed by $g$;
\item each $x$ (unless it is the last letter) is followed by $f$ or $x$;
\item the letter $e$ can only appear as the last letter.
\end{itemize}
In particular, the empty word (corresponding to the unity $1$)
is normal, as are all one-letter words.
% Equivalently, a normal word can be followed by the
% next letter according to the following table:
% \[
% \begin{array}{c|c}
% \text{last letter} & \text{allowed following letters}\\ \hline
 % e & \text{none}\\
 % f & g\\
 % g & e,f,g,x\\
 % x & f,x.
% \end{array}
% \]
In particular, the words $g$, $1$, $x$ and $xf$ are nonzero, and $B$ is not
commutative, since $eg=f$ whereas $ge$ is a normal word distinct from $f$.

Right multiplication by $g$ sends normal words $w$ to normal words as follows:
\[
w \mapsto wg = \begin{cases}
uf, & \text{if }w=ue;\\
uxf, & \text{if }w=ux;\\
wg, & \text{otherwise.}
\end{cases}
\]
This map from the set of normal words to itself is easily seen to be injective.
Thus, right
multiplication by $g$ is injective. Moreover, this map has no finite cycles:
indeed, the total number of $f$'s and $g$'s in a normal word
increases under every application.
Hence, right multiplication by $g-1$ is injective as well. Indeed, if a finite
linear combination $a$ of normal words satisfied $a\left(  g-1\right)  =0$,
then the finite set of normal words appearing in $a$ would be stable under the
above injective map $w\mapsto wg$, which would force a finite cycle.

Similarly, left multiplication by $g$ sends each normal word $w$ to the normal
word $gw$. This map is injective and strictly increases the length of $w$.
Therefore, left multiplication by both $g$ and $g-1$ is injective. Thus, $g$
and $g-1$ are two-sided regular in $B$.

Now let $L$ denote the set of all id-unipotent elements of $B$. We have
$x\in L$. Indeed, from $\Delta\left(e\right) = e \otimes e$
and $\epsilon\left(e\right) = 1$, we easily see (by induction on $k$)
that
\[
\id^{\oast k}\left(  e\right)  = e
\qquad\text{for all }k>0.
\]
Next, we note that $e$ is a two-sided identity for $x$, and
\[
\Delta\left(  x\right)  =x\otimes e+e\otimes x,
\]
so an induction on $k$ gives
\[
\id^{\oast k}\left(  x\right)  =kx
\qquad\text{for all }k\in\NN.
\]
Therefore, the sequence $\left(  \operatorname{id}^{\oast0}\left(
x\right)  ,\operatorname{id}^{\oast1}\left(  x\right)
,\operatorname{id}^{\oast2}\left(  x\right)  ,\ldots\right)  $ is
$1$-polynomial. Hence, by Lemma \ref{lem.L.polynomiality}, the number $1$ is a
degree-upper bound of $x$.
Therefore, $x$ is id-unipotent. In other words,
$x\in L$.

Likewise, $xf\in L$: We have
\[
\Delta\left(  xf\right)  =xf\otimes f+f\otimes xf,
\]
and $f$ is a two-sided identity for $xf$, whence
\[
\id^{\oast k}\left(  xf\right)  =kxf
\qquad\text{for all }k\in\NN.
\]
Thus, $1$ is a degree-upper bound of $xf$ as well;
hence $xf \in L$.

However, the relation $xg=xf$ gives
\[
 xg-\left(  xf\right)  1=0.
\]
This is a nontrivial left $L$-linear relation between the two grouplike
elements $g$ and $1$, since $x\in L$ and $xf\in L$ and since the normal-word
basis shows that $x$ and $xf$ are nonzero. Hence, these two grouplike elements
are not left $L$-linearly independent, despite the fact that $g$, $1$ and
$g-1$ are two-sided regular.
\end{example}

\section{\label{sect.gen3}Linear independence in the dual algebra}

So far we have considered grouplike elements in
coalgebras and bialgebras.
A related (and, occasionally, equivalent) concept are
the \emph{characters} of an algebra.

\subsection{\label{subsect.chars-grouplikes}Some words on characters}

We recall that a \emph{character} of a $\kk$-algebra
$A$ means a $\kk$-algebra homomorphism from
$A$ to $\kk$.
Before we study linear independence questions for
characters, let us briefly survey their relation to
grouplike elements.

The simplest way to connect characters with grouplike
elements is the following (easily verified)
fact:

\begin{proposition}
\label{prop.char-vs-gl-1}
Let $C$ be a $\kk$-coalgebra that is free as a $\kk$-module.
Let $c \in C$.
Consider the $\kk$-linear map $c^{\vee\vee} : C^\vee \to \kk$
that sends each $f \in C^\vee$ to $f\left(c\right) \in \kk$.
Then, $c^{\vee\vee}$ is a character of the $\kk$-algebra $C^\vee$
if and only if $c$ is grouplike in $C$.
\end{proposition}

\begin{proof}
[Proof of Proposition \ref{prop.char-vs-gl-1} (sketched).]We shall prove the
two equivalences
\begin{align}
& \ \left(  \Delta\left(  c\right)  =c\otimes c\right)  \nonumber\\
& \Longleftrightarrow\ \left(  c^{\vee\vee}\left(  f\oast g\right)
=c^{\vee\vee}\left(  f\right)  \cdot c^{\vee\vee}\left(  g\right)  \text{ for
all }f,g\in C^{\vee}\right)
\label{pf.prop.char-vs-gl-1.equiv1}
\end{align}
and
\begin{equation}
\left(  \epsilon\left(  c\right)  =1\right)  \ \Longleftrightarrow\ \left(
c^{\vee\vee}\left(  \epsilon\right)  =1\right)
.
\label{pf.prop.char-vs-gl-1.equiv2}
\end{equation}

Indeed, the equivalence \eqref{pf.prop.char-vs-gl-1.equiv2} follows from the
fact that $c^{\vee\vee}\left(  \epsilon\right)  =\epsilon\left(  c\right)  $
(which is a direct consequence of the definition of $c^{\vee\vee}$). Let us
now prove the equivalence \eqref{pf.prop.char-vs-gl-1.equiv1}.

The $\kk$-module $C$ is free. Hence, it is easy to prove the following fact:

\begin{statement}
\textit{Fact 1:} Let $x$ and $y$ be two elements of $C\otimes C$. Then, $x=y$
if and only if we have%
\[
\left(  f\otimes g\right)  \left(  x\right)  =\left(  f\otimes g\right)
\left(  y\right)  \qquad\qquad\text{for all }f,g\in C^{\vee}.
\]

\end{statement}

\noindent Fact 1 is often stated as ``the pure tensors in
$C^{\vee}\otimes C^{\vee}$ separate $C\otimes C$''. It can be
proved by picking a basis $\left(  e_{i}\right)  _{i\in I}$ of $C$ and its
corresponding dual ``basis'' $\left(
e_{i}^{\ast}\right)  _{i\in I}$ of $C^{\vee}$ (with $e_{i}^{\ast}\left(
e_{j}\right)  $ being the Kronecker delta $\delta_{i,j}$ for all $i,j\in I$),
and arguing that $\left(  e_{i}^{\ast}\otimes e_{j}^{\ast}\right)  _{\left(
i,j\right)  \in I\times I}$ is the dual ``basis'' to the basis
$\left(  e_{i}\otimes e_{j}\right)
_{\left(  i,j\right)  \in I\times I}$ of $C\otimes C$.

Applying Fact 1 to $x=\Delta\left(  c\right)  $ and $y=c\otimes c$, we obtain
the following equivalence:%
\begin{align}
& \ \left(  \Delta\left(  c\right)  =c\otimes c\right)  \nonumber\\
& \Longleftrightarrow\ \left(  \left(  f\otimes g\right)  \left(
\Delta\left(  c\right)  \right)  =\left(  f\otimes g\right)  \left(  c\otimes
c\right)  \text{ for all }f,g\in C^{\vee}\right)
.\label{pf.prop.char-vs-gl-1.equiv3}%
\end{align}
But the multiplication map $\mu_{\kk} : \kk \otimes \kk \to \kk$ of
$\kk$ is bijective; thus, for any $f,g\in
C^{\vee}$, we have the following chain of equivalences:%
\begin{align*}
& \ \left(  \left(  f\otimes g\right)  \left(  \Delta\left(  c\right)
\right)  =\left(  f\otimes g\right)  \left(  c\otimes c\right)  \right)  \\
& \Longleftrightarrow\ \left(  \mu_{\kk}\left(  \left(  f\otimes
g\right)  \left(  \Delta\left(  c\right)  \right)  \right)  =\mu_{\kk}
\left(  \left(  f\otimes g\right)  \left(  c\otimes c\right)  \right)
\right)  \\
& \Longleftrightarrow\ \left(  \left(  f\oast g\right)  \left(  c\right)
=f\left(  c\right)  \cdot g\left(  c\right)  \right)  \\
& \qquad\qquad\left(
\begin{array}
[c]{c}%
\text{since }\mu_{\kk}\left(  \left(  f\otimes g\right)  \left(
\Delta\left(  c\right)  \right)  \right)  =\underbrace{\left(  \mu
_{\kk}\circ\left(  f\otimes g\right)  \circ\Delta\right)
}_{\substack{=f\oast g\\\text{(by the definition of }\oast\text{)}}}\left(
c\right)  =\left(  f\oast g\right)  \left(  c\right)  \\
\text{and }\mu_{\kk}\left(  \left(  f\otimes g\right)  \left(  c\otimes
c\right)  \right)  =\mu_{\kk}\left(  f\left(  c\right)  \otimes
g\left(  c\right)  \right)  =f\left(  c\right)  \cdot g\left(  c\right)
\end{array}
\right)  \\
& \Longleftrightarrow\ \left(  c^{\vee\vee}\left(  f\oast g\right)
=c^{\vee\vee}\left(  f\right)  \cdot c^{\vee\vee}\left(  g\right)  \right)
\\
& \qquad\qquad\left(
\begin{array}
[c]{c}%
\text{since the definition of }c^{\vee\vee}\text{ yields }c^{\vee\vee}\left(
f\oast g\right)  =\left(  f\oast g\right)  \left(  c\right)  \\
\text{and }c^{\vee\vee}\left(  f\right)  =f\left(  c\right)  \text{ and
}c^{\vee\vee}\left(  g\right)  =g\left(  c\right)
\end{array}
\right)  .
\end{align*}
Hence, the equivalence \eqref{pf.prop.char-vs-gl-1.equiv3} is saying the same
thing as the equivalence \eqref{pf.prop.char-vs-gl-1.equiv1}. Thus, the latter
equivalence is proven.

We have now proved both equivalences \eqref{pf.prop.char-vs-gl-1.equiv1} and
\eqref{pf.prop.char-vs-gl-1.equiv2}. But the definition of grouplike elements
yields the following chain of equivalences:%
\begin{align*}
& \ \left(  c\text{ is grouplike in }C\right)  \\
& \Longleftrightarrow\ \left(  \Delta\left(  c\right)  =c\otimes c\right)
\text{ and }\left(  \epsilon\left(  c\right)  =1\right)  \\
& \Longleftrightarrow\ \left(  c^{\vee\vee}\left(  f\oast g\right)
=c^{\vee\vee}\left(  f\right)  \cdot c^{\vee\vee}\left(  g\right)  \text{ for
all }f,g\in C^{\vee}\right)  \text{ and }
\left(  c^{\vee\vee}\left( \epsilon\right) =1\right)  \\
& \qquad\qquad\left(  \text{by \eqref{pf.prop.char-vs-gl-1.equiv1} and
\eqref{pf.prop.char-vs-gl-1.equiv2}}\right)  \\
& \Longleftrightarrow\ \left(  c^{\vee\vee}\text{ is a }
\kk\text{-algebra homomorphism from }C^{\vee}\text{ to }\kk\right)  \\
& \qquad\qquad\left(  \text{since the map }c^{\vee\vee}\text{ is }%
\kk\text{-linear, and }\epsilon\text{ is the unity of the }%
\kk\text{-algebra }C^{\vee}\right)  \\
& \Longleftrightarrow\ \left(  c^{\vee\vee}\text{ is a character of the
}\kk\text{-algebra }C^{\vee}\right)  .
\end{align*}
This proves Proposition \ref{prop.char-vs-gl-1}.
\end{proof}

% \tcb{
% \begin{proof}[Old proof.]
% We always have 
% $\scal{c^{\vee\vee}}{\epsilon_C}=\scal{\epsilon_C}{c}$, 
% so the equivalence has to be to checked on the multiplicative condition.\\ 
% Now (providing $\scal{\epsilon_C}{c}=1_\kk$ satisfied), saying that $c^{\vee\vee}$ is a character of the 
% $\kk$-algebra $C^\vee$ is equivalent to 
% (pairings are always with the dual on the left)
% \begin{equation}\label{bidual1}
% \scal{f\otimes g}{\Delta_C(c)}=\scal{f\oast g}{c}=\scal{c^{\vee\vee}}{f\oast g}=\scal{f}{c}\scal{g}{c}
% =\scal{f\otimes g}{c\otimes c}
% \end{equation}
% but, as $C$ is a free $\kk$-module, $C^\vee\otimes C^\vee$ separates $C\otimes C$. Then equation \eqref{bidual1} is equivalent to $\Delta_C(c)=c\otimes c$. This completes the proof.
% \end{proof}
% }

Proposition~\ref{prop.char-vs-gl-1} can be used to characterize the
characters of some algebras as grouplikes.
To wit:
If $A$ is a $\kk$-algebra
that is finite free as a $\kk$-module, then the dual
$\kk$-module $A^\vee$ canonically receives the structure of
a $\kk$-coalgebra\footnote{This structure can be defined as
follows:
Its comultiplication $\Delta_{A^\vee}
: A^\vee \to A^\vee \otimes A^\vee$ is
the composition of the map
$\mu_A^\vee : A^\vee \to \left(A \otimes A\right)^\vee$
(which is
the dual of the multiplication $\mu_A : A \otimes A \to A$
of $A$)
with the canonical isomorphism
$\left(A \otimes A\right)^\vee \to A^\vee \otimes A^\vee$
(which exists because $A$ is finite free).
The counit $\epsilon_{A^\vee} : A^\vee \to \kk$ of $A^\vee$
is the dual of the unit map $\eta_A : \kk \to A$ of $A$.},
whose dual
$\left(A^\vee\right)^\vee$ is canonically isomorphic to the
$\kk$-algebra $A$ (via the standard $\kk$-module isomorphism
$A \to \left(A^\vee\right)^\vee$).
Thus, Proposition~\ref{prop.char-vs-gl-1} yields that
the characters of a $\kk$-algebra $A$ that is finite free
as a $\kk$-module are precisely the grouplike elements of
its dual coalgebra $A^\vee$.

There are some ways to extend this result to $\kk$-algebras
$A$ that are not finite free; the best-known case is when
$\kk$ is a field. In this case, every $\kk$-algebra $A$
has a \emph{Hopf dual} $A^o$ (also known as the
\emph{zero dual} or \emph{Sweedler dual}), which is a
$\kk$-coalgebra whose grouplike elements are precisely the
characters of $A$.
(See \cite[Section 6.0]{Sweedl69} for the definition and
fundamental properties of this Hopf dual and \cite{D21,DucTol09} 
for questions linked to rationality.)
When $\kk$ is not a field, this Hopf dual is not defined
any more. 
Indeed, the canonical map
$M^\vee \otimes N^\vee \to \left(M \otimes N\right)^\vee$
that is defined for any two $\kk$-modules $M$ and $N$
is known to be injective when $\kk$ is a field, but may
fail to be injective even when $\kk$ is an integral
domain\footnote{This will happen when
$M^\vee \otimes N^\vee$ has torsion
(see Subsection~\ref{subsect.dual-tensor} and Proposition \ref{DualInj} in particular).}.
Other concepts can be used to salvage the relation between
grouplikes and characters:
bases in duality, dual laws constructed directly \cite{Bui,Bui2,Duchamp-Luque},
pseudo-coproducts \cite[Section 2]{PatReu02}.
We shall not dwell on these things;
%\tcr{D2.[or, more likely, dwell on them later]};
but the upshot for us is that while the concepts of
characters of an algebra and grouplike elements of a coalgebra
are closely connected, neither concept subsumes the other.

% the range of the transpose of a law is not necessarily the suitable tensor square (this is the reason of Sweedler duals, see below).

For this reason, our third main result\footnote{which
has been briefly announced at
\url{https://mathoverflow.net/questions/310354} } will be stated directly in the language of characters 
%instead concerns the dual notion: that of characters
on a bialgebra (i.e., algebra morphisms from this
bialgebra to the base ring).

\subsection{Examples of characters}\label{CKS}

Before we state our result, let us see some examples
of characters of $\kk$-bialgebras.
The following notation will come rather useful:

\begin{definition}
\label{def.kleene-star}
Let $a$ be an element of a ring. Then, the
\emph{Kleene star} $a^*$ of $a$ is defined to be the
sum $\sum_{i=0}^\infty a^i = 1 + a + a^2 + a^3 + \cdots$,
assuming that this sum is well-defined (e.g., because $a$
is nilpotent or the sum converges for some other reason).
Note that $a^* = \dfrac{1}{1-a}$ whenever $a^*$ is
defined.
\end{definition}

The dual $\calB^\vee$ of a $\kk$-bialgebra $\calB$ is a
$\kk$-algebra, thus a ring; oftentimes, this ring has a
topology on it. (For example, if $\calB$ is graded, with
$\calB = \bigoplus_{n \geq 0} \calB_n$, then
its dual 
$\calB^\vee = \prod_{n \geq 0} \tup{\calB_n}^\vee$
has a product topology\footnote{In this case, $\kk$ is
usually considered as endowed with the discrete
topology. See \cite[Ch. II, \S 5.1]{Bourbaki-Lie1} for
an application to the combinatorics of the free Lie
algebra and the free group.}.) Oftentimes, this causes
the Kleene stars of some elements of $\calB^\vee$ to be
well-defined. As we shall soon see, characters of
$\calB$ are oftentimes found among such Kleene stars.

\subsubsection{Characters of polynomial rings}

We begin with polynomial rings: noncommutative and
commutative.

\begin{example}
\label{exa.char-as-kleene.ncp}
Consider the $\kk$-algebra $\ncp{\kk}{x,y}$ of
polynomials in two noncommuting variables $x, y$ over
$\kk$.
As a $\kk$-module, it has a basis consisting of words
in the alphabet $\set{x, y}$.
A character of this $\kk$-algebra is
uniquely determined by the images $\alpha$ and $\beta$
of $x$ and $y$ (indeed, this follows from the universal
property of $\ncp{\kk}{x,y}$). We can identify the dual
$\tup{\ncp{\kk}{x,y}}^\vee$ with the ring
$\ncs{\kk}{x,y}$ of noncommutative power series in $x, y$
(by means of the bilinear form that sends two equal
words to $1$ and sends two distinct words to $0$).
Thus, the character of $\ncp{\kk}{x,y}$ that sends
$x$ and $y$ to $\alpha$ and $\beta$ is explicitly given
as the Kleene star
\[
(\alpha.x+\beta.y)^*=\sum_{n\ge0}\,(\alpha.x+\beta.y)^n .
\]
This characterizes all characters of $\ncp{\kk}{x,y}$.
A similar characterization can be given for noncommutative
polynomial rings in multiple variables.
\end{example}

\begin{example}
\label{exa.char-as-kleene.cp}
Consider the $\kk$-algebra $\kk[x,y]$ of
polynomials in two commuting variables $x, y$ over
$\kk$.
As a $\kk$-module, it has a basis consisting of
monomials in $x, y$.
A character of this $\kk$-algebra is
uniquely determined by the images $\alpha$ and $\beta$
of $x$ and $y$ (indeed, this follows from the universal
property of $\kk[x,y]$). We can identify the dual
$\tup{\kk[x,y]}^\vee$ with the ring
$\kk[[x,y]]$ of power series in $x, y$
(by means of the bilinear form that sends two equal
monomials to $1$ and sends two distinct monomials to $0$).
Thus, the character of $\kk[x,y]$ that sends
$x$ and $y$ to $\alpha$ and $\beta$ is explicitly given
as the Kleene star
\[
(\alpha.x+\beta.y-\alpha\beta.xy)^*=\sum_{n\ge0}\,(\alpha.x+\beta.y-\alpha\beta.xy)^n .
\]
This characterizes all characters of $\kk[x,y]$.
A similar characterization can be given for
polynomial rings in multiple variables.
\end{example}

It is worth saying that the $\kk$-algebra $\Lambda$
of symmetric functions (\cite[Chapter 2]{GriRei})
is a polynomial ring in infinitely many variables over
$\kk$, and thus its characters can be described as in
Example~\ref{exa.char-as-kleene.cp}.
These characters have multiple names: they are known
as \emph{virtual alphabets}, or as \emph{big Witt
vectors}, or as \emph{specializations}.

\subsubsection{Characters of monoid rings}

Example~\ref{exa.char-as-kleene.ncp} and
Example~\ref{exa.char-as-kleene.cp} have a common
generalization. In fact, both noncommutative and
commutative polynomial rings are particular cases of
monoid rings of trace monoids. Let us thus briefly
discuss characters of monoid rings.

Let $\tup{M,\star,1_M}$ be a monoid, and let
$\calA=\kk[M]$ be its monoid algebra (see
Example~\ref{co-Hadamard}).
As in Example~\ref{co-Hadamard}, we define the
standard pairing $\scal{.}{.} :
\kk^M\otimes \kk[M]\to \kk$ by \eqref{main_pairing_M}.
The multiplication of $\kk[M]$ can thus be written as
\begin{equation}\label{convol_alg_mon}
P \star Q := \sum_{uv=w}\scal{P}{u}\scal{Q}{v}\,w
\end{equation}
(where $u, v, w$ range over $M$).

If the map $\star : M \times M\to M$ has finite fibers\footnote{Recall that a map $f:\ X\to Y$ between two sets $X$ and $Y$ \emph{has finite fibers} if and only if for each $y\in Y$, the preimage $f^{-1}(y)$ is finite. Thus, the map $\star : M \times M\to M$ has finite fibers if and only if for each $w \in M$, there are only finitely many pairs $\left(u,v\right)\in M \times M$ satisfying $uv=w$.} (this is condition (D) in \cite[Ch. III, \S 2.10]{Bourbaki-Alg1}), then we can extend
the formula \eqref{convol_alg_mon} to arbitrary
$P, Q \in \kk^M$ (as opposed to merely
$P, Q \in \kk[M]$).
In this case, the $\kk$-algebra $\tup{\kk^M,\star,1_M}$
is called the \emph{total algebra of $M$}
(see \cite[\S III.2.10]{Bourbaki-Alg1}),
%\footnote{See also \url{https://en.wikipedia.org/wiki/Total_algebra}.}
and its product is the Cauchy product between series.
For every $S\in \kk^M$, the family $(\scal{S}{m}\,m)_{m\in M}$ is summable\footnote{We say that a family $\tup{a_s}_{s \in S}$ of elements of $\kk^M$ is \emph{summable} if for any given $n \in M$, all but finitely many $s \in S$ satisfy $\scal{ a_s }{n} = 0$. Such a summable family will always have a well-defined infinite sum $\sum_{s \in S} a_s \in \kk^M$, whence the name ``summable''.
% Note: This is not quite my "pointwise finitely supported". The latter refers to families of linear maps, whereas "summable" refers to families of power series. --Darij
},
and its sum is $S=\sum_{m\in M}\scal{S}{m}\,m$.
We set $M_+:=M\rsetminus \{1_M\}$ and, for all series 
$S\in \kk^M$, we likewise set $S_+=S-\scal{S}{1_M}.1_M=\sum_{m\in M_+}\scal{S}{m}\,m$.
In order for the family $((S_+)^n)_{n\ge0}$ to be summable,
% pointwise finitely supported (see \cite[Definition 1.7.2]{GriRei} for the definition),
it is sufficient that the iterated multiplication map 
$\mu^{*}:\ (M_+)^*\to M$ defined by 
\begin{equation}\label{iterated_mult}
\mu^{*}[m_1,\ldots ,m_n]=m_1\cdots m_n\mbox{ (product within $M$)}
\end{equation}
have finite fibers (where we have written the word $[m_1, \ldots, m_n] \in (M_+)^*$ as a list to avoid confusion).%
\footnote{Furthermore, this condition is also necessary (if $S_+$ is generic) if $\kk=\Z$. These monoids are called ``locally finite'' in \cite{Eil}.}

\smallskip
Now, we assume that this condition (i.e., that $\mu^*$ has finite fibers) is satisfied.
Then, $\sum_{n\ge0} (S_+)^n=(1-S_+)^{-1}$ is the Kleene star $(S_+)^*$ of $S_+$ (see Definition~\ref{def.kleene-star}).

Define the series $\underline{M}=\sum_{m\in M}\,m \in \kk^M$; this is called the \emph{characteristic series} of the 
monoid $M$.
This series $\underline{M}$ is invertible\footnote{Here we are using the condition that $\mu^*$ has finite fibers, these monoids are called \emph{locally finite} in \cite{Eil}.}, and its inverse is\footnote{Here we are using the fact that $\underline{M}_+ = \underline{M} - 1$, so that $1 + \underline{M}_+ =\underline{M}$.}
\begin{align}
\underline{M}^{-1}
&=(-\underline{M}_+)^*=\sum_{n\ge 0}(-\underline{M}_+)^n=1-\underline{M}_++(\underline{M}_+)^2 \pm \cdots  \nonumber \\
&=\sum_{m\in M}\mu(m)\,m ,
\label{Mobdef}
\end{align}
where $\mu(m)$ is defined to be the $m$-th coordinate of $\underline{M}^{-1}$.
This defines a function $\mu : M \to \kk$, which is called the \emph{M\"obius function} of $M$.
Note that it is easy to see that $\mu(1_M)=1_\kk$,
and thus the equality \eqref{Mobdef} is equivalent to 
\begin{equation}\label{charMon}
\underline{M}=\big(-\sum_{m\in M_+}\mu(m)\,m\big)^* .
\end{equation}
Every character $\chi$ of the $\kk$-algebra $\tup{\kk[M],\star,1_M}$ can now be written in the form\footnote{We regard $\chi$ as an element of $\kk^M$, since the standard pairing $\scal{.}{.}$ allows us to identify the dual of $\kk[M]$ with $\kk^M$.}
\begin{equation}\label{MobChar}
\chi=\sum_{m\in M}\chi(m)\,m=\big(-\sum_{m\in M_+} \chi(m)\mu(m)\,m\big)^* .
\end{equation}
(This follows immediately from \eqref{charMon} by applying the continuous $\kk$-algebra homomorphism $\kk^M \to \kk^M$ that sends each $m \in M$ to $\chi(m) m$. The word ``continuous'' here refers to the product topology on $\kk^M$.)

\begin{example}
\label{Cartier-Foata} The celebrated arithmetic M\"obius function $\mu:\ \N_+\to \Z$ is the M\"obius function of the multiplicative monoid 
$\tup{\N_+,\cdot,1}$.
It is a particular case of the M\"obius function of a \emph{trace monoid}, i.e., of a monoid defined by commutation relations.
In a nutshell, a \emph{trace monoid} is determined by a finite reflexive undirected graph $\vt$ with vertex set $X$ (called the \emph{alphabet}).
The monoid is generated by all vertices $x \in X$, subject to the relations $xy = yx$ for all edges $\set{x, y}$ of $\vt$.
This monoid is called $M(X,\vt)$ (see \cite{DuKr93}).
It always satisfies the condition that $\mu^*$ has finite fibers.
Its M\"obius function is given by
 \begin{equation}\label{PartMonMob}
 \sum_{m\in M}\mu(m)\,m = \sum_{\substack{C\subseteq X;\\ C\text{ is a clique of } \vt }} (-1)^{|C|} \prod_{x\in C}x
 \end{equation}
(see \cite[Th\'eor\`eme 2.4]{CaFo}).

In particular, if $X=\{x_1,x_2,\ldots ,x_n\}$ and $\vt$ is the complete graph $\vt_{\full}$ (so that $xy = yx$ for all $x, y \in X$), the monoid $M(X,\vt_{\full})$ is the free abelian monoid $\{X^{\alpha}\}_{\alpha\in \N^{(X)}}$, and then we have $\mu(X^{\alpha})=0$ if the monomial $X^{\alpha}$ contains a square and $\mu(X^{\alpha})=(-1)^{|\alpha|}$ if $X^{\alpha}$ is square-free\footnote{When $X$ is the set of all prime numbers, one recovers the classical M\"obius function.}. In this case, \eqref{MobChar} becomes
\begin{equation}
\chi=\Big( 1 - (1 - \chi(x_1) x_1)(1 - \chi(x_2) x_2)\cdots (1 - \chi(x_n) x_n) \Big)^* .
\end{equation}
For example, every character of $\kk[x,y]$ is of the form 
\[
\chi=\Big(1 - (1 - \chi(x) x)(1 - \chi(y) y)\Big)^* = \Big(\chi(x) x + \chi(y) y - \chi(x) \chi(y) xy\Big)^*.
\]
This recovers the $(\alpha.x+\beta.y-\alpha\beta.xy)^*$ formula from Example~\ref{exa.char-as-kleene.cp}.

At the opposite extreme, with no commutations  
($\vt=\{(x,x)\}_{x\in X}$), we have $M(X,\vt)=X^*$,
and the monoid $M(X, \vt)$ is the free monoid on $X$;
now, the M\"obius function is supported on
$X\cup\{1_{X^*}\}$ with 
$\mu(1)=1$ and $\mu(x)=-1$ for all $x\in X$.
Thus, every character on $\ncp{\kk}{X}$ is of the form 
\begin{equation}
\chi=\Big(\sum_{x\in X}\chi(x)\,x\Big)^*
\end{equation}
(Kleene star of the plane property, see \cite{KSOP}). 
For $n = 2$, this recovers the $(\alpha.x+\beta.y)^*$ formula from Example~\ref{exa.char-as-kleene.ncp}.
\end{example}
\subsection{Characters on bialgebras}\label{CLA}
We begin with a simple fact:

\begin{proposition}
\label{prop.gen3-mon}
Let $\calB$ %$(\calB,.,1_\calB,\Delta,\epsilon)$
be a $\kk$-bialgebra.
Then, the set $\Xi(\calB)$ of characters of $\calB$ is a monoid for the convolution product $\oast$.
\end{proposition}

\begin{proof}
We know that $\calB^\vee$ is a $\kk$-algebra under convolution, and thus a monoid.
Hence, we just need to prove that $\Xi(\calB)$ is a submonoid of this monoid
$\calB^\vee$.
But this follows from the following two observations:
\begin{itemize}
\item If $p, q \in \Xi(\calB)$, then $p \oast q \in \Xi(\calB)$.
(This is a consequence of Lemma~\ref{lem.BtoA.bimap*bimap}, applied to
$C = \calB$ and $A = \kk$.)
\item The neutral element
$\underbrace{\eta_\kk}_{=\id} \epsilon_\calB = \epsilon_\calB$
of $\oast$ belongs to $\Xi(\calB)$.
\end{itemize}
Thus, Proposition~\ref{prop.gen3-mon} is proved.
\end{proof}

The convolution monoid $\Xi(\calB)$ is not always a group.
%Example \ref{infiltr} shows a situation where it is not.

\begin{example}
Fix $q \in \kk$.
Let $\calB$ be the univariate $q$-infiltration bialgebra
$\tup{\kx, \Delta_{\uparrow_q}, \epsilon}$
from Example~\ref{exa.grouplike.q-pol}.
% Here, the $q$-infiltration coproduct $\Delta_{\uparrow_q}$ is defined by
% \[
% \Delta_{\uparrow_q}(x) := x\otimes 1+ 1\otimes x + q x\otimes x ,
% \]
% and the counit $\epsilon$ is the usual one (sending $x$ to $0$),
% whereas $\conc$ denotes the usual multiplication of the polynomial
% ring $\kk[x]$ (a particular case of concatenation of words).
% Note that this is a Hopf algebra if and only if $q$ is nilpotent.
It is easy to see (by induction on $m$) that
\begin{equation}
\Delta_{\uparrow_q}(x^m)
= \sum_{\substack{\tup{i,j,k} \in \NN^3;\\ i+j+k=m}}\dbinom{m}{i,j,k}q^jx^{i+j}\otimes x^{j+k}
\label{eq.infiltr.Deltaxm.1}
\end{equation}
holds in $\calB$ for every $m \in \NN$,
where $\dbinom{m}{i,j,k}$ denotes the multinomial coefficient
$\dfrac{m!}{i!j!k!}$.
(This is actually a particular case of a beautiful formula that holds
for any multivariate $q$-infiltration algebra\footnote{Namely,
\begin{equation}
\Delta_{\uparrow_q}(w)=\sum_{I\cup J=\set{1,2,\ldots,m}}q^{|I\cap J|}w[I]\otimes w[J]
\end{equation}
for any word $w$ of length $m$.
See ``Shuffles, stuffles and other dual laws'' in
\url{https://mathoverflow.net/questions/214927/}
% %\url{https://mathoverflow.net/questions/214927/important-formulas-in-combinatorics/216201#216201}
for details.}.)

As a $\kk$-module, the dual $\calB^\vee$ of $\calB$ can be identified with
the $\kk$-module $\kk[[x]]$ of formal power series in $x$ over $\kk$
(by equating $x^j \in \kk[[x]]$ with the $\kk$-linear form on $\calB$
that sends $x^j \in \calB$ to $1$ and sends any other power of $x$ in
$\calB$ to $0$).
We denote the convolution product $\oast$ on the dual algebra $\calB^\vee$
by $\uparrow_q$.
Thus, $\uparrow_q$ is a binary operation on $\calB^\vee = \kk[[x]]$.

The monoid of characters of $\calB$ is $\Xi(\calB)=\{(\alpha x)^*\}_{\alpha\in \kk} $,
where we are using the Kleene star notation $s^* = \dfrac{1}{1-s} = 1 + s + s^2 + \cdots$
(the multiplication being that of $\kk[[x]]$).
For any $\alpha, \beta \in \kk$, we have
\begin{align*}
(\alpha x)^*\uparrow_q (\beta x)^*
&=
\sum_{m\geq 0}\scal{(\alpha x)^*\uparrow_q (\beta x)^*}{x^m}\,x^m \\
& = \sum_{m\geq 0}\scal{(\alpha x)^*\otimes (\beta x)^*}{\Delta_{\uparrow_q}(x^m)}\,x^m
\\
&=
\sum_{m\geq 0} \ \ \sum_{\substack{\tup{i,j,k} \in \NN^3;\\ i+j+k=m}} \dbinom{m}{i,j,k}q^j
\scal{(\alpha x)^*\otimes (\beta x)^*}{x^{i+j}\otimes x^{j+k}}x^m
\\
& \qquad \qquad \tup{\text{by \eqref{eq.infiltr.Deltaxm.1}}} \\
&= \sum_{m\geq 0}x^m \ \ \sum_{\substack{\tup{i,j,k} \in \NN^3;\\ i+j+k=m}} \dbinom{m}{i,j,k} q^j \alpha^{i+j}\beta^{j+k}
=\sum_{m\geq 0}x^m (q\alpha\beta+\alpha+\beta)^m
\\
&= \big((q\alpha\beta+\alpha+\beta)x\big)^* .
\end{align*}
In particular, when $q=0$, we recover the law of composition of characters for the shuffle product:
\[
(\alpha x)^*\shuffle (\beta x)^*=\big((\alpha+\beta) x\big)^* ;
\]
thus, the monoid of characters $\Xi(\calB)$ is (in this case) isomorphic with the additive group $(\kk,+,0)$.\\
When $q \in \kk^\times$, the monoid of characters $\Xi(\calB)$ is isomorphic with the multiplicative monoid $(\kk,\cdot,1)$ through $(\alpha x)^*\mapsto q\alpha+1$;
in particular, it has a zero element $(\alpha x)^*$ with $\alpha = -\dfrac{1}{q}$.
\end{example}

\begin{theorem}
\label{thm.gen3}
Let $\calB$ %$(\calB,.,1_\calB,\Delta,\epsilon)$
be a $\kk$-bialgebra.
As usual, let $\Delta = \Delta_\calB$ and $\epsilon = \epsilon_\calB$
be its comultiplication and its counit.

Let $\calB_+=\ker(\epsilon)$. For each $N \geq 0$, let $\calB_+^N = \underbrace{\calB_+ \cdot \calB_+ \cdot \cdots \cdot \calB_+}_{N\text{ times}}$, where $\calB_+^{0} = \calB$. Note that $\left(\calB_+^0, \calB_+^1, \calB_+^2, \ldots\right)$ is called the \emph{standard decreasing filtration} of $\calB$.

For each $N \geq -1$, we define a $\kk$-submodule $\calB_N^\vee$ of $\calB^\vee$ by
\begin{equation}
\calB_N^\vee = (\calB_+^{N+1})^\perp
= \left\{ f \in \calB^\vee \mid f\left(\calB_+^{N+1}\right) = 0 \right\} .
\end{equation}  
Thus, $\left(\calB_{-1}^\vee, \calB_0^\vee, \calB_1^\vee, \ldots\right)$
is an increasing filtration of $\calB_\infty^\vee := \bigcup_{N\geq -1}\calB_N^\vee$ with $\calB_{-1}^\vee = 0$.
Then:

\begin{enumerate}

\item[\textbf{(a)}] We have $\calB_p^\vee\oast \calB_q^\vee\subseteq \calB_{p+q}^\vee$ for any $p, q \geq -1$ (where we set $\calB_{-2}^\vee = 0$).
Hence, $\calB_\infty^\vee$ is a subalgebra of the convolution algebra $\calB^\vee$.

\item[\textbf{(b)}] Assume that $\kk$ is an integral domain.
Then, the set $\Xi(\calB)^\times$ of invertible characters (i.e., of invertible elements of the monoid $\Xi(\calB)$ from Proposition~\ref{prop.gen3-mon}) is left $\calB_\infty^\vee$-linearly independent.
\end{enumerate}
\end{theorem}

\begin{proof}[Proof of Theorem \ref{thm.gen3}.]
The map $\epsilon : \calB \to \kk$ is a $\kk$-algebra homomorphism;
hence, its kernel $\calB_+$ is an ideal of $\calB$.
Thus, $\calB_+ = \calB \calB_+$, so that
\begin{equation}
\calB_+^N = \calB_+^{N-1} \calB_+
\qquad \text{for each $N \geq 1$}.
\label{pf.thm.gen3.idealpowers}
\end{equation}
Let us define a left action $\triangleright$ of the
$\kk$-algebra $\calB$ on $\calB^\vee$ %(the full dual)
by setting
\[
\scal{u\triangleright f}{v}
%=\scal{fu^{-1}}{v}
=\scal{f}{vu}
\qquad \text{ for all $f\in \calB^\vee$ and $u,v\in \calB$.}
\]
Here, $\scal{g}{b}$ means $g\left(b\right)$ whenever $g \in \calB^\vee$ and $b \in \calB$.
Thus, $\calB^\vee$ is a left $\calB$-module.
For a given $u \in \calB$, we shall refer to the operator
$\calB^\vee \to \calB^\vee,\ f \mapsto u \triangleright f$ as
\emph{shifting by $u$} or the \emph{$u$-left shift operator};
it generalizes Sch\"utzenberger's right $u^{-1}$ in automata theory \cite{Berstel,DFA}.

In the following, we shall use a variant of Sweedler
notation: Given any $u \in \calB$,
instead of writing $\sum_{(u)}u_1\otimes u_2$ for
$\Delta(u)$, we will write
$\sum_{(u)}u^{(1)}\otimes u^{(2)}$ for
$\Delta(u) - u \otimes 1 - 1 \otimes u$.
Thus,
\begin{equation}
\Delta(u)=u\otimes 1+1\otimes u+\sum_{(u)}u^{(1)}\otimes u^{(2)}
\label{pf.thm.gen3.sweedler}
\end{equation}
for each $u \in \calB$.
Moreover, if $u \in \calB_+$, then all of the
$u^{(1)}$ and $u^{(2)}$ can be chosen to belong to
$\calB_+$ themselves (because it is easy to check
that $\Delta\left(u\right) - u \otimes 1 - 1 \otimes u
=
\left( \left(\id - \eta\epsilon\right) \otimes \left(\id - \eta\epsilon\right) \right)
\left( \Delta\left(u\right) \right)
\in
\left( \left(\id - \eta\epsilon\right) \otimes \left(\id - \eta\epsilon\right) \right)
\left( \calB \otimes \calB \right)
= \calB_+ \otimes \calB_+$,
since $\left(\id - \eta\epsilon\right)\left(\calB\right) = \calB_+$).
We shall understand, in the following, that we choose
$u^{(1)}$ and $u^{(2)}$ from $\calB_+$ when $u \in \calB_+$.

The following two lemmas give simple properties of
the left action $\triangleright$:

\begin{lemma} \label{lem.der_with_heart}
Let $f_1,f_2\in \calB^\vee$ and $u\in \calB_+=\ker(\epsilon)$.
Then,
\begin{equation}\label{der_with_heart}
u\triangleright(f_1\oast f_2)
=(u\triangleright f_1)\oast f_2+f_1\oast (u\triangleright f_2)+
\sum_{(u)} (u^{(1)}\triangleright f_1)\oast (u^{(2)}\triangleright f_2) .
\end{equation}
\end{lemma} 
\begin{proof}
Immediate by direct computation. 
\end{proof}

\begin{lemma}
\label{lem.left-act.lower}
Let $k \geq 0$, and let $f \in \calB^\vee_k$ and $u \in \calB_+$.
Then,
$u \triangleright f \in \calB^\vee_{k-1}$.
\end{lemma}
\begin{proof}
Easy using \eqref{pf.thm.gen3.idealpowers}.
\end{proof}

Let us now proceed with the proof of Theorem~\ref{thm.gen3}.
%\begin{proof}[Proof of Theorem~\ref{thm.gen3} part 2)]

\textbf{(a)}
We shall first show that
$\calB_p^\vee \oast \calB_q^\vee \subseteq \calB_{p+q}^\vee$
for any $p, q \geq -1$.

Indeed, we proceed by strong induction on $p+q$.
Let $f_1\in\calB_p^\vee$ and $f_2\in\calB_q^\vee$.
We intend to show that $f_1\oast f_2\in \calB_{p+q}^\vee$.
This is trivial if $p$ or $q$ is $-1$, since $\calB_{-1}^\vee = 0$.
Thus, we WLOG assume $p,q\geq 0$.

Hence, $\calB_+^{p+q+1} = \calB_+^{p+q} \calB_+$ by \eqref{pf.thm.gen3.idealpowers}.
Thus, the $\kk$-module $\calB_+^{p+q+1}$ is spanned by the products $uv$ for $u\in \calB_+^{p+q}$ and $v\in \calB_+$. Thus, in order to prove that $f_1\oast f_2 \in \calB_{p+q}^\vee$, it suffices to show that $f_1\oast f_2$ is orthogonal to all these products.

So let $u\in \calB_+^{p+q}$ and $v\in \calB_+$. Then, we must prove that $\scal{f_1\oast f_2}{uv} = 0$. But we have
\begin{align}
\scal{f_1\oast f_2}{uv}
&=\scal{v\triangleright(f_1\oast f_2)}{u} .
\label{test1}
\end{align}  
Applying \eqref{der_with_heart} to $v$ instead of $u$, we get
\begin{align}
v \triangleright(f_1\oast f_2)
=(v\triangleright f_1)\oast f_2+f_1\oast (v\triangleright f_2)+
\sum_{(v)} (v^{(1)}\triangleright f_1)\oast (v^{(2)}\triangleright f_2)
\label{test2}
\end{align}
with all $v^{(1)}$ and $v^{(2)}$ lying in $\calB_+$.
Thus, Lemma~\ref{lem.left-act.lower} yields
$v\triangleright f_1 \in \calB_{p-1}^\vee$ and
$v\triangleright f_2 \in \calB_{q-1}^\vee$ and
$v^{(1)}\triangleright f_1 \in \calB_{p-1}^\vee$ and
$v^{(2)}\triangleright f_2 \in \calB_{q-1}^\vee$.
Thus, \eqref{test2} becomes
\begin{align*}
v \triangleright(f_1\oast f_2)
&=
\underbrace{(v\triangleright f_1)}_{\in \calB_{p-1}^\vee}
\oast \underbrace{f_2}_{\in \calB_q^\vee}
+
\underbrace{f_1}_{\in \calB_p^\vee}
\oast \underbrace{(v\triangleright f_2)}_{\in \calB_{q-1}^\vee}
+
\sum_{(v)}
\underbrace{(v^{(1)}\triangleright f_1)}_{\in \calB_{p-1}^\vee}
\oast \underbrace{(v^{(2)}\triangleright f_2)}_{\in \calB_{q-1}^\vee}
\\
&\in \calB_{p-1}^\vee \oast \calB_q^\vee
+ \calB_p^\vee \oast \calB_{q-1}^\vee
+ \calB_{p-1}^\vee \oast \calB_{q-1}^\vee
\subseteq \calB_{p+q-1}^\vee
\end{align*}
(by the induction hypothesis, applied three times).
This entails
$\scal{v\triangleright(f_1\oast f_2)}{u} = 0$
by the definition of $\calB_{p+q-1}^\vee$.
Thus, \eqref{test1} yields $\scal{f_1\oast f_2}{uv}=0$.
Thus, we have proved that $f_1\oast f_2 \in \calB_{p+q}^\vee$.
This completes the proof of $\calB_p^\vee \oast \calB_q^\vee \subseteq \calB_{p+q}^\vee$.

The fact that $\calB_\infty^\vee$ is a subalgebra of the convolution algebra $\calB^\vee$ follows from the preceding, since we also know that $\epsilon \in \calB_0^\vee \subseteq \calB_\infty^\vee$.

\smallskip 
\textbf{(b)} We define the \emph{degree}
% \footnote{The name has been chosen to stress the analogy with the least degree-upper bound in Section~\ref{sect.dg2}.}
of an element $b\in \calB_\infty^\vee$ to be the least index $d\geq -1$ such that $b\in \calB^\vee_{d}$. We denote this index by $\deg(b)$. (Note that $\deg(0)=-1$.)

Recall that for $b\in \calB_d^\vee$  and $u\in \calB_+$, we have $u\triangleright b\in \calB_{d-1}^\vee$ (by Lemma~\ref{lem.left-act.lower}).
Thus, for $b\in \calB_\infty^\vee$  and $u\in \calB_+$, we have $u\triangleright b\in \calB_\infty^\vee$.
In other words, $\calB_+ \triangleright\calB_\infty^\vee \subseteq \calB_\infty^\vee$.

Let us consider non-trivial relations of the form
\begin{equation}\label{non_triv_rel1}
\sum_{g\in F} p_g\oast g=0
\end{equation}
with a finite subset $F$ of $\Xi(\calB)^\times$ and with nonzero coefficients $p_g \in \calB_\infty^\vee$.
If there are no such relations with $F\neq \emptyset$, then we are done.
Otherwise, we pick one such relation of type \eqref{non_triv_rel1} with $|F|\neq 0$ minimal;
among all such minimum-size relations, we pick one in which $\sum_{g \in F} \deg(p_g)$ is minimal.
WLOG, we assume that $\epsilon\in F$ (otherwise, pick $g_0\in F$ and multiply both sides of the equation (\ref{non_triv_rel1}) by $g_0^{\oast-1}$).
% Thus, \eqref{non_triv_rel1} becomes
% \begin{equation}\label{non_triv_rel2}
% p_\epsilon+\sum_{\substack{g\in F;\\ g\neq \epsilon}} p_g\oast g=0 .
% \end{equation}   

It is impossible that $F=\{\epsilon\}$ (as $p_\epsilon\neq 0$); thus, we can choose $g_1\in F\setminus \{\epsilon\}$. Having chosen this $g_1$, we observe that $g_1(\calB_+) \neq 0$ (since $g_1(\calB_+)=0$ would imply $g_1 = \epsilon$ because of $g_1\in \Xi(\calB)$); in other words, there exists some $u\in \calB_+$ such that $\scal{g_1}{u}\neq 0$. Choose such a $u$.

It is easy to see (from the definition) that each character $g \in \Xi(\calB)$ and each $v\in \calB_+$ satisfy
\begin{equation}
v\triangleright g=\scal{g}{v}\,g .
\label{pf.thm.gen3.vong}
\end{equation}

Our plan is now to shift
both sides of (\ref{non_triv_rel1}) by $u$, and rewrite the
resulting equality again as an equality of the form
\eqref{non_triv_rel1} (with new values of $p_g$).
To do so, we introduce a few notations.

Write $\Delta\left(u\right)$ as in \eqref{pf.thm.gen3.sweedler},
with $u^{(1)}, u^{(2)} \in \calB_+$.
For each $g \in F$, let us set
\[
p'_g :=
u\triangleright p_g+ \scal{g}{u}p_g+
\sum_{(u)} \scal{g}{u^{(2)}}(u^{(1)}\triangleright p_g )
\in \calB^\vee .
\]
Then, we have
\begin{align}
u \triangleright (p_g \oast g)
= p'_g \oast g
\label{pf.thm3.gen.pg1}
\end{align}
for each $g \in F$, because
Lemma~\ref{lem.der_with_heart} yields
\begin{align*}
& u \triangleright (p_g \oast g) \\
&= (u\triangleright p_g)\oast g+
p_g\oast (u\triangleright g)+\sum_{(u)} (u^{(1)}\triangleright p_g)\oast (u^{(2)}\triangleright g)  \\
&= (u\triangleright p_g)\oast g+
p_g\oast \left(\scal{g}{u} g\right) +
\sum_{(u)} (u^{(1)}\triangleright p_g)\oast \left( \scal{g}{u^{(2)}} g \right)
\qquad \left(\text{by \eqref{pf.thm.gen3.vong}}\right) \\
&= \underbrace{\Big(u\triangleright p_g+ \scal{g}{u}p_g+
\sum_{(u)} \scal{g}{u^{(2)}}(u^{(1)}\triangleright p_g)\Big)}_{= p'_g} \oast\, g
\\
&= p'_g \oast g .
\end{align*}

Thus, if we shift both sides of (\ref{non_triv_rel1}) by $u$, we find
\[
u \triangleright \left(\sum_{g \in F} p_g\oast g \right) = 0.
\]
Hence,
\begin{align}
0
= u \triangleright \left( \sum_{g \in F} p_g\oast g \right)
= \sum_{g \in F} u\triangleright (p_g\oast g)
= \sum_{g \in F} p'_g \oast g
\label{pf.thm3.gen.pg-0}
\end{align}
(by \eqref{pf.thm3.gen.pg1}).

For each $g \in F$, we have $p'_g \in \calB^\vee_\infty$
(by the definition of $p'_g$, since $p_g \in \calB^\vee_\infty$
and $\calB_+ \triangleright\calB_\infty^\vee \subseteq \calB_\infty^\vee$).

It is easy to see that each $g \in F$ satisfies
\begin{align}
\deg\left(p'_g\right) \leq \deg\left(p_g\right) .
\label{pf.thm3.gen.pg-deg}
\end{align}
(Indeed, Lemma~\ref{lem.left-act.lower} shows that
any nonzero $f \in \calB^\vee_\infty$ satisfies
$\deg\left(u \triangleright f\right) < \deg f$;
for the same reason,
$\deg\left(u^{(1)} \triangleright f\right) < \deg f$.
Thus, the definition of $p'_g$ represents $p'_g$ as
a $\kk$-linear combination of elements of $\calB^\vee_\infty$
having the same degree as $p_g$ or smaller degree.
This proves \eqref{pf.thm3.gen.pg-deg}.)

But the definition of $p'_\epsilon$ easily yields
\begin{align}
p'_\epsilon = u \triangleright p_\epsilon
\label{pf.thm3.gen.pgeps}
\end{align}
(since $u, u^{(2)} \in \calB_+$ entail
$\scal{\epsilon}{u}=0$ and $\scal{\epsilon}{u^{(2)}}=0$).
Therefore,
$\deg\left(p'_\epsilon\right)
= \deg\left(u \triangleright p_\epsilon\right)
< \deg\left(p_\epsilon\right)$
(since Lemma~\ref{lem.left-act.lower} shows that
any nonzero $f \in \calB^\vee_\infty$ satisfies
$\deg\left(u \triangleright f\right) < \deg f$).

Now, recall the equality \eqref{pf.thm3.gen.pg-0}.
This equality (once rewritten as
$\sum_{g \in F} p'_g \oast g = 0$) has the same
structure as \eqref{non_triv_rel1} (since
$p'_g \in \calB^\vee_\infty$ for each $g \in F$),
and uses the same set $F$, but it satisfies
$\sum_{g \in F} \deg\left(p'_g\right)
< \sum_{g \in F} \deg\left(p_g\right)$
(because each $g \in F$ satisfies
\eqref{pf.thm3.gen.pg-deg}, but the specific
character $g = \epsilon$ satisfies
$\deg\left(p'_\epsilon\right) < \deg\left(p_\epsilon\right)$).
Thus, if the coefficients $p'_g$ in
\eqref{pf.thm3.gen.pg-0} are not all $0$, then we
obtain a contradiction to our choice of relation (which was
to minimize $\sum_{g \in F} \deg\left(p_g\right)$).
Hence, all coefficients $p'_g$ must be $0$.
In particular, we must have $p'_{g_1} = 0$.
Therefore,
\[
0 = p'_{g_1}
= u\triangleright p_{g_1}+ \scal{g_1}{u}p_{g_1} +
\sum_{(u)} \scal{g_1}{u^{(2)}}(u^{(1)}\triangleright p_{g_1} )
\]
(by the definition of $p'_{g_1}$), so that
\[
\scal{g_1}{u}p_{g_1}
= - u\triangleright p_{g_1}
- \sum_{(u)} \scal{g_1}{u^{(2)}}(u^{(1)}\triangleright p_{g_1} )
\in \calB^\vee_{\deg\left(p_{g_1}\right) - 1}
\]
(by Lemma~\ref{lem.left-act.lower}, since
$u, u^{(2)} \in \calB_+$ and
$p_{g_1} \in \calB^\vee_{\deg\left(p_{g_1}\right)}$).
In other words, $\scal{g_1}{u}p_{g_1}$ is orthogonal
to $\calB_+^{\deg\left(p_{g_1}\right)}$.
Since\footnote{We recall that
%, having remarked that $g_1(\calB_+) \neq 0$
we have chosen
$u\in \calB_+$ such that $\scal{g_1}{u} \neq 0$.}
$\scal{g_1}{u} \neq 0$, this entails that
$p_{g_1}$ is orthogonal
to $\calB_+^{\deg\left(p_{g_1}\right)}$ as well
(since its target $\kk$ is an integral domain).
This means that $p_{g_1} \in \calB^\vee_{\deg\left(p_{g_1}\right) - 1}$,
or, equivalently, $\deg\left(p_{g_1}\right)
\leq \deg\left(p_{g_1}\right) - 1$.
But this is absurd.
This is the contradiction we were looking for.
Thus, Theorem~\ref{thm.gen3} \textbf{(b)} is proved.
\end{proof}

\begin{remark}
\begin{enumerate}
\item[i)] The invertible characters in $\Xi(\calB)^{\times}$ are also right $\calB_\infty^\vee$-linearly independent. This can be proven similarly, using right shifts in the proof.

\item[ii)] The reader should beware of supposing that the standard decreasing filtration is necessarily Hausdorff\footnote{``Separated'' in \cite{Bourbaki-CA}.} (i.e., satisfies
$\bigcap_{n\ge0}\,\calB_+^{n} = \{0\}$).
A counterexample can be obtained by taking the universal enveloping bialgebra of any simple Lie algebra (or, more generally, of any perfect Lie algebra); it will satisfy
$\bigcap_{n\ge0}\,\calB_+^{n} = \calB_+$.

\item[iii)] The property (i.e., the linear independence) does not hold if we consider the set of all characters (that is, $\Xi(\calB)$) instead of $\Xi(\calB)^{\times}$. 
For example, let $\calB$ be the
univariate $q$-infiltration bialgebra
$\tup{\kx, \Delta_{\uparrow_q}, \epsilon}$
from Example~\ref{exa.grouplike.q-pol}, and
assume that $q \in \kk$ is invertible.
Define two $\kk$-linear maps
$f : \calB \to \kk$ and $g : \calB \to \kk$ by
\[
f\tup{x^n} = \delta_{n,1}
\qquad \text{ and } \qquad
g\tup{x^n} = \left(-\dfrac{1}{q}\right)^n
\qquad \text{ for all } n \in \NN
\]
(where the $\delta_{n,1}$ is a Kronecker delta).
Note that $g$ is a character of $\calB$, while
$f \in \calB_1^\vee \subseteq \calB_\infty^\vee$.

We claim that $f \oast g = 0$.
In fact,
using \eqref{eq.infiltr.Deltaxm.1}, it is straightforward to see
that $\tup{f \oast g}\tup{x^m} = 0$ for each $m \in \NN$;
thus, $f \oast g = 0$.
Since $f$ is nonzero, this shows that
the set $\Xi(\calB)$ is not left $\calB_\infty^\vee$-linearly independent.
% \tcb{we have}
% \begin{equation}\label{zero_div}
% x\oast (-\frac{1}{q}x)^*=x\uparrow_q (-\frac{1}{q}x)^*=0 .
% \end{equation}
% \tcb{where $(-\frac{1}{q}x)^*$ stands for the linear form such that 
% $\scal{(-\frac{1}{q}x)^*}{x^n}=(-\frac{1}{q})^n$.\\ 
% Indeed, for $n\ge1$
% \begin{eqnarray*}
% &&\scal{x\uparrow_q (-\frac{1}{q}x)^*}{x^n}=\scal{x\otimes (-\frac{1}{q}x)^*}{\Delta_{\uparrow}(x^n)}=\\
% &&\scal{x\otimes (-\frac{1}{q}x)^*}{(x\otimes 1+ 1\otimes x + q x\otimes x)\Delta_{\uparrow}(x^{n-1})}=0 .
% \end{eqnarray*}
% as, from \eqref{WordInf}, we have $x\uparrow_q x^n=(n+1)x^{n+1}+qnx^{n}$.\\ 
% With $\scal{x\uparrow_q (-\frac{1}{q}x)^*}{1_{X^*}}=0$, this completes the proof of \eqref{zero_div}.}
\end{enumerate}
\end{remark}

%\begin{remark}
%\tcr{D4.[Is any of this remark still needed after the new
%Section 5.1?]}
%While the dual $C^\vee$ of a $\kk$-coalgebra $C$ is always
%a $\kk$-algebra,
%there is no similarly straightforward process to turn the
%dual $A^\vee$ of a $\kk$-algebra $A$ into a $\kk$-coalgebra.
%The reason is that the canonical map
%\[
%M^\vee\otimes_{\kk} N^\vee\longrightarrow (M\otimes_{\kk} N)^\vee
%\]
%may not be injective (even if $\kk$ is an integral domain;
%see the discussion in section \ref{subsect.dual-tensor}).
%This can be ameliorated %\textit{in any case}
%by considering
%the tensor product of linear forms \textit{as linear mappings}.
%
%Moreover, the usual algebraic dualization itself may
%not separate the points of the module \tcb{(even if it is torsionless)}.
%For example, if we consider, with $X$ non void, the free algebra $\calA=\ncp{\Q}{X}$ as a $\Z$-algebra, the module 
%\[
%Hom_\Z(\calA,\Z)=\{0\}.
%\]      
%%\tcb{maybe $\Hom_\kk(\calA,Frac(\kk))$ would do the job ?}
%\tcr{D5.[I've left this as it stands for now, since I don't see
%the big picture well enough.]}
%\end{remark}

Note that our above proof of Theorem~\ref{thm.gen3} is somewhat
similar to the standard (Artin) proof of the linear independence
of characters in Galois theory (\cite[proof of Theorem 12]{Artin71}).
Shifting by $u$ in the former proof corresponds to replacing $x$
by $\alpha x$ in the latter.

\begin{corollary}\label{MonMapCor}
We suppose that $\calB$ is cocommutative, and $\kk$ is an integral domain.\\ 
Let $(g_x)_{x\in X}$ be a family of elements of $\Xi(\calB)^\times$ (the set of invertible characters of $\calB$), and let $\varphi_X: \kk[X]\to (\calB^\vee,\oast,\epsilon)$ be the $\kk$-algebra morphism that sends each $x \in X$ to $g_x$. In order for the family $(g_x)_{x\in X}$ (of elements of the commutative ring $(\calB^\vee,\oast,\epsilon)$) to be algebraically independent over the subring $(\calB_\infty^\vee,\oast,\epsilon)$, it is necessary and sufficient that the monomial map 
\begin{align}
m: \N^{(X)} &\to (\calB^\vee,\oast,\epsilon) , \nonumber\\
\alpha &\mapsto \varphi_X(X^\alpha) = \prod_{x \in X} g_x^{\alpha_x}
\label{MonMap}
\end{align}
(where $\alpha_x$ means the $x$-th entry of $\alpha$)
be injective.
\end{corollary}
\begin{proof}
Indeed, as $\calB$ is cocommutative (and $\kk$ commutative), $(\calB^\vee,\oast,\epsilon)$ is a commutative algebra (see, e.g., \cite[Exercise 1.5.5]{GriRei} or \cite[Chapter III, \S 11.2, Proposition 2]{Bourbaki-Alg1}).
Thus, the algebraic independence of the family $(g_x)_{x\in X}$ is equivalent to the statement that the family $\left(\prod_{x \in X} g_x^{\alpha_x}\right)_{\alpha\in \N^{(X)}}$ be $\calB_\infty^\vee$-linearly independent.
In other words, it is equivalent to the claim that the family 
$(m(\alpha))_{\alpha\in \N^{(X)}}$ (with $m$ as in \eqref{MonMap}) be $\calB_\infty^\vee$-linearly independent
(since $m(\alpha) = \prod_{x \in X} g_x^{\alpha_x}$ for each $\alpha \in \NN^{(X)}$).
It suffices to remark that the elements $m(\alpha)$ are invertible characters (since they are products of invertible characters).
Therefore, in view of Theorem~\ref{thm.gen3} \textbf{(b)}, the $(m(\alpha))_{\alpha\in \N^{(X)}}$ are $\calB_\infty^\vee$-linearly independent if they are distinct; but this amounts to saying that $m$ is injective.
\end{proof}

\begin{example}
\begin{enumerate}
\item[i)]
Let $\kk$ be an integral domain, and let us consider the bialgebra $\calB=\tup{\kk[x], \Delta, \epsilon}$ from Example~\ref{exa.grouplike.pol} (the standard univariate polynomial bialgebra).
As it is a particular case of the situation of the free algebra\footnote{See \cite[Proposition 1.6.7]{GriRei} and \cite[Section 1.4]{FLA}.}, we will let $\shuffle$ denote the convolution $\oast$ on its dual $\kk$-module $\calB^\vee = \kk[x]^\vee \cong \kk[[x]]$; thus, $\tup{\kk[[x]],\shuffle,1}$ becomes a commutative $\kk$-algebra.

For every $\alpha\in \kk$, there exists only one character of $\kk[x]$ sending 
$x$ to $\alpha$; we will denote this character by $(\alpha.x)^*\in \kk[[x]]$ (see \cite{Eil,KSOP,DucTol09,DGM2} for motivations about this notation).
Thus, $\Xi\tup{\calB} = \set{(\alpha.x)^* \mid \alpha \in \kk}$.
It is easy to check that $(\alpha.x)^*\shuffle (\beta.x)^*=\bigl((\alpha+\beta).x\bigr)^*$ for any $\alpha, \beta \in \kk$.
Thus, any $c_1, c_2, \ldots, c_k \in \kk$ and any $\alpha_1, \alpha_2, \ldots, \alpha_k \in \NN$ satisfy
\begin{align}
&\tup{\tup{c_1.x}^*}^{\shuffle \alpha_1}
\shuffle
\tup{\tup{c_2.x}^*}^{\shuffle \alpha_2}
\shuffle
\cdots
\shuffle
\tup{\tup{c_k.x}^*}^{\shuffle \alpha_k}
\nonumber \\
&= \bigl((\alpha_1 c_1 + \alpha_2 c_2 + \cdots + \alpha_k c_k).x\bigr)^*\ .
\label{shuffle_multiindex1}
\end{align}
The monoid $\Xi\tup{\calB}$ is thus isomorphic to the abelian group $\tup{\kk,+,0}$; in particular, it is a group, so that $\Xi\tup{\calB}^\times = \Xi\tup{\calB}$.
 
% $(c_i)_{i\in I}$ of coefficients in $\kk$ and $\alpha\in \N^{(I)}$ being given, we get that, for all  
% $\{i_1,\cdots ,i_k\}\supset supp(\alpha)$, 
% \begin{equation}\label{shuffle_multiindex1}
% \shuffle_{j=1}^k \bigl((c_{i_j}.x)^*\bigr)^{\shuffle\alpha(i_j)}=\bigl((\sum_{j=1}^k\alpha(i_j)c_j).x\bigr)^*\ .
% \end{equation}
The decreasing filtration of $\calB$ is given by $\calB_+^n=\kk[x]_{\geq n}$ (the ideal of polynomials of degree $\geq n$).
Hence, the reader may check easily that $\calB_n^\vee=\kk[x]_{\leq n}$ (the module of polynomials of degree $\leq n$), whence 
$\calB_\infty^\vee=\kk[x]$.

Now, let $\left((c_i.x)^*\right)_{i \in I}$ be a family of elements of $\Xi\left(\calB\right)^\times$.
Taking $X = I$ and $g_i = (c_i.x)^*$ for each $i \in I$,
we can then apply Corollary~\ref{MonMapCor}, and we conclude that
the family $\left((c_i.x)^*\right)_{i \in I}$ of elements of the power series ring $\kk[[x]]$
is algebraically independent over the subring $(\calB_\infty^\vee,\oast,\epsilon) = \kk[x]$
if and only if the monomial map \eqref{MonMap} is injective.

But \eqref{shuffle_multiindex1} shows that the monomial map \eqref{MonMap} is injective if and only if the family $(c_i)_{i\in I}$ is $\Z$-linearly independent in $\kk$.

To illustrate this, take $\kk=\overline{\Q}$ (the algebraic closure of $\Q$) and $c_n=\sqrt{p_n}\in \kk$, where $p_n$ is the 
$n$-th prime number.  
 Our above observations show that the family of series $\big((\sqrt{p_n}.x)^*\big)_{n\ge1}$ is algebraically independent over the polynomials (i.e., over $\overline{\Q}[x]$) within the commutative $\overline{\Q}$-algebra 
 $\tup{\overline{\Q}[[x]],\shuffle,1}$ (since the family $(\sqrt{p_n})_{n\ge1}$ is $\Z$-linearly independent\footnote{This holds even more generally if the $p_n$ are squarefree positive integers rather than primes. See \url{https://math.stackexchange.com/questions/30687/} for proofs and references.}).
 This example can be double-checked using partial fractions decompositions as, in fact, $(\sqrt{p_n}.x)^*=\dfrac{1}{1-\sqrt{p_n}x}$ (this time, the inverse is taken within the ordinary product in $\kk[[x]]$) and 
\[
\Bigl(\dfrac{1}{1-\sqrt{p_n}x}\Bigr)^{\shuffle n}=\dfrac{1}{1-n\sqrt{p_n}x} .
\]
%Todo detailler

\item[ii)] The preceding example can be generalized as follows:
Let $\kk$ still be an integral domain; let $V$ be a $\kk$-module, and let $\calB=\tup{T(V),\conc,1_{T(V)},\Delta_{\boxtimes},\epsilon}$ be the standard tensor conc-bialgebra\footnote{The one defined by 
\[
\Delta_{\boxtimes}(1)=1\otimes 1 \mbox{ and } \Delta_{\boxtimes}(u)=u\otimes 1 + 1\otimes u\ ;\ 
\epsilon(u)=0\mbox{ for all }u\in V.
\]
}. 
%\tcr{D6.[I don't see what you are referencing here.]} and \cite{FLA} section 1.4, \cite{Bourbaki-Lie1} Ch II \S 2 for the construction when $V=\kk^{(X)}$.). The dual algebra is then $\calB^\vee=\tup{T(V)^\vee,\shuffle,\epsilon}$. 
For every linear form $\varphi\in V^\vee$, there is a unique character $\varphi^*$ of 
$\tup{T(V),\conc,1_{T(V)}}$ such that all $u\in V$ satisfy
\begin{equation}
\scal{\varphi^*}{u}=\scal{\varphi}{u} .
\end{equation} 
Again, it is easy to check\footnote{For this bialgebra $\shuffle$ stands for $\oast$ on the space $\Hom(\calB, \kk)$.} that 
$(\varphi_1)^*\shuffle (\varphi_2)^*=\bigl(\varphi_1+\varphi_2\bigr)^*$ for any $\varphi_1, \varphi_2 \in V^\vee$, because, from Lemma~\ref{lem.BtoA.bimap*bimap}, both sides are characters of $\tup{T(V),\conc,1_{T(V)}}$ so that the equality has only to be checked on $V$. Again, from this, any 
$\varphi_1, \varphi_2, \ldots, \varphi_k \in V^\vee$ and any $\alpha_1, \alpha_2, \ldots, \alpha_k \in \NN$ satisfy
\begin{align}
& \tup{\tup{\varphi_1}^*}^{\shuffle \alpha_1}
  \shuffle
  \tup{\tup{\varphi_2}^*}^{\shuffle \alpha_2}
  \shuffle
  \cdots
  \shuffle
  \tup{\tup{\varphi_k}^*}^{\shuffle \alpha_k}
\nonumber\\
&= \bigl(\alpha_1 \varphi_1 + \alpha_2 \varphi_2 + \cdots + \alpha_k \varphi_k \bigr)^*\ .
\label{shuffle_multiindex2}
\end{align}
The decreasing filtration of $\calB$ is given by $\calB_+^n=\bigoplus_{k\geq n}T_k(V)$ (the ideal of tensors of degree $\geq n$) and the reader may check easily that, in this case, $\calB_\infty^\vee$ is the shuffle algebra of finitely supported linear forms -- i.e., for each $\Phi \in \calB^\vee$, we have the equivalence
\[
\Phi\in \calB_\infty^\vee \Longleftrightarrow (\exists N\in \N)(\forall k\geq N)(\Phi(T_k(V))=\{0\}) .
\] 
Then, Corollary \ref{MonMapCor} shows that $(\varphi_i^*)_{i\in I}$ are $\calB_\infty^\vee$-algebraically independent within  $\tup{T(V)^\vee,\shuffle,\epsilon}$ if and only if the corresponding monomial map is injective, and \eqref{shuffle_multiindex2} shows that it is so if and only if the family $(\varphi_i)_{i\in I}$ of 
linear forms is $\Z$-linearly independent in $V^\vee$.
\end{enumerate}

\end{example}
\subsection{\label{subsect.dual-tensor}Appendix: Remarks on the dual of a tensor product}

In Subsection~\ref{subsect.chars-grouplikes}, we have mentioned the difficulties of defining the dual coalgebra of a $\kk$-algebra in the general case when $\kk$ is not necessarily a field.
As we said, these difficulties stem from the fact that the canonical map $M^\vee \otimes N^\vee \to \left(M \otimes N\right)^\vee$ (for two $\kk$-modules $M$ and $N$) is not generally an isomorphism, and may fail to be injective even if $\kk$ is an integral domain.
Even worse, the $\kk$-module $M^\vee \otimes N^\vee$ may fail to be torsionfree (which automatically precludes any $\kk$-linear map $M^\vee \otimes N^\vee \to \left(M \otimes N\right)^\vee$, not just the canonical one, from being injective).
We shall soon give an example where this happens (Example~\ref{exa.dual-tensor.not-tf}).
First, let us prove a positive result:

\begin{proposition}\label{DualInj}
Let $\kk$ be an integral domain.
Let $M$ and $N$ be two $\kk$-modules.
Consider the canonical $\kk$-linear map
\begin{equation}\label{DualInjArr}
\Phi:M^\vee\otimes_{\kk} N^\vee\longrightarrow (M\otimes_{\kk} N)^\vee 
\end{equation}
defined by 
\[
\left(\Phi(f\otimes g)\right)
\left(u\otimes v\right) = f(u) g(v)
\qquad \text{for all $f \in M^\vee$, $g \in N^\vee$, $u \in M$ and $v \in N$}.
\]
Then, the following are equivalent:
\begin{enumerate}
\item The $\kk$-module $M^\vee\otimes_{\kk} N^\vee$ is torsionfree.
\item The map $\Phi$ is injective.
\end{enumerate} 
\end{proposition}

\begin{proof}[First proof of Proposition~\ref{DualInj}.]
1. $\Longrightarrow$ 2.)
Assume that statement 1. holds. Thus,
the $\kk$-module $M^\vee\otimes_{\kk} N^\vee$ is torsionfree.

We must prove that the map $\Phi$ is injective.
Assume the contrary. Thus, $\ker\Phi \neq 0$.
Therefore, there exists some nonzero
$t \in M^\vee \otimes N^\vee$ such that $\Phi(t)=0$.
Consider this $t$.
Consider all choices of nonzero $m \in \kk$ and
of elements $u_1, u_2, \ldots, u_r \in M^\vee$ and
$v_1, v_2, \ldots, v_r \in N^\vee$ such that
\begin{equation}
mt = \sum_{i=1}^r u_i \otimes v_i .
\label{pf.DualInj.mt=}
\end{equation}
Among all such choices, choose one for which $r$ is minimum.

% We have $t \neq 0$ and thus $mt \neq 0$ (since
% $M^\vee\otimes_{\kk} N^\vee$ is torsionfree).
% Hence, \eqref{pf.DualInj.mt=} shows that $r > 0$.
From the fact that $r$ is minimum, we conclude that the
elements $v_1, v_2, \ldots, v_r$ of $N^\vee$
are $\kk$-linearly independent\footnote{%
Here is the \textit{proof} in detail:
Assume the contrary. Thus,
$\lambda_1 v_1 + \lambda_2 v_2 + \cdots + \lambda_r v_r = 0$
for some scalars
$\lambda_1, \lambda_2, \ldots, \lambda_r \in \kk$,
not all of which are zero.
Consider these scalars, and assume WLOG that
$\lambda_r \neq 0$.
Hence, $\lambda_r m \neq 0$ (since $\kk$ is an
integral domain) and
$\lambda_r v_r = -\sum_{i=1}^{r-1} \lambda_i v_i$
(since
$\lambda_1 v_1 + \lambda_2 v_2 + \cdots + \lambda_r v_r = 0$).
Now, multiplying the equality \eqref{pf.DualInj.mt=}
with $\lambda_r$, we obtain
\begin{align*}
\lambda_r m t
&= \lambda_r \sum_{i=1}^r u_i \otimes v_i
= \sum_{i=1}^r u_i \otimes \lambda_r v_i
= \sum_{i=1}^{r-1} \underbrace{u_i \otimes \lambda_r v_i}_{= \lambda_r u_i \otimes v_i}
+ u_r \otimes \underbrace{\lambda_r v_r}_{= -\sum_{i=1}^{r-1} \lambda_i v_i} \\
&= \sum_{i=1}^{r-1} \lambda_r u_i \otimes v_i
+ u_r \otimes \left(-\sum_{i=1}^{r-1} \lambda_i v_i\right)
= \sum_{i=1}^{r-1} \lambda_r u_i \otimes v_i
- \sum_{i=1}^{r-1} \lambda_i u_r \otimes v_i \\
&= \sum_{i=1}^{r-1} \left(\lambda_r u_i - \lambda_i u_r\right) \otimes v_i .
\end{align*}
This is an equality of the same shape as
\eqref{pf.DualInj.mt=}, but with $r-1$ instead of
$r$ (since $\lambda_r m \neq 0$).
Hence, it contradicts the minimality of $r$.
This contradiction completes our proof.}.
But the map $\Phi$ is $\kk$-linear; hence,
$\Phi(mt) = m \Phi(t) = 0$ (since $\Phi(t) = 0$).
Thus,
$0 = \Phi(mt) = \sum_{i=1}^r \Phi(u_i \otimes v_i)$
(by \eqref{pf.DualInj.mt=}).
Hence, for all $x \in M$ and $y \in N$, we have 
\begin{equation}
0 = \left(\sum_{i=1}^r \Phi(u_i \otimes v_i)\right)(x\otimes y)=\sum_{i=1}^r u_i(x) v_i(y)
\end{equation}
(by the definition of $\Phi$).
This can be rewritten as
\begin{equation}
0=\sum_{i=1}^r u_i(x) v_i \text{ in $N^\vee$} \qquad \text{for all $x \in M$}.
\end{equation}
But this entails that $u_i(x) = 0$ for all $1\leq i\leq r$ (since
$v_1, v_2, \ldots, v_r$
are $\kk$-linearly independent).
Since this holds for all $x\in M$, we thus find that
$u_i = 0$ for all $1\leq i\leq r$.
Hence, \eqref{pf.DualInj.mt=} shows that
$mt = 0$, so that $t = 0$ (since
$M^\vee \otimes N^\vee$ is torsionfree).
This contradicts $t \neq 0$. This contradiction completes
the proof of 1. $\Longrightarrow$ 2.

2. $\Longrightarrow$ 1.)
The $\kk$-module $\left(M \otimes_{\kk} N\right)^\vee$
is torsionfree,
since the dual of any $\kk$-module is torsionfree.
Hence,
if $\Phi$ is injective, then $M^\vee\otimes_{\kk} N^\vee$ is torsionfree
(since a submodule of a torsionfree $\kk$-module is
always torsionfree).
\end{proof}

\begin{proof}[Second proof of Proposition~\ref{DualInj} (sketched).]
1. $\Longrightarrow$ 2.)
Assume that statement 1. holds. Thus,
the $\kk$-module $M^\vee\otimes_{\kk} N^\vee$ is torsionfree.

We must prove that the map $\Phi$ is injective.

Let $\FF$ be the fraction field of $\kk$. For any $\kk$-module $P$,
we let $P_{\FF}$ denote the $\FF$-vector space $\FF \otimes P$
(which can also be defined as the localization of $P$ with
respect to the multiplicatively closed subset $\kk \setminus \{0\}$
of $\kk$),
and we let $\iota_P$ denote the canonical $\kk$-linear map
$P \to P_{\FF}$. We note that $\iota_P$ is injective when $P$
is torsionfree.

If $V$ is an $\FF$-vector space, then we shall write $V^*$ for
the dual space $\Hom_{\FF}\left(V, \FF\right)$ of $V$.
This should be distinguished from the $\kk$-module dual
of $V$, which we denote by $V^\vee$.

For any $\kk$-module $P$, the canonical $\FF$-linear
map $\rho_P : (P^\vee)_{\FF} \to (P_{\FF})^*$
(which sends each $f \in P^\vee$ to the unique
$\FF$-linear map $P_{\FF} \to \FF$ that extends $f$)
is injective.
(This is easily checked by hand, using the fact that
the map $\kk \to \FF$ is injective.)
% finitely many fractions in $F$ always have a common denominator.)
Hence, we obtain two injective $\FF$-linear maps
$\rho_M : (M^\vee)_{\FF} \to (M_{\FF})^*$ and
$\rho_N : (N^\vee)_{\FF} \to (N_{\FF})^*$.
Their tensor product is an injective $\FF$-linear map
$\rho_M \otimes_{\FF} \rho_N
: (M^\vee)_{\FF} \otimes_{\FF} (N^\vee)_{\FF} \to (M_{\FF})^* \otimes_{\FF} (N_{\FF})^*$
(since the
tensor product of two injective $\FF$-linear maps over
$\FF$ is injective, because $\FF$ is a field).

For any $\kk$-modules $M$ and $N$,
there is a canonical isomorphism
$\nu_{M, N} : \left(M \otimes N\right)_{\FF} \to M_{\FF}
\otimes_{\FF} N_{\FF}$.
(This is a general property of base change.)
Hence, $M_{\FF} \otimes_{\FF} N_{\FF} \cong
\left(M \otimes N\right)_{\FF}$, so that
$\left(M_{\FF} \otimes_{\FF} N_{\FF}\right)^*
\cong \left(\left(M \otimes N\right)_{\FF}\right)^*$.

Also, it is known from linear algebra
that the canonical map
$\kappa_{V, W} : V^* \otimes_{\FF} W^* \to \left(V \otimes_{\FF} W\right)^*$
is injective whenever $V$ and $W$ are two $\FF$-vector spaces.

We assumed that $M^\vee \otimes N^\vee$ is torsionfree.
Thus, the map $\iota_{M^\vee \otimes N^\vee}$ is injective.
We have the following commutative diagram of $\kk$-modules:
\begin{align}
\xymatrix{
& M^\vee \otimes N^\vee \ar[ddr]^\Phi \arinjrev[dl]_{\iota_{M^\vee \otimes N^\vee}} \\
\left(M^\vee \otimes N^\vee\right)_{\FF} \ar[d]_{\nu_{M^\vee, N^\vee}}^\cong \\
\left(M^\vee\right)_{\FF} \otimes_{\FF} \left(N^\vee\right)_{\FF} \arinj[d]_{\rho_M \otimes_{\FF} \rho_N} & & \left(M\otimes N\right)^\vee \ar[d]_{\iota_{\left(M \otimes N\right)^\vee}} \\
\left(M_{\FF}\right)^* \otimes_{\FF} \left(N_{\FF}\right)^* \arinj[d]_{\kappa_{M_{\FF}, N_{\FF}}} & & \left(\left(M \otimes N\right)^\vee\right)_{\FF} \ar[ddl]_{\rho_{M \otimes N}} \\
\left(M_{\FF} \otimes_{\FF} N_{\FF}\right)^* \ar[dr]^\cong \\
& \left(\left(M \otimes N\right)_{\FF}\right)^*
}
\end{align}

All maps along the left boundary of this diagram
are injective; thus, their composition is injective as well.
But this composition equals the composition of the maps
along the right boundary of the diagram.
Hence, the latter composition is injective.
Thus, $\Phi$ (being the initial map in this composition) must
be injective.
This proves 1. $\Longrightarrow$ 2.

2. $\Longrightarrow$ 1.)
As in the First proof above.
\end{proof}

We can now give an example (communicated to us by Jeremy Rickard) where $M^\vee \otimes N^\vee$ fails to be torsionfree:

\begin{example} \label{exa.dual-tensor.not-tf}
Let $\kk$ be a field, $I$ and $J$ infinite sets,
and $A$ the $\kk$-subalgebra of
\[
\kk(t)[x_i,y_j: i\in I,j\in J]
\]
generated by 
\[
\{x_i,y_j,tx_i,t^{-1}y_j: i\in I \text{ and } j\in J\}.
\]
Then, $A$ is an integral domain.
In this example, all duals are taken with respect to $A$ (not with respect to $\kk$); that is, $M^\vee$ means $\Hom_A\left(M, A\right)$.

The free $A$-modules $A^{(I)}$, $A^{(J)}$ and $A^{(I \times J)}$ satisfy
$\left(A^{(I)}\right)^\vee \cong A^I$,
$\left(A^{(J)}\right)^\vee \cong A^J$,
$\left(A^{(I \times J)}\right)^\vee \cong A^{I \times J}$, and
$A^{(I \times J)} \cong A^{(I)} \otimes_A A^{(J)}$.
Moreover, there is a natural $A$-linear map
\begin{align*}
\Psi : A^I \otimes_A A^J &\to A^{I\times J}, \\
\left(a_i\right)_{i \in I} \otimes_A \left(b_j\right)_{j \in J} &\mapsto
\left(a_i b_j\right)_{\left(i, j\right) \in I \times J}
\end{align*}
% $\Psi : A^I \otimes_A A^J \to A^{I\times J}$
% (sending each $\left(a_i\right)_{i \in I} \otimes_A \left(b_j\right)_{j \in J}
% \in A^I \otimes_A A^J$ to
% $\left(a_i b_j\right)_{\left(i, j\right) \in I \times J}
% \in A^{I \times J}$)
that forms a commutative diagram
\begin{equation}
\xymatrix@C=5pc{
A^I \otimes_A A^J \ar[r]^\Psi \ar[d]_\cong & A^{I \times J} \ar[d]_\cong \\
\left(A^{(I)}\right)^\vee \otimes_A \left(A^{(J)}\right)^\vee \ar[r]_-\Phi & \left(A^{(I\times J)}\right)^\vee
}
\label{eq.exa.dual-tensor.not-tf.cd}
\end{equation}
with the canonical map $\Phi : \left(A^{(I)}\right)^\vee \otimes_A \left(A^{(J)}\right)^\vee
\to \left(A^{(I\times J)}\right)^\vee$ defined as
in Proposition~\ref{DualInj}.

Now, we define an element
\[
\xi
= (tx_i)_{i\in I} \otimes_A (t^{-1}y_j)_{j\in J} - (x_i)_{i\in I} \otimes_A (y_j)_{j\in J}
\in A^I \otimes_A A^J .
\]
It is clear by computation that $\Psi\tup{\xi} = 0$.
%whence $\xi \in \ker\Psi$.
We shall now show that $\xi \neq 0$.

Indeed, assume the contrary. Thus, $\xi = 0$.
Hence, $\xi$ can be ``shown to be zero using only finitely many elements of $A$''
-- i.e., there exists a finitely generated $\kk$-subalgebra $B$ of $A$ such that
\begin{equation}
(tx_i)_{i\in I}\otimes_B (t^{-1}y_j)_{j\in J}
- (x_i)_{i\in I}\otimes_B (y_j)_{j\in J}
=0
\qquad \text{ in } A^I \otimes_B A^J .
\label{eq.exa.dual-tensor.not-tf.overB}
\end{equation}
(Indeed, $A^I \otimes_A A^J$ can be viewed as a quotient of $A^I \otimes_{\mathbb{Z}} A^J$ modulo the $\mathbb{Z}$-submodule spanned by all the differences $ua \otimes_{\mathbb{Z}} v - u \otimes_{\mathbb{Z}} av$ for $u \in A^I$, $v \in A^J$ and $a \in A$. Thus, $\xi = 0$ means that the element
$(tx_i)_{i\in I}\otimes_{\mathbb{Z}} (t^{-1}y_j)_{j\in J}
- (x_i)_{i\in I}\otimes_{\mathbb{Z}} (y_j)_{j\in J}$
of $A^I \otimes_{\mathbb{Z}} A^J$ is a $\mathbb{Z}$-linear combination of finitely many such differences.
Now take the (finitely many) $a$'s involved in these differences, and define $B$ to be the $\kk$-subalgebra of $A$ generated by these finitely many $a$'s.
Then, \eqref{eq.exa.dual-tensor.not-tf.overB} holds, as desired.)

Consider this $B$. Choose $r\in I$ and $s\in J$ so that
\[
B\subseteq \kk(t)[x_i,y_j:i\neq r \text{ and } j\neq s].
\]
(Such $r$ and $s$ exist, since $B$ is finitely generated.)

If $m$ is a monomial in the variables $\{x_i,y_j:i\in I,j\in J\}$, then $\deg_x m$ shall denote the total degree of $m$ in the variables $x_i$, while $\deg_y$ shall denote the total degree of $m$ in the variables $y_j$.
A \emph{good monomial} shall mean a product of the form $t^lm$, where $m$ is a monomial in the variables $\{x_i,y_j:i\in I,j\in J\}$ and $l \in \mathbb{Z}$ satisfies $-\deg_y m \leq l \leq \deg_x m$.
As a $\kk$-vector space, $A$ has a basis consisting of all good monomials.

A good monomial $t^lm$ will be called \emph{strict} if $m$ is just a power of $x_r$ or just a power of $y_s$. (In particular, the monomial $1$ is strict.)
The non-strict good monomials span a proper ideal of $A$.
Let $\bar{A}$ be the corresponding quotient algebra of $A$; then the image $\bar{B}$ of $B$ in $\bar{A}$ is just (a copy of) $\kk$ (since every element of $B$ is a $\kk$-linear combination of non-strict good monomials and of the monomial $1$).
For any $f \in A$, we let $\overline{f}$ denote the canonical projection of $f$ on the quotient ring $\bar{A}$.

Now, we have a $\kk$-linear map obtained by composing
\[
A^I \otimes_B A^J
\to \bar{A}^I\otimes_{\bar{B}}\bar{A}^J
\overset{\cong}{\to} \bar{A}^I\otimes_{\kk}\bar{A}^J
\to \bar{A} \otimes_{\kk} \bar{A} ,
\]
where
\begin{itemize}
\item the first arrow sends each tensor $u \otimes_B v \in A^I \otimes_B A^J$ to the tensor $\overline{u} \otimes_{\bar{B}} \overline{v} \in \bar{A}^I\otimes_{\bar{B}}\bar{A}^J$ (with $\overline{u}$ denoting the projection of $u \in A^I$ to $\bar{A}^I$, and with $\overline{v}$ defined similarly);
\item the second arrow is due to $\bar{B}$ being (a copy of) $\kk$;
\item the third arrow is obtained by tensoring
\begin{align*}
& \text{ the canonical projection } \bar{A}^I \to \bar{A},\ \tup{a_i}_{i\in I} \mapsto a_r \\
\text{with } & \text{ the canonical projection } \bar{A}^J \to \bar{A},\ \tup{b_j}_{j\in J} \mapsto b_s .
\end{align*}
\end{itemize}
Applying this $\kk$-linear map
to both sides of the equation \eqref{eq.exa.dual-tensor.not-tf.overB}, we obtain
\[
\overline{tx_r} \otimes_{\kk} \overline{t^{-1}y_s} - \overline{x_r} \otimes_{\kk} \overline{y_s}
=0 \qquad \text{ in } \bar{A} \otimes_{\kk} \bar{A} .
\]
But this contradicts the fact that the four basis elements
$\overline{tx_r}, \overline{t^{-1}y_s}, \overline{x_r}, \overline{y_s}$
of $\bar{A}$ are $\kk$-linearly independent.

Hence, our assumption ($\xi=0$) was false.
Thus, $\xi \neq 0$. In view of $\Psi\tup{\xi} = 0$,
this shows that $\Psi$ is not injective.
Due to the commutative diagram \eqref{eq.exa.dual-tensor.not-tf.cd},
this means that $\Phi$ is not injective.
Hence, Proposition~\ref{DualInj} shows that the $A$-module
$\left(A^{(I)}\right)^\vee \otimes_A \left(A^{(J)}\right)^\vee$ is
not torsionfree.
In view of \eqref{eq.exa.dual-tensor.not-tf.cd}, this means in turn
that the $A$-module $A^I \otimes_A A^J$ is not torsionfree
(despite being the tensor product of the two torsionfree, and even
torsionless, $A$-modules $A^I$ and $A^J$).

The presence of torsion in $A^I \otimes_A A^J$ can also be seen directly
using the element $\xi$ from the above argument:
For any $s\in I$, we have
\[
x_s \xi
= x_s\left((tx_i)_{i\in I} \otimes_A (t^{-1}y_j)_{j\in J}-(x_i)_{i\in I} \otimes_A (y_j)_{j\in J}\right)=0,
\]
since
\begin{align*}
x_s\left((tx_i)_{i\in I} \otimes_A (t^{-1}y_j)_{j\in J}\right)
&=(tx_sx_i)_{i\in I} \otimes_A (t^{-1}y_j)_{j\in J}\\
&=(x_i)_{i\in I} \otimes_A (x_sy_j)_{j\in J}\\
&=x_s\left((x_i)_{i\in I} \otimes_A (y_j)_{j\in J}\right).
\end{align*}
\end{example}

Another example of a non-injective canonical map
$A^I \otimes_A A^J \to A^{I \times J}$
(and thus, of a tensor product of the form
$M^\vee \otimes_A N^\vee$ having torsion)
can be found in the last two paragraphs of
\cite{Goodea72}.

Note that such examples can only exist when $A$
is not Noetherian.
Indeed, it has been shown in \cite[Proposition 1.2]{AbGoWi99}
that if $A$ is a Noetherian ring, then the
canonical map
$A^I \otimes_A A^J \to A^{I \times J}$ for any
two sets $I$ and $J$ is injective.

\end{document}